\documentclass[11pt]{amsart}
\usepackage[utf8x]{inputenc}
\usepackage[english]{babel}
\usepackage[T1]{fontenc}
 
\usepackage{graphicx}
\usepackage{caption}
\usepackage{subcaption}
\usepackage{amssymb,amsmath}
\usepackage[hidelinks]{hyperref}
\usepackage{indentfirst}
\usepackage{enumerate,amsmath,amssymb, mathrsfs,mathtools}
\usepackage{appendix}
\usepackage{latexsym}
\usepackage{url}
\usepackage{color}
\usepackage{accents}
\usepackage{setspace}
\usepackage{pdfpages}

\usepackage[margin=2.5cm]{geometry}
\allowdisplaybreaks

\newcommand{\hcirc}{\accentset{\circ}{h}}
\newcommand{\hbarcirc}{\accentset{\circ}{\bar{h}}}
\newcommand{\sigmacirc}{\accentset{\circ}{\sigma}}
\newcommand{\ubar}[1]{\underaccent{\bar}{#1}}

\mathcode`l="8000
\begingroup
\makeatletter
\lccode`\~=`\l
\DeclareMathSymbol{\lsb@l}{\mathalpha}{letters}{`l}
\lowercase{\gdef~{\ifnum\the\mathgroup=\m@ne \ell \else \lsb@l \fi}}%
\endgroup

\def\XXint#1#2#3{{\setbox0=\hbox{$#1{#2#3}{\int}$ }
		\vcenter{\hbox{$#2#3$ }}\kern-.6\wd0}}

\newtheorem{prop}{Proposition}
\newtheorem{thm}[prop]{Theorem}
\newtheorem{lem}[prop]{Lemma}
\newtheorem{coro}[prop]{Corollary}
\newtheorem{rema}[prop]{Remark}

\title[Large area-constrained Willmore surfaces]{Large area-constrained Willmore surfaces in asymptotically Schwarzschild $3$-manifolds}
\author{Michael Eichmair}
\address{ University of Vienna,
	Oskar-Morgenstern-Platz 1,
	1090 Vienna,
	Austria}
\email{michael.eichmair@univie.ac.at}
\author{Thomas Koerber}
\address{ University of Vienna,
	Oskar-Morgenstern-Platz 1,
	1090 Vienna,
	Austria}
\email{thomas.koerber@univie.ac.at}

\begin{document}

\date{\today}
\onehalfspacing

\begin{abstract}
		We apply the method of Lyapunov-Schmidt reduction to study large area-constrained Willmore surfaces in Riemannian $3$-manifolds asymptotic to Schwarzschild. In particular, we prove that the end of such a manifold is foliated by distinguished area-constrained Willmore spheres.  The leaves are the unique area-constrained Willmore spheres with large area, non-negative Hawking mass, and distance to the center of the manifold at least a small multiple of the area radius. Unlike previous related work, we only require that the scalar curvature satisfies mild asymptotic conditions. We also give explicit examples to show that these conditions on the scalar curvature are necessary.
\end{abstract}

\maketitle

\section{Introduction}

Let $(M, g)$ be an asymptotically flat Riemannian $3$-manifold with non-negative scalar curvature. Such manifolds arise as maximal initial data sets for the Einstein field equations and thus play an important role in general relativity. 
\\ \indent Let $\Sigma \subset M$ be a sphere with unit normal $\nu$, mean curvature vector $- H  \,\nu$, area measure $\mathrm{d}\mu$, and area $|\Sigma|$. The Hawking mass
\[
m_H(\Sigma)=\sqrt{\frac{|\Sigma|}{16\,\pi}}\bigg(1-\frac{1}{16\,\pi}\int_{\Sigma} H^2\,\mathrm{d}\mu\bigg)
\]
of $\Sigma$ has been used to probe the gravitational field in the domain bounded by $\Sigma$; see e.g.~\cite{hawking1968gravitational, christodoulou71some}. R.~Geroch \cite[p.~115]{Geroch:1973} has noted that the Hawking mass does not increase if $\Sigma$ flows in direction of the unit normal $\nu$ at a speed equal to $H^{-1}$, provided $H >0$. Moreover, he has proposed a proof of the positive energy theorem based on evolving by {inverse mean curvature flow} a small geodesic sphere in $(M,g)$ with Hawking mass close to zero  into a large, centered sphere in the asymptotically flat end whose Hawking mass is close to the ADM-mass of $(M, g)$. Expanding upon Geroch's idea, P.~S.~Jang and R.~Wald \cite[p.~43]{JangWald1977} have sketched a proof of the Riemannian Penrose inequality in the special case where the apparent horizon is connected. These programs have been completed in the paper \cite{huisken2001inverse} by G.~Huisken and T.~Ilmanen, where a suitable, necessarily non-smooth notion of inverse mean curvature flow is developed.  H.~Bray has proven the Riemannian Penrose inequality with no restriction on the number of boundary components in \cite{Bray:2001} using a different method. 
\\
\indent D.~Christodoulou and S.-T.~Yau \cite{christodoulou71some} have noted that the Hawking mass of stable constant mean curvature spheres is non-negative. Note that $m_H (\Sigma) \leq 0$ in flat $\mathbb{R}^3$ with equality if and only if $\Sigma$ is a {round} sphere. The apparent tension between these results is indicative of the potential role of the Hawking mass as a measure of the gravitational field. In this relation, note that stable constant mean curvature surfaces abound in every initial data set. Indeed, as discussed in Appendix K of \cite{mineffectivePMT}, there exist isoperimetric regions of every volume. 
\\ 
\indent To describe our contributions here, we say that $(M,g)$ is $C^k$-asymptotic to Schwarzschild with mass $m>0$ if there is a non-empty compact set whose complement in $M$ is diffeomorphic to $\{x\in\mathbb{R}^3:|x|>1/2\}$ and such that, in this so-called asymptotically flat, there holds
\[
g=\bigg(1+\frac{m}{2\,|x|}\bigg)^4\bar g+\sigma.
\]
  Here, $x$ is the Euclidean position vector and $\bar g$ is the Euclidean metric on $\mathbb{R}^3$, while $\sigma$ is a symmetric two-tensor that satisfies, as $x\to\infty$ for every multi-index $J$ with $|J|\leq k$, 
\[
\partial_J \sigma=O(|x|^{-2-|J|}).
\]  Note that $(M,g)$ is modeled upon the initial data of a Schwarzschild black hole given by
\begin{align} \bigg(\left\{x\in\mathbb{R}^3:|x|\geq\frac{m}{2}\right\},\,\left( 1+\frac{m}{2\,|x|}\right)^4\bar g \bigg). \label{schwarzschild initial data}
\end{align}
 \indent Given $r>1/2$, we define $B_r\subset M$ to be the compact domain whose boundary corresponds to $S_r(0)$ in the asymptotically flat chart.  We say that a surface $\Sigma\subset M$ is on-center if it bounds a compact region that contains $B_1$. If $\Sigma$ bounds a compact region disjoint from $B_1$, it will be called outlying.
\\ \indent In pioneering work \cite{huisken1996definition}, G.~Huisken and S.-T.~Yau have shown that an end that is $C^4$-asymptotic to Schwarzschild with positive mass is foliated by stable constant mean curvature spheres. This foliation detects fundamental physical quantities associated with the initial data set such as the ADM mass and the Hamiltonian center of mass. Moreover, they have shown that the leaves of the foliation are the only stable constant mean curvature spheres of their respective mean curvature within large classes of competing surfaces. The original characterization of the leaves in \cite{huisken1996definition} has been sharpened by J.~Qing and G.~Tian in \cite{qing2007uniqueness}, by S.~Brendle and the first-named author in \cite{brendle2014large}, and by A.~Carlotto, O.~Chodosh, and the first-named author in \cite{mineffectivePMT}. The optimal uniqueness result for large stable constant mean curvature spheres in asymptotically Schwarzschild initial data sets has recently been obtained by O.~Chodosh and the first-named author \cite{chodosh2017global,chodosh2019far}.
\\ \indent The characterization of the leaves of the  foliation as the unique solutions of the isoperimetric problem for large volumes has been established by H.~Bray in \cite{bray1997penrose} for exact Schwarzschild (\ref{schwarzschild initial data}) and by J.~Metzger and the first-named author in \cite{isostructure, eichmair2013unique} for initial data asymptotic to Schwarzschild. In fact, these optimal global uniqueness results for large isoperimetric surfaces hold for asymptotically flat manifolds with positive mass, in particular for the examples constructed by A.~Carlotto and R.~Schoen in \cite{carlotto2016localizing}, as has recently been shown by O.~Chodosh, Y.~Shi, H.~Yu, and the first-named author in \cite{CESY} and by H.~Yu in \cite{Yu:2020}. \\ \indent A different approach to obtain surfaces that are well-adapted to the ambient geometry is to maximize the  Hawking mass  under a suitable geometric constraint. Here, fixing the area is a natural choice. Area-constrained critical points of the Hawking mass are also area-constrained critical points of the Willmore energy 
\begin{align*} 
\mathcal{W}(\Sigma)=\frac14 \int_{\Sigma} H^2\,\mathrm{d}\mu.
\end{align*}
We refer to such surfaces as area-constrained Willmore surfaces. Note that in e.g.~\cite{lamm2011foliations}, such surfaces are said to be of Willmore type.\\
\indent Critical points of the Willmore energy, known as Willmore surfaces, satisfy the Euler-Lagrange equation $-W=0$ where 
\begin{align}
W=\Delta H+(|\hcirc|^2+\operatorname{Ric}(\nu,\nu))\,H.
\label{Willmore quantity}
\end{align}
Here, $\Delta$ is the non-positive Laplace-Beltrami operator, $\hcirc$ the traceless part of the second fundamental form $h$,  and $\operatorname{Ric}$ the Ricci curvature of $(M,g)$. Likewise, area-constrained Willmore surfaces satisfy the area-constrained Willmore equation
\begin{align}
-W=\kappa\, H,
\label{constrained Willmore equation}
\end{align}
where $\kappa\in\mathbb{R}$ is a Lagrange multiplier. Note that $\kappa$ is denoted by $\lambda$ in \cite{lamm2011foliations}. The linearization of the Willmore operator is denoted by $Q$. It measures how $-W$ changes along a normal variation of the surface $\Sigma$. We refer to Appendix \ref{Willmore appendix} for more details, including a discussion of the notion of stability of such surfaces. 
\\ \indent The cross-sections of rotationally symmetric Riemannian manifolds are easily seen to form a foliation by area-constrained Willmore spheres. This observation applies in particular to the spheres of symmetry in the spatial Schwarzschild manifold (\ref{schwarzschild initial data}). In \cite{lamm2011foliations}, T.~Lamm, J.~Metzger, and F.~Schulze have applied a delicate singular perturbation analysis  to prove the existence of such a foliation also in the case of small perturbations of the Schwarzschild manifold.  To state their result, we define the area radius $\lambda(\Sigma)>0$ of a surface $\Sigma\subset M\setminus K$ by $$4\,\pi\, \lambda(\Sigma)^2=|\Sigma|$$ and its inner radius $\rho(\Sigma)$ by
$$
\rho(\Sigma)=\sup\{r>1/2:B_r\cap \Sigma=\emptyset\}.$$ 
Moreover, we use $R$ to denote the scalar curvature of $(M,g)$. Below, we summarize Theorem 1 and Theorem 2 in \cite{lamm2011foliations}.
\begin{thm}[\cite{lamm2011foliations}] \label {LMS theorem} Given $m>0$, there is a constant $\eta > 0$ with the following property. Suppose that $(M,g)$ is $C^3$-asymptotic to Schwarzschild with mass $m>0$ such that
\begin{align}
\limsup_{|x|\to\infty}\bigg (|x|^2\,|\sigma|+|x|^3\,|D\sigma|+|x|^4\,|D^2\sigma|+|x|^5\,|D^3\sigma|\bigg)<\eta \label{smallness}
\end{align}
and
\[
\limsup_{|x|\to\infty}|x|^5\,|R| <\eta.
\]
There is a compact set $K\subset M$, a number $\kappa_0>0$, and spheres $\{\Sigma(\kappa): \kappa \in(0,\kappa_0)\}$ such that the following hold:
\begin{itemize}
	\item[$\circ$] $\Sigma(\kappa)$ is a stable area-constrained Willmore sphere that satisfies (\ref{constrained Willmore equation}) with parameter $\kappa$.
	\item[$\circ$] $M\setminus K$ is smoothly foliated by the family $\{\Sigma(\kappa): \kappa\in(0,\kappa_0)\}$.
\end{itemize}
Moreover, there is a constant $\epsilon_0>0$ such that every on-center, strictly mean convex area-constrained Willmore sphere $\Sigma\subset M\setminus K$ with \begin{align}\bigg|\frac{\lambda(\Sigma)}{\rho(\Sigma)}-1\bigg|<\epsilon_0 \quad \text{and} \quad  \int_{\Sigma}|\hcirc|^2\,\mathrm{d}\mu<\epsilon_0 
\label{uniqueness requirement}
\end{align} is part of this foliation. 
\end{thm}
\begin{rema} \normalfont
	In \cite{lamm2011foliations}, the uniqueness result is stated in terms of smallness conditions on the rescaled barycenter
	\begin{align*}
	\frac{1}{\lambda(\Sigma)\,|\Sigma|}\int_\Sigma x\,\mathrm{d}\mu 
	\end{align*} 
	and the quotient $\rho(\Sigma)^{-2}\,\lambda(\Sigma)$. These conditions are implied by (\ref{uniqueness requirement}) and Theorem 1.1 in \cite{deLellisMueller}.
\end{rema} 
The stability of the leaves $\Sigma(\kappa)$ suggests that each contains a maximal amount of Hawking mass given their surface area. Locally, this has been confirmed by the second-named author; see Theorem 1.2 in \cite{koerber2020area}.
\begin{thm}[{\cite{koerber2020area}}]
	Assumptions as in Theorem \ref{LMS theorem}. 	\label{centred uniqueness} 
Let $\Sigma\subset M\setminus K$ be a closed, on-center sphere with  $$\bigg|\frac{\lambda(\Sigma)}{\rho(\Sigma)}-1\bigg|<\epsilon_0$$ and $|\Sigma|=|\Sigma(\kappa)|$ for some $\kappa\in(0,\kappa_0)$. Then
$$
m_H(\Sigma)\leq m_H(\Sigma(\kappa))
$$
with equality if and only if $\Sigma=\Sigma(\kappa)$.  Moreover, every on-center area-constrained Willmore sphere $\Sigma\subset M\setminus K$ with $$\bigg|\frac{\lambda(\Sigma)}{\rho(\Sigma)}-1\bigg|<\epsilon_0 \quad \text{and} \quad  \int_{\Sigma}|\hcirc|^2\,\mathrm{d}\mu<\epsilon_0 $$ is part of the foliation from Theorem \ref{LMS theorem}. 
\end{thm} 
\begin{rema} \normalfont
	Unlike in Theorem \ref{LMS theorem}, there is no assumption on the sign of the mean curvature in the uniqueness statement in Theorem \ref{centred uniqueness}.
\end{rema}
Comparing with the results available for stable constant mean curvature surfaces, the assumptions of Theorem \ref{LMS theorem} and Theorem \ref{centred uniqueness} are quite restrictive. Yet, there has been no subsequent result  which either establishes  the existence of a foliation by area-constrained Willmore spheres in a more general setting or characterizes the leaves of such a foliation more globally.  Even in  exact Schwarzschild initial data (\ref{schwarzschild initial data}), our variational understanding of the Willmore energy is limited. In fact, it is not known if an area-constrained maximizer of the Hawking mass exists unless the prescribed area is either very small or an integer multiple of the area of the horizon; see \cite[Theorem 1.6 and Remark 1.7]{wei2020maximizers}. By contrast, S.~Brendle has shown in \cite{brendle2013constant} that the spheres of symmetry are the only closed, embedded constant mean curvature surfaces in exact Schwarzschild (\ref{schwarzschild initial data}). Previously, it had been known from the work of H.~Bray in \cite{bray1997penrose} that these spheres are the only solutions of the isoperimetric problem for the volume they enclose. Note that such a result fails for the Willmore energy, as one can construct surfaces of arbitrarily large area and Hawking mass by gluing small catenoidal necks between spheres of symmetry that are close to the horizon; see the remark below Corollary 5.4 in \cite{koerber2020area}. Consequently, any reasonable characterization of such surfaces can only possibly hold outside a compact set or under a small energy assumption. \\ \indent 
In his habilitation thesis \cite{laurain2019analyse}, P.~Laurain conjectures the existence of a foliation by area-constrained Willmore spheres if the metric $g$  satisfies the so-called Regge-Teitelboim condition (see \cite{Regge-Teitelboim}) and that all area-constrained Willmore surfaces with small energy that enclose a sufficiently large compact set are part of this foliation; cf.~\cite[Theorem 49 (In progress)]{laurain2019analyse}. Note that -- being non-linear and of fourth order -- the area-constrained Willmore equation (\ref{constrained Willmore equation}) poses hard analytical challenges and is not as accessible geometrically as the constant mean curvature equation. What is more, Willmore stability does not appear to be as useful of a condition as the stability of a constant mean curvature surface. For instance, every closed minimal surface is a stable Willmore surface.
\\ \indent  In this work, we establish the existence and uniqueness of foliations by area-constrained Willmore spheres in a generality  analogous to the optimal results for stable constant mean curvature surfaces in  \cite{chodosh2017global,chodosh2019far}. In summary, we discover optimal conditions on the scalar curvature under which the end of every asymptotically Schwarzschild manifold is foliated by large stable area-constrained Willmore spheres. These surfaces are unique among  all large area-constrained Willmore spheres with non-negative Hawking mass whose inner radius is at least a small multiple of the area radius. Our results differ from those in Theorem \ref{LMS theorem} in that we do not require smallness of the perturbation $\sigma$ off Schwarzschild or the centering quantity
\[
\frac{\lambda(\Sigma)}{\rho(\Sigma)}-1
\]
such as (\ref{smallness}) or (\ref{uniqueness requirement}), respectively.   \\ \indent 
 More precisely, we first establish the existence of a foliation by area-constrained Willmore spheres assuming that $(M,g)$ is $C^4$-asymptotic to Schwarzschild. We also assume that  the scalar curvature is asymptotically even and satisfies a certain growth condition.
\begin{thm}
Let $(M,g)$ be $C^4$-asymptotic to Schwarzschild with mass $m>0$ and suppose that the scalar curvature $R$ satisfies
\begin{align}
\sum_{i=1}^3	x^i\,\partial_i(|x|^2\,R)&\leq o(|x|^{-2})\text{ and} \label{decay R}\\
R(x)-R(-x)&=o(|x|^{-4}). \label{even R}
\end{align}
There exists a compact set $K\subset M$, a number $\kappa_0>0$, and on-center stable area-constrained Willmore spheres $\Sigma(\kappa)$, $\kappa \in(0,\kappa_0)$, satisfying (\ref{constrained Willmore equation}) with parameter $\kappa$ such that $M\setminus K$ is foliated by the family $\{\Sigma(\kappa): \kappa\in(0,\kappa_0)\}$. Moreover, there holds
\label{existence thm}
 \begin{align*}
 \lim_{\kappa\to0}\frac{\lambda({\Sigma({\kappa})})}{\rho({\Sigma({\kappa})})}=1. 
 \end{align*}
\end{thm}
\begin{rema} \normalfont Note that the assumptions of the theorem are satisfied if $(M,g)$ is $C^4$-asymptotic to Schwarzschild with  $R=o(|x|^{-4})$. The $C^4$-decay gives  $DR=o(|x|^{-5})$ in this case, which implies (\ref{decay R}). 
	\end{rema}
\begin{rema} \normalfont In Theorem \ref{existence thm} and Theorem \ref{uniqueness thm}, it would be sufficient to require appropriate $C^{3,\alpha}$-decay of the metric for some $\alpha\in(0,1)$. We use the slightly stronger assumption for the sake of readability.
\end{rema}
Next, we focus on the geometric characterization of the foliation $\{\Sigma(\kappa):\kappa\in(0,\kappa_0)\}$. Continuing to assume the same asymptotic conditions on the scalar curvature, we show that the leaves of the foliation are the unique  large area-constrained Willmore spheres whose inner radius and area radius are comparable and with traceless second fundamental form small in $L^2$. The conclusion of Theorem \ref{uniqueness thm} below is illustrated in Figure \ref{Figure alternatives}.
\begin{thm}
Assumptions as in Theorem \ref{existence thm}. There exist a small constant $\epsilon_0>0$ and a compact set $K\subset M$ which only depend on $(M,g)$ such that the following holds. For every $\delta>0$, there exists a large constant $\lambda_0>1$ such that every area-constrained Willmore sphere $\Sigma\subset M\setminus K$ with
\begin{align} \label{area and centering} |\Sigma|>4\,\pi\,\lambda_0^2,\qquad\qquad \delta\, \lambda(\Sigma)<\rho(\Sigma), \qquad \qquad \delta\, \rho(\Sigma)<\lambda(\Sigma),
\end{align} 
and 
\begin{align}\int_{\Sigma} |\hcirc|^2\,\mathrm{d}\mu<\epsilon_0
\label{small energy assumption}
\end{align}
belongs to the \label{uniqueness thm} foliation from Theorem \ref{existence thm}.
\end{thm}
\begin{rema}\normalfont
The small energy assumption (\ref{small energy assumption}) and the assumption that $\Sigma$ be spherical may be replaced by requiring a lower bound on the Hawking mass and an upper bound on the genus of $\Sigma$; see Proposition \ref{prop lower hawking mass bound}. 
\end{rema}
\begin{rema}\normalfont
	The proof of Theorem \ref{uniqueness thm} shows that there are no outlying area-constrained Willmore spheres satisfying  \eqref{area and centering} and \eqref{small energy assumption} if the scalar curvature of $(M,g)$ satisfies \eqref{decay R} but not necessarily \eqref{even R}.  The condition \eqref{decay R} has been discovered by O.~Chodosh and the first-named author in \cite[Theorem 1.4]{chodosh2019far} to be sufficient to rule out sequences of large outlying stable constant mean curvature spheres whose area radius and inner radius are comparable.
\end{rema}
Finally, we consider large area-constrained Willmore spheres that are far-outlying in the sense that their inner radius dominates their area radius. In this regime, the contribution of the Schwarzschild metric to the Hawking mass is so weak  that a stronger assumption on the scalar curvature is needed to preclude the existence of such surfaces. 
\begin{thm}
	Suppose that $(M,g)$ is $C^5$-asymptotic to Schwarzschild with mass $m>0$ and that its scalar curvature $R$ satisfies
	\begin{align}
\sum_{i=1}^3	x^i\,\partial_i( |x|^2\,R)\leq 0. \label{R growth without error}
	\end{align}
	There exist small constants $\epsilon_0,\,\delta_0>0$ and a large constant $\lambda_0>1$ which only depend on $(M,g)$ such that the following holds.
Every area-constrained Willmore sphere $\Sigma\subset M$ with
	\begin{align*} |\Sigma|>4\,\pi\,\lambda_0^2,\qquad\text{and} \qquad \int_{\Sigma} |\hcirc|^2\,\mathrm{d}\mu<\epsilon_0
	\end{align*} 
	satisfies  \label{far outlying thm}  
$$
\delta_0\,\rho(\Sigma)<\lambda(\Sigma). 
$$

\end{thm}
\begin{rema} \normalfont Condition (\ref{R growth without error}) is stronger than (\ref{decay R}) and, for instance, satisfied if the scalar curvature of $(M,g)$ vanishes. In any case, the assumptions of Theorem \ref{far outlying thm} are weaker than those discovered in \cite{chodosh2019far} to be sufficient to rule out far-outlying stable constant mean curvature spheres, where stronger decay of the metric is required and the scalar curvature is assumed to either vanish or to be radially convex. This improvement owes to a conservation law for the Einstein tensor known as the Pohozaev identity. In the generality required here, this law has been observed by R.~Schoen, see \cite[Proposition 1.4]{MR929283} and \cite{MR0192184}, and applied by T. Lamm, J. Metzger, and F. Schulze in \cite{lamm2011foliations} in a similar context. This identity precisely brings out the contribution of the scalar curvature to the Willmore energy as we explain in Lemma \ref{Willmore energy sphere}. It turns out that the assumptions on the scalar curvature required in Theorem \ref{far outlying thm} are sufficient to rule out large far-outlying stable constant mean curvature spheres as well. We include a proof of this fact in Appendix \ref{CMC appendix}.
\end{rema}

\begin{figure}
\centering
\begin{subfigure}{0.33\textwidth}
	
	\includegraphics[width=1\linewidth]{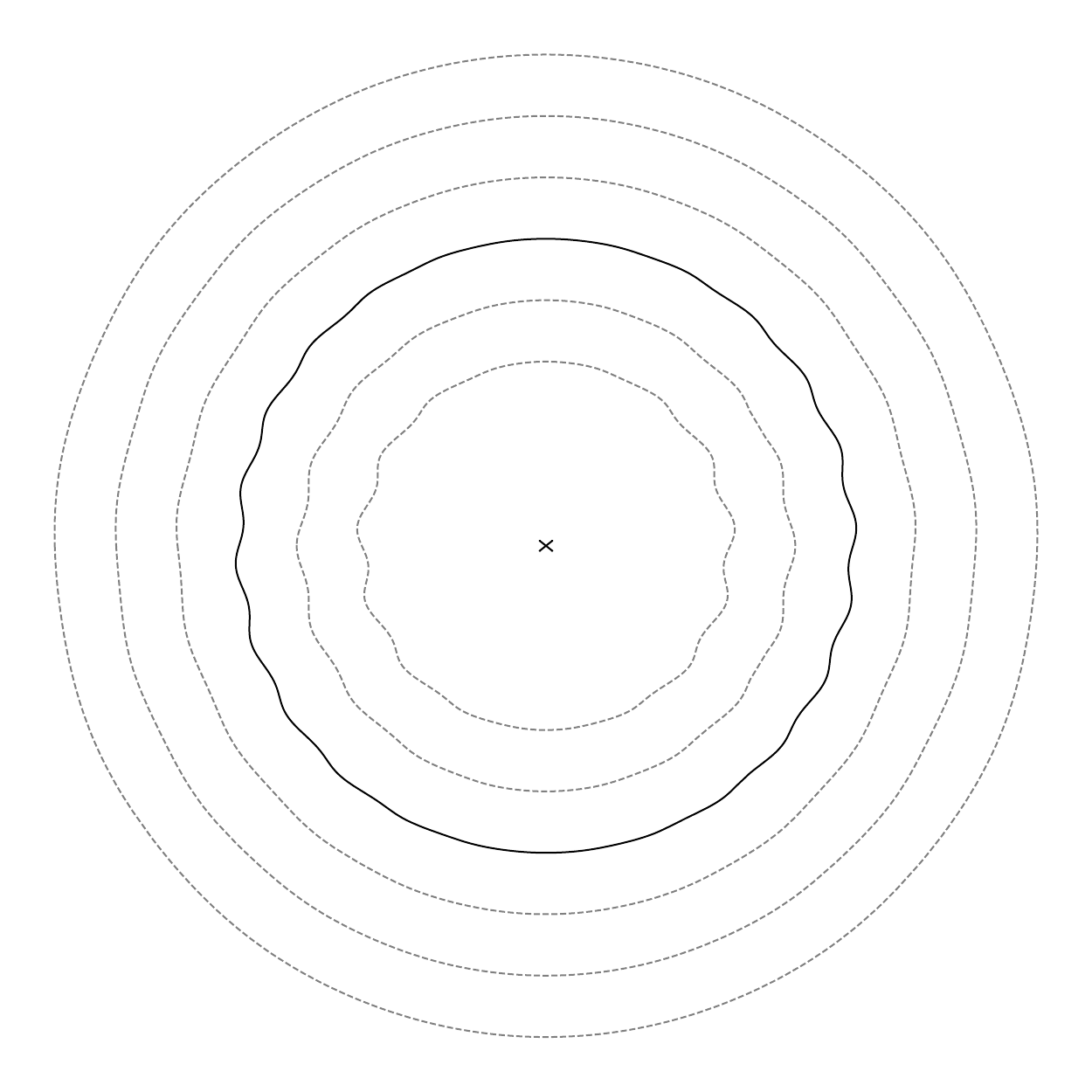}
	
\end{subfigure}%
\begin{subfigure}{0.33\textwidth}
	
	\includegraphics[width=1\linewidth]{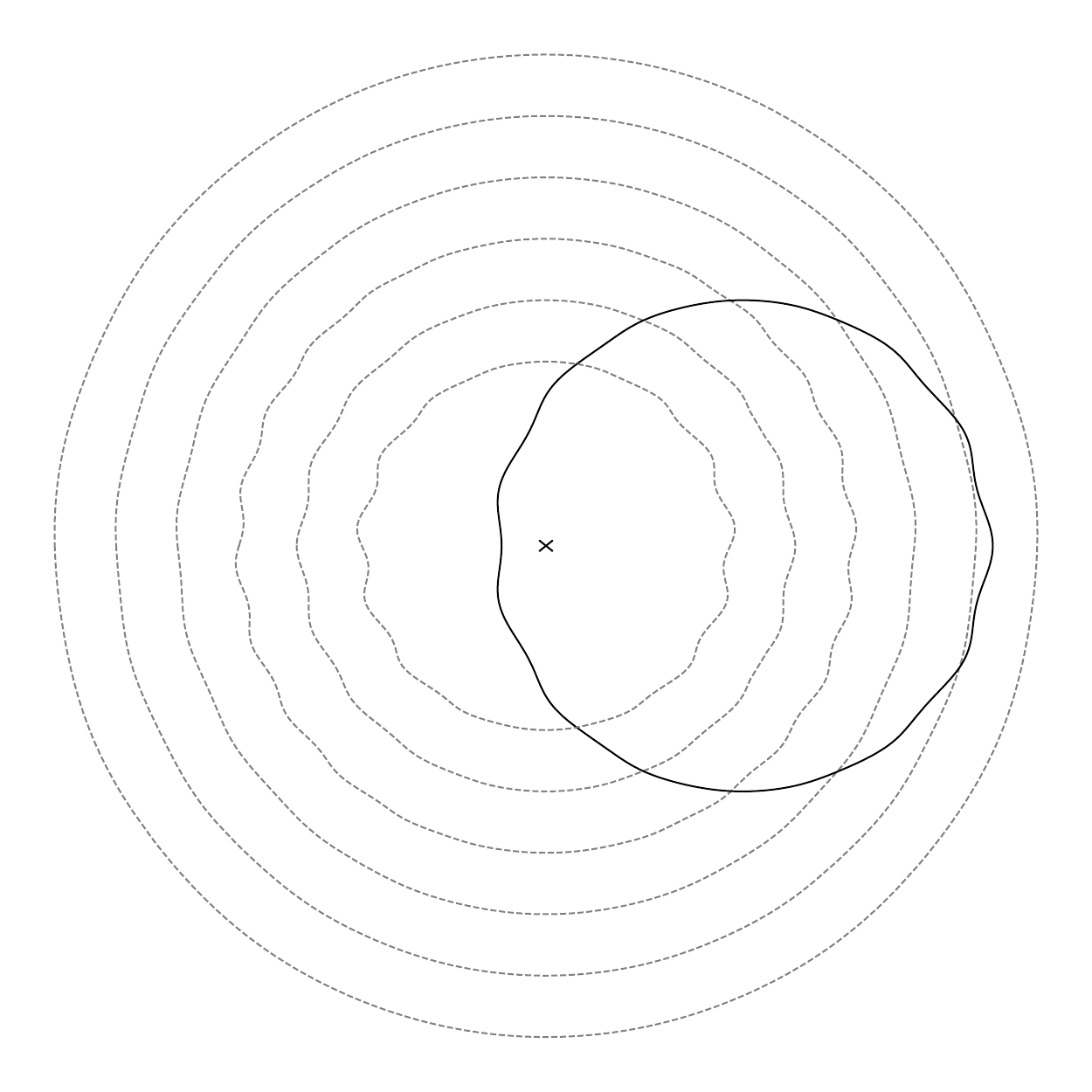}
\end{subfigure}
\begin{subfigure}{0.33\textwidth}
	
	\includegraphics[width=1\linewidth]{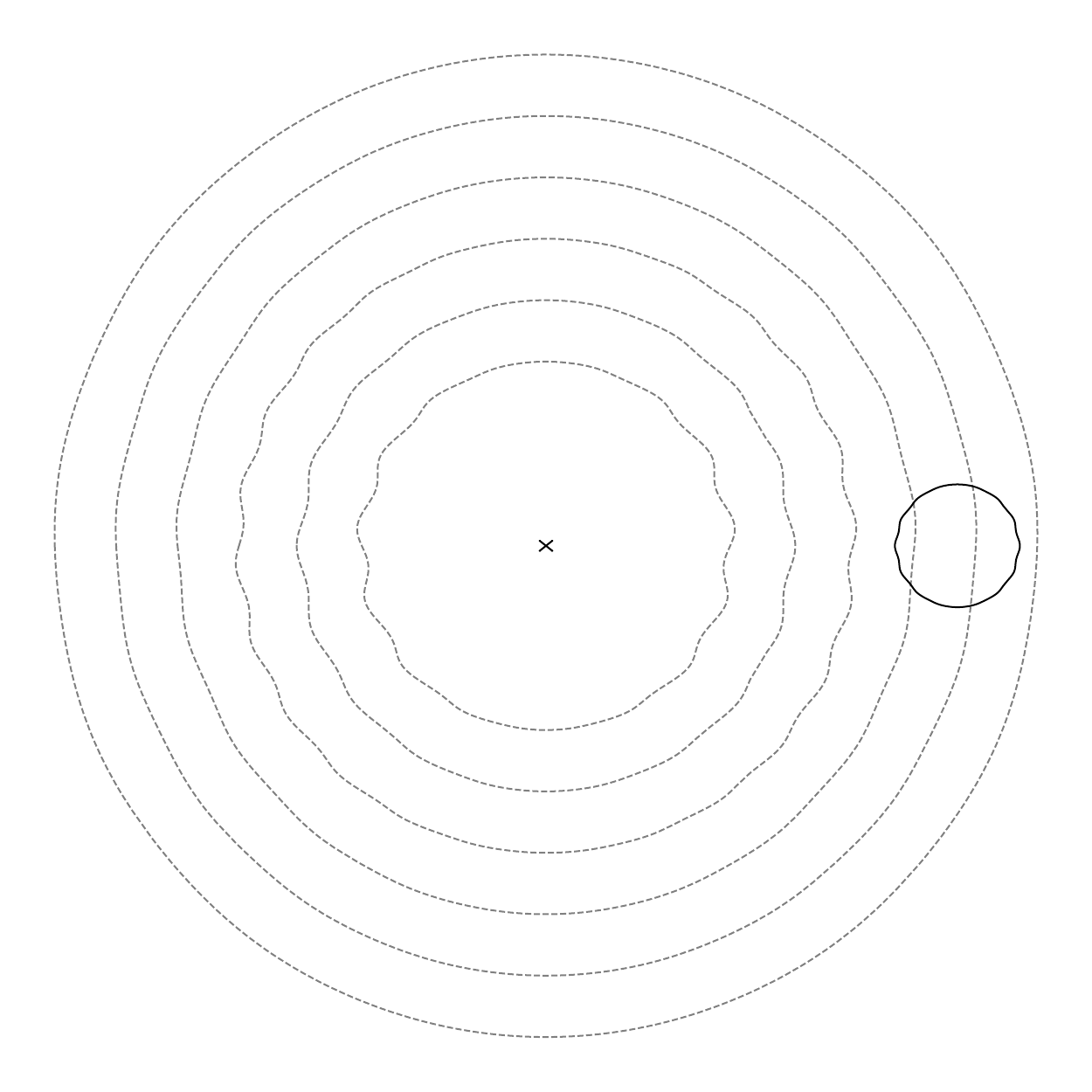}
\end{subfigure}
\caption{An illustration of the situation in Theorem \ref{uniqueness thm}. The dashed, gray lines indicate the leaves of the  foliation from Theorem \ref{existence thm} while the solid, black line indicates the surface $\Sigma$. On the left, $\Sigma$ belongs to the  foliation. In the middle, $\Sigma$ is on-center and the area radius dominates the inner radius. On the right, $\Sigma$ is outlying and the inner radius dominates the area radius.  $\Sigma$ violates the assumption \eqref{area and centering} of Theorem \ref{uniqueness thm} for $\delta=1/4$ in the latter two scenarios. Under the assumptions of Theorem \ref{far outlying thm}, the scenario on the right can be ruled out.}
\label{Figure alternatives}
\end{figure}
The assumptions on the scalar curvature in
Theorem \ref{existence thm}, Theorem \ref{uniqueness thm}, and Theorem \ref{far outlying thm} are essentially optimal. We show that the growth condition (\ref{decay R})
cannot be relaxed to requiring the scalar curvature to be non-negative and even. In fact, these weaker conditions are not sufficient to preclude large area-constrained Willmore spheres -- on-center or outlying -- that satisfy \eqref{area and centering} and (\ref{small energy assumption}) but do not belong to the foliation from  Theorem \ref{existence thm}. In particular, the assumptions conjecturally proposed in \cite{laurain2019analyse} are not quite sufficient to conclude that  large area-constrained Willmore spheres are unique.
\begin{thm}
	\label{counterexample thm 1}
There exist rotationally symmetric metrics $g_1$ and $g_2$  on $M=\{x\in\mathbb{R}^3:|x|>1\}$ both $C^k$-asymptotic to Schwarzschild with mass $m=2$ for every $k\geq 2$ and with non-negative scalar curvature such that the following holds. There exist sequences of  stable area-constrained Willmore spheres $\{\Sigma^1_j\}_{j=1}^\infty$ and $\{\Sigma^2_j\}_{j=1}^\infty$ that are on-center in $(M,g_1)$ and outlying in $(M,g_2)$, respectively, such that $$\lim_{j\to\infty}|\Sigma^1_j|=\lim_{j\to\infty}|\Sigma^2_j|=\infty,\quad  \lim_{j\to\infty}m_H(\Sigma^1_j)= 2,\quad  \lim_{j\to\infty}m_H(\Sigma^2_j)= 0,$$
while, for all $j$,
$$
\frac{1}{4}<\frac{\rho({\Sigma^1_j})}{\lambda({\Sigma^1_j})}<\frac{7}{8} \qquad\text{ and }\qquad  2\sqrt{2}<\frac{\rho({\Sigma^2_j})}{\lambda({\Sigma^2_j})}<5.
$$
\end{thm}
Conversely, centering of the  foliation from Theorem \ref{existence thm} may fail if the assumption that the scalar curvature is asymptotically even is dropped. We refer to the work of C.~Cederbaum and C.~Nerz  \cite{cederbaumexplicit} for a thorough investigation of various divergent notions of center of mass.
\begin{thm}
		\label{counterexample thm 2}
There exists a metric $g_3$ on $\{x\in\mathbb{R}^3:|x|>1\}$ $C^k$-asymptotic to Schwarzschild with mass $m=2$ for every $k\geq 2$ with non-negative scalar curvature satisfying (\ref{decay R}) such that the following holds. There exists  a number $\kappa_0>0$ and a smooth asymptotic foliation $\{\Sigma(\kappa): \kappa \in(0,\kappa_0)\}$ by on-center stable area-constrained Willmore spheres   such that 
$$
\limsup_{\kappa\to0}\frac{\lambda({\Sigma({\kappa})})}{\rho({\Sigma({\kappa})})}>1.
$$

\end{thm}
Finally, the following result shows that if we relax the growth condition on the scalar curvature only slightly, large far-outlying area-constrained Willmore spheres may exist. 
\begin{thm}
		\label{counterexample thm 3}
	There exists a rotationally symmetric metric 
	$$g_4=\left(1+|x|^{-1}\right)^4\bar g+\sigma_4$$ on $\{x\in\mathbb{R}^3:|x|>1\}$ with non-negative scalar curvature and, as $x\to\infty$ for every multi-index $J$,
$$\partial_J \sigma_4=O(|x|^{-3-|J|})$$ 
	  that has the following property. There exists  a sequence of  stable area-constrained Willmore spheres $\{\Sigma_j\}_{j=1}^\infty$ such that 
	 $$
	 \lim_{j\to\infty} |\Sigma_j|=\infty, \qquad \lim_{j\to\infty}m_H(\Sigma_j)=0, \qquad\text{and}\qquad \lim_{j\to\infty} \frac{\rho({\Sigma_j})}{\lambda({\Sigma_j})}=\infty.
	 $$ 

\end{thm}
The study of large area-constrained Willmore spheres $\Sigma$ with $\lambda(\Sigma)\gg\rho(\Sigma)$ is  challenging. This is on account of the loss of analytic control in the part of the surface where the area radius $\lambda(\Sigma)$ is much larger than $|x|$. In particular, the non-linearities owing to the Schwarzschild background dominate so the method  in this paper loses its grip.  We remark that large coordinate spheres cease to be mean convex in this regime and exhibit a first-order defect in their Hawking mass; see Remark \ref{slow divergence}. This suggests that the comparability assumptions on the area radius and inner radius in Theorem \ref{uniqueness thm} are not necessary.
\\
\indent
In order to prove Theorem \ref{existence thm} and Theorem \ref{uniqueness thm}, we use a strategy modeled upon the Lyapunov-Schmidt reduction developed for stable constant mean curvature spheres in \cite{brendle2014large,chodosh2019far}. We will follow by and large the notation of \cite{brendle2014large,chodosh2019far}  throughout this paper. \\ \indent By scaling, we may assume that $m=2$, that is,
$$
g=(1+|x|^{-1})^4\,\bar g +\sigma.
$$  
\indent We use a bar to indicate that a geometric quantity has been computed with respect to the Euclidean background metric $\bar g$. When the Schwarzschild metric
$$
g_S=(1+|x|^{-1})^4\,\bar g
$$
 with mass $m=2$ has been used in the computation, we use the subscript $S$.
\\ \indent 
For every $\xi\in\mathbb{R}^3$ and $\lambda>1$ large, depending on $|1-|\xi||^{-1}$, we use the implicit function theorem to perturb the sphere ${S}_{\lambda}(\lambda\,\xi)$ to a surface $\Sigma_{\xi,\lambda}$ with area $4\,\pi\, \lambda^2$ and which satisfies the area-constrained Willmore equation (\ref{constrained Willmore equation}) up to a sum of first spherical harmonics. On the one hand, we show that $\Sigma_{\xi,\lambda}$ is an area-constrained Willmore surface if and only if $\xi$ is a critical point of the function $G_\lambda$ given by
\begin{equation*}
G_\lambda(\xi)=\begin{dcases}&\lambda^2\bigg(\int_{\Sigma_{\xi,\lambda}} H^2\,\mathrm{d}\mu-16\,\pi+64\,\pi\,\lambda^{-1}\bigg)\, \text{ if }|\xi|<1, \\ 
&\lambda^2\bigg(\int_{\Sigma_{\xi,\lambda}} H^2\,\mathrm{d}\mu-16\,\pi\bigg)\hspace{1.75cm}\,\,\,\text{ if }|\xi|>1.
\end{dcases}
\end{equation*}
The absence of the term $64\,\pi\,\lambda^{-1}$ in the case $|\xi|>1$, which would equal $32\,\pi\,\lambda^{-1}\,m$ without the normalization $m=2$,  owes to the fact that the Hawking mass of an outlying coordinate sphere does not detect the  mass of $(M,g)$; see Lemma \ref{Willmore energy sphere}. On the other hand, we show that 
\begin{align*}
G_\lambda(\xi)=64\,\pi+\frac{32\,\pi}{1-|\xi|^2}-48\,\pi\,|\xi|^{-1}\log\frac{1+|\xi|}{1-|\xi|}-128\,\pi\log(1-|\xi|^2)
+2\,\lambda \int_{\mathbb{R}^3\setminus{B_{\lambda}(\lambda\,\xi)}} R\,\mathrm{d}\bar{v}+o(1)
\end{align*}
if $|\xi|<1$ and
\begin{align} \label{expansion 2}
G_\lambda(\xi)= \frac{32\,\pi}{1-|\xi|^2}-48\,\pi\,|\xi|^{-1}\log\frac{|\xi|+1}{|\xi|-1}-128\,\pi\log(1-|\xi|^{-2})
-2\,\lambda \int_{{B_{\lambda}(\lambda\,\xi)}} R\,\mathrm{d}\bar{v}+o(1)
\end{align}
if $|\xi|>1$. Assuming that the scalar curvature is asymptotically even (\ref{even R}) and satisfies the growth condition (\ref{decay R}), we show that $G_\lambda$ has a unique minimum near the origin for every  $\lambda>1$ large. Moreover,  the corresponding surfaces $\Sigma_{\xi,\lambda}$ form a foliation. We show that these are in fact the only critical points of $G_\lambda$  if $\lambda>1$ is large. Theorem \ref{uniqueness thm}  follows from this and a compactness argument. The strategy for the proof of Theorem \ref{far outlying thm} is formally similar. However, a key difference is that the first three terms in (\ref{expansion 2}) become very small when $\xi$ is large. In a more precise analysis, we  verify that (\ref{expansion 2}) remains true  up to lower-order error terms that decay sufficiently fast as  $|\xi|\to\infty$. Finally, we give explicit examples that show that $G_\lambda$ may have other critical points for suitable choices of $\sigma$ that violate the particular assumptions on the scalar curvature.
\\
\indent Unlike large area-constrained Willmore spheres, small area-constrained Willmore spheres in closed manifolds are well-understood. It is shown in the work  of T.~Lamm and J.~Metzger \cite{LammMetzger} and the work  of A.~Mondino and T.~Riviere \cite{rivieremondino} that minimizers of the Willmore energy with small prescribed area exist in closed manifolds. Moreover, T.~Lamm, J.~Metzger, and F.~Schulze \cite{lms2019local} as well as N.~Ikoma, A.~Malchiodi, and A.~Mondino \cite{ikoma2017embedded} have shown that a neighborhood of a non-degenerate critical point of the scalar curvature is foliated by small area-constrained Willmore spheres. This extends previous work of T.~Lamm
and J.~Metzger \cite{lm2} as well as of A.~Mondino and P.~Laurain \cite{laurainmondino}. We also mention the work \cite{alessandroni2016local} of R.~Alessandroni and E.~Kuwert on small area-constrained Willmore surfaces with free boundary and the work  by A.~Mondino \cite{mondino2013conformal}, where unconstrained Willmore surfaces are studied in a semi-perturbative setting. The method of Lyapunov-Schmidt reduction is used in \cite{ikoma2017embedded,alessandroni2016local,mondino2013conformal}.\\ \newline
\indent \textbf{Acknowledgments.} The authors would like to thank Simon Brendle, Otis Chodosh, Jan Metzger, and Felix Schulze for helpful discussions and the anonymous referees for helpful remarks. The authors acknowledge the support of the START-Project Y963-N35 of the Austrian Science Fund (FWF).
\section{Proof of Theorem \ref{existence thm} and Theorem \ref{uniqueness thm}}
In this section, we assume that $g$ is a Riemannian metric on $\mathbb{R}^3$ such that
\begin{align} \label{s2 asymptotic to Schwarzschild} 
g=(1+|x|^{-1})^4\,\bar{g} +\sigma
\end{align} 
where $\sigma$ is a symmetric, covariant two-tensor with, as $x\to\infty$ for every multi-index $J$ with $|J|\leq 4$,
$$\partial_J \sigma=O(|x|^{-2-|J|}).$$
  Applying a strategy similar to that in \cite{brendle2014large,chodosh2019far}, we perform a Lyapunov-Schmidt reduction  to analyze area-constrained Willmore spheres $\Sigma\subset\mathbb{R}^3$ with large area  and comparable area radius $\lambda({\Sigma})$ and inner radius $\rho({\Sigma})$. \\\indent 
Given $\xi\in\mathbb{R}^3$ and $\lambda>0$, we abbreviate
$${S}_{\xi,\lambda}=S_{\lambda}(\lambda\,\xi)=\{x\in\mathbb{R}^3:|x-\lambda\,\xi|=\lambda\}.$$
Given $u\in C^{\infty}(S_{\xi,\lambda})$, we define the map
\begin{align}
\Phi^u_{\xi,\lambda}:S_{\xi,\lambda}\to \mathbb{R}^3 \qquad \text{given by} \qquad  \Phi^u_{\xi,\lambda}(x)=x+u(x)\, (\lambda^{-1}\,x-\xi). \label{graph map}
\end{align}
We denote by $$\Sigma_{\xi,\lambda}(u)=\Phi^u_{\xi,\lambda}(S_{\xi,\lambda})$$ the Euclidean graph of $u$ over $S_{\xi,\lambda}$. Throughout, for instance in (\ref{willmore in lambda1}) below, we tacitly identify functions defined on $\Sigma_{\xi,\lambda}(u)$ with functions defined on $S_{\xi,\lambda}$ by precomposition with $\Phi^u_{\xi,\lambda}$. 
 \\\indent
Let $\alpha\in(0,1)$ and $\gamma>1$. We denote by $\mathcal{G}$  the space of $C^{3,\alpha}$-Riemmanian metrics  on  $$\{y\in\mathbb{R}^3:\gamma^{-1}\leq |y|\leq \gamma\}$$   with the $C^{3,\alpha}$-topology. Let $\Lambda_0(S_1(0))$ and $\Lambda_1(S_1(0))$ be the constants and first spherical harmonics viewed as subspaces of $C^{4,\alpha}(S_1(0))$, respectively. We use the symbol $\perp$ to denote the $L^2(S_1(0))$-orthogonal complements of these spaces.
\begin{lem}
	There exist open neighborhoods $\mathcal U$ of $\bar g\in \mathcal{G}$, $\mathcal{V}$ of $0\in  \Lambda_1(S_1(0))^\perp$, and $I$ of $0\subset \mathbb{R}$ as well as smooth maps $u:\mathcal{U}\to \mathcal{V}$ and $\kappa:\mathcal{U}\to I$ such that the surface  $\Sigma_{0,1}(u(g))$ has area equal to $4\,\pi$ and satisfies
	\begin{align}
	\Delta H+(|\hcirc|^2+\operatorname{Ric}(\nu,\nu)+\kappa(g))\,H\in \Lambda_1(S_1(0)).
	\label{willmore in lambda1}
	\end{align}
Here,	\label{inverse function theorem lemma}
	 all geometric quantities are computed with respect to the surface $\Sigma_{0,1}(u(g))$ and the metric $g$. Moreover, if $g_0\in\mathcal{U}$, $u_0\in\mathcal{V}$, and $\kappa_0\in I$ are such that $\Sigma_{0,1}(u_0)$ satisfies (\ref{willmore in lambda1}) with respect to $g_0$ and has area equal to $4 \,\pi$, then $u_0=u(g_0)$ and $\kappa_0=\kappa(g_0)$.
\end{lem}
\begin{proof}

Let $\Lambda_{0,0}(S_1(0))$ and $\Lambda_{1,0}(S_1(0))$ be the constants and first spherical harmonics viewed as subspaces of $C^{0,\alpha}(S_1(0))$, respectively. Note that there are neighborhoods $\tilde {\mathcal U}$ of $\bar g\in \mathcal{G}$ and $\tilde {\mathcal{V}}$ of $0\in  \Lambda_1(S_1(0))^\perp$ such that the map
\[
T:\tilde{\mathcal V}\times\mathbb{R}\times \tilde{\mathcal U} \to\Lambda_{1,0}(S_1(0))^\perp\times \mathbb{R}
\]
given by 
\[
T(u,\,\kappa,\,g)=\left(\operatorname{proj}_{\Lambda_{1,0}(S_1(0))^\perp}\left[ \Delta H+(|\hcirc|^2+\operatorname{Ric}(\nu,\nu)+\kappa)\,H\right],\,|\Sigma|\right),
\]
where all geometric quantities  are with respect to $\Sigma_{0,1}(u)$ and the metric $g$, is well-defined and smooth. Specifying  (\ref{Q general formula}) to $S_1(0)$ in flat $\mathbb{R}^3$, we find
$$
(D T)|_{(0,\,0,\,\bar g)}(u,\,0,\,0)=\left(-\bar \Delta^2 u-2\,\bar \Delta u,\,8\,\pi\,\operatorname{proj}_{\Lambda_0(S_1(0))} u\right).
$$
Moreover, there holds
$$
(D T)|_{(0,\,0,\,\bar g)}(0,\,\kappa,\,0)=(2\,\kappa,\,0).
$$
As discussed in Corollary \ref{Willmore eigenvalues}, the kernel of the operator
$$
-\bar\Delta^2-2\,\bar\Delta:C^{4,\alpha}(S_1(0))\to C^{0,\alpha}(S_1(0)) 
$$
is given by $\Lambda_0(S_1(0))\oplus\Lambda_1(S_1(0))$. It follows from the Fredholm alternative and elliptic regularity
that
$$
-\bar\Delta^2-2\,\bar\Delta:[\Lambda_0(S_1(0))\oplus\Lambda_1(S_1(0))]^\perp\to [\Lambda_{0,0}(S_1(0))\oplus \Lambda_{1,0}(S_1(0))]^\perp 
$$
is an isomorphism. Thus, 
$$
(DT)|_{(0,\,0,\,\bar g)}(\,\cdot\,,\,\cdot\,,\,0):\Lambda_1(S_1(0))^\perp\times \mathbb{R}\to \Lambda_{1,0}(S_1(0))^\perp\times\mathbb{R}
$$
is an isomorphism, too. The assertions follow from this and the implicit function theorem.
\end{proof}
We consider the map $$\Theta_{\xi,\lambda}:\mathbb{R}^3\to\mathbb{R}^3\qquad \text{ given by }\qquad  \Theta_{\xi,\lambda}(y)= \lambda\,(\xi+y).$$ Note that $\Theta_{\xi,\lambda}(S_1(0))=S_{\xi,\lambda}$.  The rescaled metric
\begin{align}
 g_{\xi,\lambda}=\lambda^{-2}\,\Theta_{\xi,\lambda}^*\,g \label{pull back metric}
\end{align}
satisfies, as $\lambda\to\infty$,
$$
||g_{\xi,\lambda}-\bar{g}||_{\mathcal{G}}=O(\lambda^{-1}\,|1-|\xi||^{-1}).
$$

 \indent Let $\delta\in(0,1/2)$.  The following proposition follows from  Lemma \ref{inverse function theorem lemma} and scaling. We let $\Lambda_0(S_{\xi,\lambda})$ and $\Lambda_1(S_{\xi,\lambda})$ be the constants and first spherical harmonics viewed as subspaces of $C^{4,\alpha}(S_{\xi,\lambda})$, respectively. 
\begin{prop}
	There are constants $\lambda_0>1$, $c>1$, and $\epsilon>0$ depending on $g$ and $\delta\in(0,1/2)$ such that for every $\xi\in\mathbb{R}^3$ with $|\xi|<1-\delta$ or $|\xi|>1+\delta$ and every $\lambda>\lambda_0$ there exist $u_{\xi,\lambda}\in C^{\infty}(S_{\xi,\lambda})$ and $\kappa_{\xi,\lambda}\in\mathbb{R}$ such that the following hold. The surface $$\Sigma_{\xi,\lambda}=\Sigma_{\xi,\lambda}(u_{\xi,\lambda})$$ has the  properties
	\begin{itemize}
	\item[$\circ$] $\Delta H+(|\hcirc|^2+\operatorname{Ric}(\nu,\nu)+\kappa_{\xi,\lambda})\,H\in \Lambda_1(S_{\xi,\lambda})$ and
	\item[$\circ$] $|\Sigma_{\xi,\lambda}|=4\,\pi\, \lambda^2$.
	\end{itemize}
There holds $u_{\xi,\lambda}\perp \Lambda_1(S_{\xi,\lambda}) $  and 	\label{LS proposition}
	\begin{equation*}
	\begin{aligned}
	|u_{\xi,\lambda}|+\lambda\,|\bar \nabla u_{\xi,\lambda}|+\lambda^2\,|\bar \nabla^2 u_{\xi,\lambda}|+\lambda^3\,|\bar \nabla^3 u_{\xi,\lambda}|+\lambda^4\,|\bar \nabla^4 u_{\xi,\lambda}|&<c,\\
	\lambda^3\,|\kappa_{\xi,\lambda}|&<c.
	\end{aligned}
	\end{equation*}
   Moreover, if $\kappa\in\mathbb{R}$ and  $\Sigma_{\xi,\lambda}(u)$ with $u\perp \Lambda_1(S_{\xi,\lambda})$  are such that
\begin{itemize}
	\item[$\circ$] $\Delta H+(|\hcirc|^2+\operatorname{Ric}(\nu,\nu)+\kappa)\,H\in \Lambda_1(S_{\xi,\lambda}),$
	\item[$\circ$]  $|\Sigma_{\xi,\lambda}(u)|=4\,\pi\, \lambda^2,$
\end{itemize}
and
\begin{align*}
	|u|+\lambda\,|\bar \nabla u|+\lambda^2\,|\bar \nabla^2 u|+\lambda^3\,|\bar \nabla^3 u|+\lambda^4\,|\bar \nabla^4 u|&<\epsilon\,\lambda, \\\lambda^3\,|\kappa|&<\epsilon\,\lambda,
\end{align*}
then $u=u_{\xi,\lambda}$ and $\kappa=\kappa_{\xi,\lambda}$.
\end{prop}
\begin{rema} \normalfont
By the implicit function theorem and scaling, 
	\begin{align*}
(\bar D u)|_{(\xi,\,\lambda)}=O(\lambda^{-1})  \qquad\text{and} \qquad  u'|_{(\xi,\,\lambda)}=O(\lambda^{-2}) 
	\end{align*} 
	where  \label{u ambient derivatives ren} 
$\bar D$ and the dash indicate differentiation with respect to the parameters $\xi$ and $\lambda$, respectively. 
\end{rema}
We abbreviate $u_{\xi,\lambda}$ by $u$, $\kappa_{\xi,\lambda}$ by $\kappa$, and $\Lambda_l(S_{\xi,\lambda})$ by $\Lambda_l$ for $l=0,\,1$. To obtain more precise information about the perturbation, we expand $u$ in terms of spherical harmonics. 
\begin{lem}
	There holds 
	\begin{align*}
	\operatorname{proj}_{\Lambda_0}u=\begin{dcases}
&-2+{O}(\lambda^{-1})\, \hspace{0.8cm}\quad\text{ if }|\xi|<1-\delta , 
\\&-2\,|\xi|^{-1}+{O}(\lambda^{-1}) \quad\text{ if }|\xi|>1+\delta.
\end{dcases} 
	\end{align*}
	\label{Lambda0 estimate}
\end{lem}
\begin{proof}
	On the one hand, we have $|\Sigma_{\xi,\lambda}|=4\,\pi\, \lambda^2$ by construction. By Lemma \ref{area expansion}, $$|S_{\xi,\lambda}|=
	\begin{dcases}
	&4\,\pi\, \lambda^2+16\,\pi\,  \lambda+{O}(1)\hspace{1.2cm}\,\text{  if } |\xi|<1-\delta,\\
	&4\,\pi\, \lambda^2+16\,\pi\,  \lambda\,|\xi|^{-1}+{O}(1) \quad \hspace{.03cm}\text{ if } |\xi|>1+\delta.
	\end{dcases}$$ On the other hand,  by the first variation of area formula, 
	$$
	|\Sigma_{\xi,\lambda}|-|S_{\xi,\lambda}|=\int_{S_{\xi,\lambda}}H\,u\,\mathrm{d}\mu+{O}(1)=\int_{S_{\xi,\lambda}}\bar H\,u\,\mathrm{d}\bar\mu+{O}(1).
	$$ 
	In the second equation, we have used Lemma \ref{Schwarzschild identities} and Lemma \ref{perturbations}. Since 
	$$
	\frac{1}{4\,\pi}\,\lambda^{-2}\,\int_{S_{\xi,\lambda}}u\,\mathrm{d}\bar\mu =\operatorname{proj}_{\Lambda_0}u,
	$$
	the assertion follows.
\end{proof}
For the statement of the next lemma, recall the definition of the Legendre polynomials $P_l$ from Appendix \ref{spherical harmonics appendix}.
\begin{lem}
	\label{u, lambda expansions}
	If $|\xi|<1-\delta$, there holds
	\begin{align*}
\kappa&=4\,\lambda^{-3}+O(\lambda^{-4}),\\
W({\Sigma_{\xi,\lambda}})+\kappa\, H(\Sigma_{\xi,\lambda}) &={O}(\lambda^{-5}), 
\\
u&=-2+4\sum_{l=2}^\infty \frac{|\xi|^{l}}{l}\,P_l\left(-|\xi|^{-1}\,\bar{g}(y,\xi)\right)+O(\lambda^{-1}).
\end{align*}
	If $|\xi|>1+\delta$, there holds
\begin{align*}
\kappa&=O(\lambda^{-4}),\\
W({\Sigma_{\xi,\lambda}})+\kappa\, H(\Sigma_{\xi,\lambda})&={O}(\lambda^{-5}),
\\
u&=-2\,|\xi|^{-1}-4\sum_{l=2}^\infty \frac{|\xi|^{-l-1}}{l+1}\,P_l\left(-|\xi|^{-1}\,\bar{g}(y,\xi)\right)+O(\lambda^{-1}).
\end{align*}
\end{lem}
\begin{proof}
It follows from (\ref{Q change of W}) and  Lemma \ref{Q lem} that 
	$$
	{W}({\Sigma_{\xi,\lambda}})-{W}({{S}_{\xi,\lambda}})=-Q_{S_{\xi,\lambda}}u+O(\lambda^{-5})=-\bar\Delta_{S_{\xi,\lambda}}^2u-2\,\lambda^{-2}\,\bar\Delta_{S_{\xi,\lambda}} u+{O}(\lambda^{-5}).
	$$
	Likewise, (\ref{mean curvature change}), Proposition \ref{LS proposition}, Lemma \ref{Schwarzschild identities}, and Lemma \ref{perturbations} imply that
	$$
	H(\Sigma_{\xi,\lambda})=2\,\lambda^{-1}+O(\lambda^{{-2}}).
	$$
	Conversely, by Proposition \ref{LS proposition},
	$$
	{W}({\Sigma_{\xi,\lambda}})+H(\Sigma_{\xi,\lambda})\,\kappa=Y_1\in\Lambda_1.
	$$
 Thus,  
\begin{align}
\bar 	\Delta_{S_{\xi,\lambda}}^2u+2\,\lambda^{-2}\,\bar \Delta_{S_{\xi,\lambda}} u-2\, \lambda^{-1}\,\kappa=
	W({{S}_{\xi,\lambda}})-Y_1+{O}(\lambda^{-5}).
\label{Willmore sphere vs u}
\end{align}
By Corollary \ref{Willmore operator on sphere spherical harmonics}, there holds	$$
		W({{S}_{\xi,\lambda}})=
		\begin{dcases}
		&\,\,4\,\lambda^{-4}\,\sum_{l=0}^\infty (l-1)\,(l+1)\,(l+2)\,|\xi|^{l}\,P_l\left(-|\xi|^{-1}\,\bar{g}(y,\xi)\right)+{O}(\lambda^{-5})
\quad \text{ if }|\xi|<1-\delta,\\&-\,4\,\lambda^{-4}\sum_{l=0}^\infty (l-1)\,l\,(l+2)\,|\xi|^{-l-1}\,P_l\left(-|\xi|^{-1}\,\bar{g}(y,\xi)\right)+{O}(\lambda^{-5})\,\,
\quad \,\text{ if }|\xi|>1+\delta.
\end{dcases}$$  Projecting (\ref{Willmore sphere vs u}) onto $\Lambda_1$, we find $Y_1=O(\lambda^{-5})$. The assertions follow from Corollary \ref{Willmore eigenvalues}. 
\end{proof}
To relate the variational structure of the area-constrained Willmore equation on the families of surfaces $\{\Sigma_{\xi,\lambda}:|\xi|<1-\delta\}$ and $\{\Sigma_{\xi,\lambda}:|\xi|>1+\delta\}$ to  a 3-dimensional problem, we introduce the functional
\begin{align}
F_\lambda(\Sigma)=\begin{dcases}
&\lambda^2\,\bigg(\int_{\Sigma} H^2\,\mathrm{d}\mu-16\,\pi+64\,\pi\,\lambda^{-1}\bigg) \quad
\text{ if }\Sigma\text{ is on-center},\\
&\lambda^2\,\bigg(\int_{\Sigma} H^2\,\mathrm{d}\mu-16\,\pi\bigg)\, \hspace{1.8cm}\quad \text{ if }\Sigma\text{ is outlying},
\end{dcases} 
\label{F functional}
\end{align}
for closed, two-sided surfaces $\Sigma\subset M$. Essentially, $F_\lambda$  measures the Willmore energy on the relevant scales for on-center and outlying surfaces, respectively. We then define the function
\begin{align}
G_\lambda:\{\xi\in\mathbb{R}^3:|\xi|<1-\delta\text{ or }|\xi|>1+\delta \}\to\mathbb{R}\qquad\text{given by} \qquad G_\lambda(\xi)=F_\lambda(\Sigma_{\xi,\lambda}). \label{G definition}
\end{align}
\begin{lem} There is $\lambda_0>1$ depending  on $g$ and $\delta\in(0,1/2)$ with the following property. Let $\lambda>\lambda_0$. Then $\Sigma_{\xi,\lambda}$ is an area-constrained 	\label{variational to 3 dim} Willmore sphere if and only if $\xi$ is a critical point of $G_\lambda$.
\end{lem}
\begin{proof}
Fix $\xi\in\mathbb{R}^3$ with either $|\xi|<1-\delta$ or $|\xi|>1+\delta$. \\ \indent Let $a\in\mathbb{R}^3$ with $|a|=1$ and $\epsilon>0$ be small. Note that the normal speed $f$ of the area-preserving variation
	$$
	\{\Sigma_{\xi+s\hspace{0.02cm}a,\lambda}:|s|<\epsilon\}
	$$
	of $\Sigma$ at $s=0$ is given by
	$$
	f=g\bigg(\nu,\frac{d}{ds}\bigg|_{s=0}\Phi^{u_{\xi+s\,a,\lambda}}_{\xi+s\,a,\lambda}\bigg)=\lambda\, \bar g(a,\bar\nu)+O(1).
	$$ 
 \indent Assume that $\xi$ is a critical point of $G_\lambda$. Using (\ref{first variation}), we find
	$$
	\int_{\Sigma_{\xi,\lambda}} W(\Sigma_{\xi,\lambda}) \,f\,\mathrm{d}\mu=0, \qquad \int_{\Sigma_{\xi,\lambda}}H(\Sigma_{\xi,\lambda})\,f\,\mathrm{d}\mu=0.
	$$
	In particular,
	$$
	\int_{\Sigma_{\xi,\lambda}} (W(\Sigma_{\xi,\lambda})+\kappa\, H(\Sigma_{\xi,\lambda}))\,(\bar g(a,\bar\nu)+O(\lambda^{-1}))\,\mathrm{d}\mu=0
	$$
	for every choice of $a\in\mathbb{R}^3$ with $|a|=1$. Since $W(\Sigma_{\xi,\lambda})+\kappa\, H(\Sigma_{\xi,\lambda})\in \Lambda_1$ by Proposition \ref{LS proposition}, it follows that 
	$
	W(\Sigma_{\xi,\lambda})+\kappa\, H(\Sigma_{\xi,\lambda})=0
	$
	provided $\lambda_0>1$ is sufficiently large. \\ \indent
	Conversely, if $\Sigma_{\xi,\lambda}$ is an area-constrained Willmore sphere, then
	$$
	\int_{\Sigma_{\xi,\lambda}} W(\Sigma_{\xi,\lambda})\,f\,\mathrm{d}\mu=-\kappa\int_{\Sigma_{\xi,\lambda}} H(\Sigma_{\xi,\lambda})\,f\,\mathrm{d}\mu=0.
	$$
In conjunction with	Lemma \ref{variation lemma}, we see that $\xi$ is a critical point of $G_\lambda$. 
\end{proof} 
In the next step, we compute the asymptotic expansions of $G_\lambda$ as $\lambda\to\infty$. 
\begin{lem}
		\label{G expansion}
	If $|\xi|<1-\delta$, there holds
\begin{equation*}
\begin{aligned}
G_\lambda(\xi)=64\,\pi+\frac{32\,\pi}{1-|\xi|^2}-48\,\pi\,|\xi|^{-1}\log\frac{1+|\xi|}{1-|\xi|}-128\,\pi\log(1-|\xi|^2)
+2\,\lambda \int_{\mathbb{R}^3\setminus{B_{\lambda}(\lambda\,\xi)}} R\,\mathrm{d}\bar{v}+O(\lambda^{-1}).
\end{aligned}
\end{equation*}
If $|\xi|>1+\delta$, there holds	
\begin{equation*}
\begin{aligned}
G_\lambda(\xi)= -\frac{32\,\pi}{|\xi|^2-1}-48\,\pi\,|\xi|^{-1}\,\log\frac{|\xi|+1}{|\xi|-1}-128\,\pi\log(1-|\xi|^{-2})
-2\,\lambda \int_{B_{\lambda}(\lambda\,\xi)} R\,\mathrm{d}\bar{v}+O(\lambda^{-1}).
\end{aligned}
\end{equation*}
\end{lem}
\begin{proof}
		Using Lemma \ref{variation lemma}, we compute
	\begin{align*}
	\int_{\Sigma_{\xi,\lambda}} H(\Sigma_{\xi,\lambda})^2\,\mathrm{d}\mu=&\,\int_{{S}_{\xi,\lambda}}H^2\,\mathrm{d}\mu-2\int_{{S}_{\xi,\lambda}} {W}\,u\,\mathrm{d}\mu	+\int_{{S}_{\xi,\lambda}}\left[ u\,Qu-{W}\,H\,u^2\right]\mathrm{d}\mu+O(\lambda^{-3})
	\\=&\,\int_{{S}_{\xi,\lambda}}H^2\,\mathrm{d}\mu-2\int_{{S}_{\xi,\lambda}} {W}\,u\,\mathrm{d}\mu
	+\int_{{S}_{\xi,\lambda}} u\,Qu\,\mathrm{d}\mu+O(\lambda ^{-3}),
	\end{align*}
where we have abbreviated $H=H(S_{\xi,\lambda})$, $W=W({S_{\xi,\lambda}})$, $\Delta=\Delta_{S_{\xi,\lambda}}$, and $Q=Q_{S_{\xi,\lambda}}$.
Using (\ref{Q formula}) and (\ref{stability operator}), we find that
	$$
	\int_{{S}_{\xi,\lambda}}u\, Qu\,\mathrm{d}\mu=\int_{{S}_{\xi,\lambda}}\left[(\bar\Delta u)^2+2\,\lambda^{-2}\,u\,\bar\Delta u\right]\mathrm{d}\bar\mu+O(\lambda^{-3}).
	$$
	Conversely, Lemma \ref{u, lambda expansions} and  (\ref{Willmore sphere vs u}) imply
	$$
\int_{{S}_{\xi,\lambda}} {W}\,u\,\mathrm{d}\mu=	
\int_{{S}_{\xi,\lambda}}\left[ (\bar\Delta u)^2+2 \,\lambda ^{-2}\,u\,\bar\Delta u-2\,\lambda^{-1}\,\kappa\, u\right]\mathrm{d}\bar\mu +O(\lambda^{-3}). 
	$$
Using Lemma \ref{u, lambda expansions} again, we obtain  
\begin{align*}
2\,\lambda^{-1}\,\kappa\int_{{S}_{\xi,\lambda}}   u\,\mathrm{d}\bar\mu=
\begin{dcases}
&-64\,\pi\, \lambda^{-2}+{O}(\lambda^{-3}) \quad \text{ if } |\xi|<1-\delta, \\
&\,{O}(\lambda^{-3}) \hspace{2.45cm}\text{ if } |\xi|>1+\delta.
\end{dcases}
\end{align*}
 Using Corollary \ref{Willmore eigenvalues}, Lemma \ref{u, lambda expansions}, (\ref{legendre polynomials are spherical harmonics}), and Lemma \ref{Legendre identities} in the case where $|\xi|<1-\delta$, we compute
\begin{align*}
&-\int_{{S}_{\xi,\lambda}}\left[ (\bar\Delta u )^2+2\,\lambda^{-2}\,u\,\bar\Delta u\right]\mathrm{d}\bar\mu
\\&\qquad=\,-16\,\lambda^{-4}\sum_{l=2}^\infty  \frac{(l-1)\,(l+1)\,(l+2)}{l}\,|\xi|^{2l}\int_{{S}_{\xi,\lambda}} P_l^2\left(-|\xi|^{-1}\,\bar{g}(y, \xi)\right)\mathrm{d}\bar \mu
\\&\qquad=\,-64\,\pi\, \lambda^{-2}\sum_{l=2}^\infty \frac{(l-1)\,(l+1)\,(l+2)}{l\,(2\,l+1)}\,|\xi|^{2l}
\\&\qquad=\,-16\,\pi\,\lambda^{-2}\bigg[\frac{9}{2}\,|\xi|^{-1}\log\frac{1+|\xi|}{1-|\xi|}+8\,\log(1-|\xi|^2)+\frac{23\,|\xi|^2-12\,|\xi|^4-9}{(1-|\xi|^2)^2}\bigg].
\end{align*}
Similarly, if $|\xi|>1+\delta$, we obtain	  
\begin{align*}
-\int_{{S}_{\xi,\lambda}} \left[(\bar \Delta u )^2+2\,\lambda^{-2}\,u\,\bar\Delta u \right]\mathrm{d}\bar\mu
=&\,-64\,\pi\,\lambda^{-2} \sum_{l=2}^\infty \frac{(l-1)\,l\,(l+2)}{(l+1)\,(2\,l+1)}\,|\xi|^{-2l-2}.
\\
=&\,-16\,\pi\,\lambda^{-2}\bigg[\frac{9}{2}\,|\xi|^{-1}\log\frac{|\xi|+1}{|\xi|-1}+8\,\log(1-|\xi|^{-2})+\frac{3-|\xi|^2}{(|\xi|^2-1)^2}
\bigg].
\end{align*}
We have used Lemma \ref{log identities} in the last equation. The assertions follow from this and  Lemma \ref{Willmore energy sphere}.
\end{proof}
In order to proceed, we note the following technical result.
\begin{lem}
		\label{C2  estimate}
 There are  $c>0$ and $\lambda_0>1$ which only depend on $g$ and $\delta\in(0,1/2)$ such that 
	$$
	||G_\lambda||_{C^{3}(\{\xi\in\mathbb{R}^3:|\xi|\leq1-\delta\text{ or }|\xi|\geq 1+ \delta\})}< c
	$$
for every $\lambda> \lambda_0$.
\end{lem}
\begin{proof}
 This estimate follows from a straightforward computation using the regularity properties of the implicit function theorem, \eqref{s2 asymptotic to Schwarzschild}, and the variational formulae for the Willmore energy in the same way as in the proof of  Proposition 6 in \cite{brendle2014large}.
\end{proof}
We now investigate the qualitative behavior of $G_\lambda$ for  large values of $\lambda$. In this analysis, the assumptions
\begin{align}
\sum_{i=1}^3x^i\,\partial_i (|x|^2\,R)&\leq o(|x|^{-2}), \label{sRT}\\
R(x)-R(-x)&=o(|x|^{-4}), \label{sRT2}
\end{align}
are used. Note that (\ref{sRT}) integrates to \begin{align} R\geq-o(|x|^{-4}).\label{sRT3}\end{align}
We remark that the weaker decay \begin{align} R=O(|x|^{-4}) 
\label{R decay}\end{align} is implied by \eqref{s2 asymptotic to Schwarzschild}.
\begin{lem} Suppose that the scalar curvature $R$ of $(\mathbb{R}^3,g)$ satisfies (\ref{sRT}) and (\ref{sRT2}). \label{G der sRT}
	There exist $\tau>0$, $\delta_0\in(0,1/2)$, and $\lambda_0>1$ depending only on $g$ such that, provided $\lambda>\lambda_0$,
	$$ \bar D^2
	G_\lambda\geq \tau\,\operatorname{Id}$$ holds  on $ \{\xi\in\mathbb{R}^3:|\xi|<\delta_0\}$. Moreover, given $\delta\in(0,1/2)$ and $\delta_1\in(0,1-\delta)$, there is $\lambda_1>\lambda_0$ such that  $G_\lambda$ is strictly increasing in radial directions on $\{\xi\in\mathbb{R}^3:\delta_1<|\xi|<1-\delta\}$ provided  $\lambda>\lambda_1$. 
\end{lem}
\begin{proof}  We write 
	\begin{align}
	G_\lambda=G_{1}+G_{\lambda,2}+O(\lambda^{-1})
	\label{G decomposition}
	\end{align}
	 where
	\begin{align}
	G_{1}(\xi)=64\,\pi+\frac{32\,\pi}{1-|\xi|^2}-48\,\pi\,|\xi|^{-1}\log\frac{1+|\xi|}{1-|\xi|}-128\,\pi\log(1-|\xi|^2) \label{G1}
	\end{align}
	and
	$$
	G_{\lambda,2}(\xi)=2\,\lambda\int_{\mathbb{R}^3\setminus{B_{\lambda}(\lambda\,\xi)}}R\,\mathrm{d}\bar{v}.
	$$
The $C^4$-decay of the metric implies that the family of functions $\{G_{2,\lambda}:\lambda>\lambda_0\}$ is uniformly bounded in $C^{3}(\{\xi\in\mathbb{R}^3:|\xi|\leq1-\delta\})$, provided $\lambda_0>1$ is sufficiently large. Hence, by Lemma \ref{C2  estimate} and interpolation,  the error term in (\ref{G decomposition}) converges to $0$ in $C^2(\{\xi\in\mathbb{R}^3:|\xi|\leq 1- \delta\})$. \\ \indent  
We compute
$$
\sum_{i=1}^3|\xi|^{-1}\,\xi^i\,(\partial_i  G_{\lambda,2})(\xi)=-2\,\lambda^2\int_{S_{\xi,\lambda}}|\xi|^{-1}\,\bar{g}(\xi,\bar\nu)\,R\,\mathrm{d}\bar\mu.
$$
For ease of notation, we will assume that $\xi=(0,\,0,\,\xi^3)$ and $\xi^3>0$. Consider the subsets $$S_+=\{x\in S_{\xi,\lambda}:\bar g(x, e_3)\leq 0\}, \quad S_-=\{x\in S_{\xi,\lambda}:\bar{g}(\xi,\bar \nu(x))\geq 0\}, \quad -S_+=\{-x:x\in S_+\}.$$ 
\begin{figure}\centering
	\includegraphics[width=0.4\linewidth]{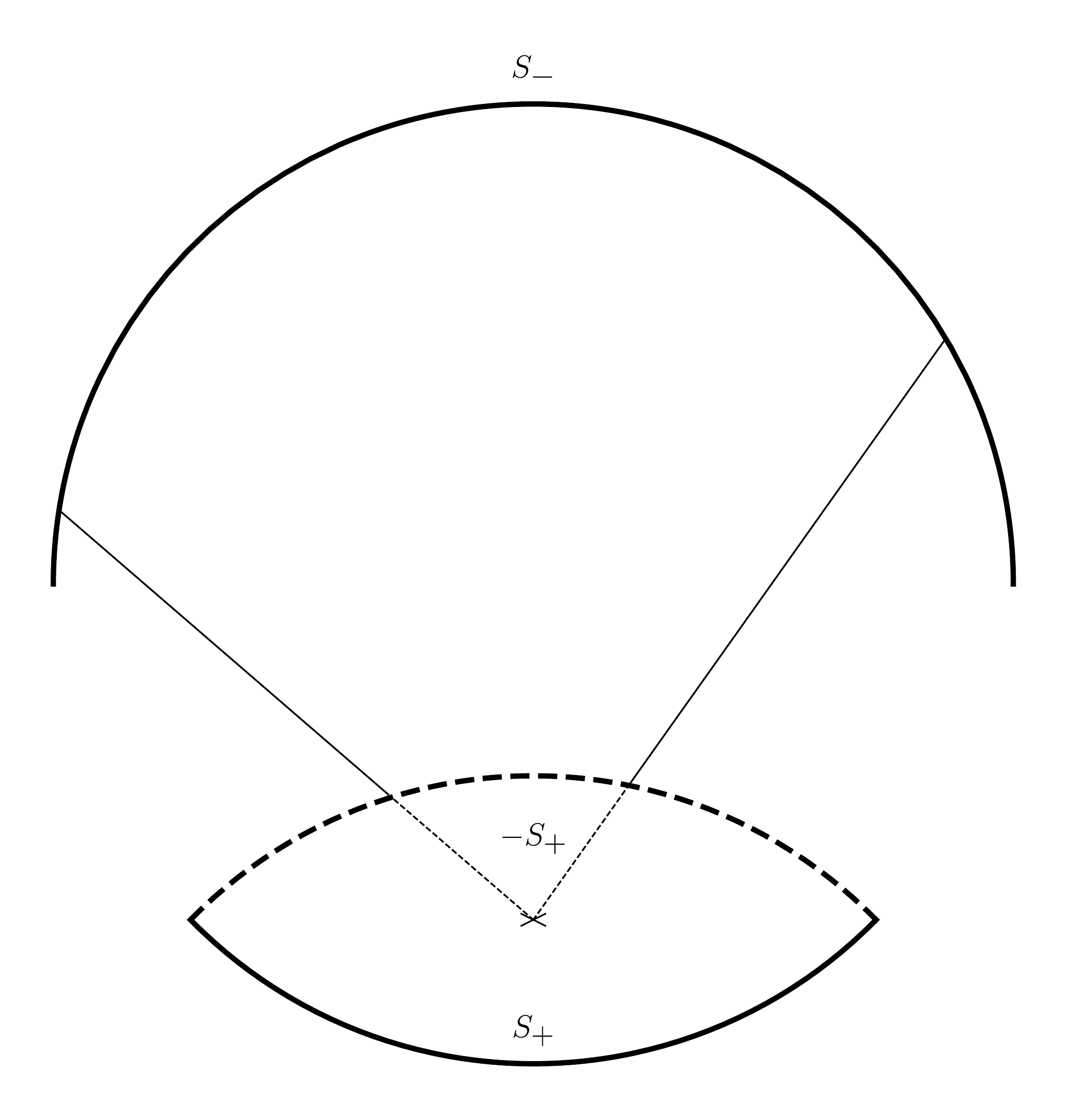}
	\caption{An illustration of the proof of Lemma \ref{G der sRT}. The scalar curvature is compared along the solid black lines connecting $S_-$ and $-S_+$. The number $t$ is given by the sum of one and the quotient of the solid black line and the dashed black line.  The cross marks the origin in the asymptotically flat chart.  }
	\label{figure reflection}
\end{figure}
Using (\ref{sRT2}) and (\ref{sRT3}), we obtain
$$
-2\,\lambda^2\int_{S_{\xi,\lambda}}|\xi|^{-1}\,\bar{g}(\xi,\bar\nu)\,R\,\mathrm{d}\bar\mu \geq 2\,\lambda^2\int_{-S_+}|\xi|^{-1}\,\bar{g}(\xi,\bar\nu)\,R\,\mathrm{d}\bar\mu-2\,\lambda^2\int_{S_-}|\xi|^{-1}\,\bar{g}(\xi,\bar\nu)\,R\,\mathrm{d}\bar\mu-o(1).
$$
We parametrize  almost all of $S_-$ via
$$
\Psi:(0,\pi/2)\times(0,2\pi)\to S_-\qquad \text{ given by } \qquad \Psi(\zeta,\,\varphi)= \lambda\,(\sin\zeta\,\sin\varphi,\,\sin\zeta\,\cos\varphi,\,\cos\zeta+\xi_3).
$$
Likewise, we parametrize almost all of $-S_+$ via
$$
(0,\arccos(\xi^3))\times(0,2\,\pi)\to-S_+, \qquad  (\theta,\,\varphi)\mapsto \lambda\,(\sin\theta\,\sin\varphi,\,\sin\theta\,\cos\varphi,\,\cos\theta-\xi_3).
$$
As shown in Figure \ref{figure reflection}, given $\zeta\in(0,\pi)$, there exists a unique angle $\theta=\theta(\zeta)\in(0,\arccos(\xi_3))$ with $\theta< \zeta$ and a number $t=t(\zeta)>1$ such that
$$
t\,(\sin\theta\,\sin\varphi,\,\sin\theta\,\cos\varphi,\,\cos\theta-\xi_3)=(\sin\zeta\,\sin\varphi,\,\sin\zeta\,\cos\varphi,\,\cos\zeta+\xi_3).
$$
Moreover,
$$
t=\frac{\sin\zeta}{\sin\theta}.
$$
Since $0< \theta< \zeta <\pi/2$, we have $\tan\zeta>\tan\theta$. It follows that $t^{-1}$ is increasing on $(0,\pi/2)$. Consequently, $-\log t$ is also increasing on $(0,\pi/2)$ and thus
\begin{align} \label{change of variables}  
\dot{\theta}\,\sin\theta\,\cos\theta\geq t^{-2}\,\sin\zeta\,\cos\zeta.
\end{align} 
Performing a change of variables and using \eqref{change of variables} and \eqref{sRT3}, we obtain
\begin{align*}
&2\,\lambda^2\int_{-S_+}|\xi|^{-1}\,\bar{g}(\xi,\bar\nu)\,R\,\mathrm{d}\bar\mu-2\,\lambda^2\int_{S_-}|\xi|^{-1}\,\bar{g}(\xi,\bar\nu)\,R\,\mathrm{d}\bar\mu
\\&\qquad\geq2\,\lambda^4 \int_0^{2\pi}\int_0^{\pi/2} \left[t^{-2}\,R(t^{-1}\, \Psi(\zeta,\,\varphi))-R( \Psi(\zeta,\,\varphi))\right] \sin\zeta\,\cos\zeta\,\mathrm{d}\zeta\, \mathrm{d}\varphi-o(1).
\end{align*}
By (\ref{sRT}),
$$
\lambda^4\,[t^{-2}\,R(t^{-1}\, \Psi(\zeta,\,\varphi))-R( \Psi(\zeta,\,\varphi))]\geq-o(1).
$$
In particular,
$$
\sum_{i=1}^3|\xi|^{-1}\,\xi^i\,(\partial_i G_{\lambda,2})(\xi)\geq -o(1).
$$
Conversely, it is elementary to check that the function $G_{1}$ defined in (\ref{G1})
is strictly increasing in radial directions on $\{\xi\in\mathbb{R}^3:0<|\xi|<1\}$. Given $\delta_1\in(0,1-\delta)$, it follows that  $G_\lambda$ is strictly increasing in radial directions on $ \{\xi\in\mathbb{R}^3:\delta_1<|\xi|<1-\delta\}$, provided $\lambda>1$ is sufficiently large.\\ \indent
 It remains to show that $G_\lambda$ is strictly convex near the origin. Again, it is elementary to check that $G_{1}$
is strictly convex on $\{\xi\in\mathbb{R}^3:|\xi|<1\}$. Moreover, given $a\in\mathbb{R}^3$  with $|a|=1$, we compute
\begin{align*}
(\bar D^2G_{\lambda,2})|_\xi(a,\,a)=-2\,\lambda^2 \int_{S_{\xi,\lambda}}\bigg( \lambda\, \bar{g}(a,\bar \nu)^2\,\bar D_{\bar\nu} R +3\,\bar g(a,\bar\nu)^2\, R-R\bigg)\,\mathrm{d}\bar\mu.
\end{align*}
If $\xi=0$, the growth condition (\ref{sRT})  implies that $$-\lambda \,\bar D_{\bar\nu} R\geq 2\,R-o(\lambda^{-4}).$$ Combined  with  (\ref{sRT3}), this gives  
$$
(\bar D^2G_{\lambda,2})|_{(0,\,0,\,0)}\geq-o(1)\, \operatorname{Id}.
$$
Using \eqref{s2 asymptotic to Schwarzschild}, we conclude that there are  $c>0$ and $\delta_0\in(0,1-\delta)$, both independent of $\lambda$, such that
$$
(\bar D^2G_{\lambda,2})|_{\xi}\geq -(o(1)+c\,|\xi|) \operatorname{Id}
$$ 
for every $\xi\in\mathbb{R}^3$ with $|\xi|< \delta_0$, provided  $\lambda>1$ is sufficiently large. The assertions of the lemma follow.
\end{proof}
\begin{lem}
	 Suppose that the scalar curvature $R$ of $(\mathbb{R}^3,g)$ satisfies (\ref{sRT}). There is a constant $\lambda_2>\lambda_1$ depending only on $g$ and $\delta\in(0,1/2)$ such that $G_\lambda$ is strictly increasing in radial directions on $\{\xi\in\mathbb{R}^3:1+\delta<|\xi|<1+\delta^{-1}\}$, provided  $\lambda>\lambda_2$. \label{G der sRT 2}
\end{lem}
\begin{proof}
	The argument is almost the same as the proof of Lemma \ref{G der sRT}   except that we do not need to consider the reflection $-S_+$. In particular, the assumption (\ref{sRT2}) is not required. 
\end{proof}
\begin{proof}[Proof of Theorem \ref{existence thm} ]
	First, we show that for every $\lambda>1$ sufficiently large, there exists a stable area-constrained Willmore sphere with area radius $\lambda$.
	\\ \indent To see this, we decompose
		\begin{align*}
	G_\lambda(\xi)=G_{1}(\xi)+2\,\lambda\int_{\mathbb{R}^3\setminus{B_{\lambda}(\lambda\,\xi)}}R\,\mathrm{d}\bar{v}+O(\lambda^{-1})
	\end{align*}
as in (\ref{G decomposition}).	Note that $G_1(0)=0$ and $\lim_{|\xi|\nearrow 1}G_1(\xi)=\infty$. Using (\ref{R decay}), we find
$$
2\,\lambda\int_{\mathbb{R}^3\setminus{B_{\lambda}(0)}}R\,\mathrm{d}\bar{v}=O(1).
$$
Conversely, (\ref{sRT3}) implies, as $\lambda\to\infty$,
$$
2\,\lambda\int_{\mathbb{R}^3\setminus{B_{\lambda}(\lambda\,\xi)}}R\,\mathrm{d}\bar{v}\geq -2\,\lambda \int_{\mathbb{R}^3\setminus{B_{\lambda\,(1-|\xi|)}(0)}}o(|x|^{-4})\,\mathrm{d}\bar{v}\geq -o(|1-|\xi||^{-1}).
$$	
 It follows that there is a number $z\in(1/2,1)$ such that
	$$
	G_\lambda(0)<G_\lambda(\xi)
	$$
	for every $\xi$ with $|\xi|=z$ and every sufficiently large $\lambda>1$. Consequently, there is a local minimum $\xi(\lambda)$ of $G_\lambda$ with $$|\xi(\lambda)|<z. $$  Lemma \ref{variational to 3 dim} shows that $\Sigma(\lambda)=\Sigma_{\xi(\lambda),\lambda}$ is a stable area-constrained Willmore surface. We claim that
	 $$\xi(\lambda)=o(1).$$
	 Otherwise,  there exists  a suitable sequence $\{\lambda_j\}_{j=1}^\infty$ with $\lim_{j\to\infty}\lambda_j=\infty$ such that 
	  \begin{itemize}
	  	\item[$\circ$] $\lim_{j\to\infty}\xi({\lambda_j})= \xi_0$ with $\xi_0\in\{\xi\in\mathbb{R}^3:0<|\xi|<1\}$ and
	  	\item[$\circ$] $(\bar DG_{\lambda_j})|_{\xi(\lambda_j)}=0$.
	  \end{itemize}
  This is incompatible with Lemma \ref{G der sRT}. Lemma \ref{u, lambda expansions} implies that $\kappa=\kappa({\Sigma({\lambda})})$ is decreasing and approaches $0$ as $\lambda\to\infty$. By Proposition \ref{foliation prop}, the surfaces $\{\Sigma({\lambda}):\lambda>\lambda_0\}$ form a smooth foliation. This finishes the proof of Theorem \ref{existence thm}.
	   \end{proof}
   \begin{proof}[Proof of Theorem  \ref{uniqueness thm}]
Suppose, for a contradiction, that there is a sequence of area-constrained Willmore sphere $\{\Sigma_j\}_{j=0}^\infty$ with
	$$
	\liminf_{j\to\infty}|\Sigma_j|=\infty,\quad \limsup_{j\to\infty} \int_{\Sigma_j} |\hcirc|^2\,\mathrm{d}\mu<\epsilon_0, \quad 0<\liminf_{j\to\infty} \frac{\rho({\Sigma_j})}{\lambda({\Sigma_j})}\leq \limsup_{j\to\infty} \frac{\rho({\Sigma_j})}{\lambda({\Sigma_j})}<\infty,
	$$
	and $\Sigma_j\neq \Sigma({\lambda_j})$ where $\lambda_j=\lambda(\Sigma_j)$. For $\epsilon_0>0$ small enough, Lemma 4.2 in \cite{koerber2020area} implies the uniform, scale invariant curvature estimate
$$
||h||_{L^\infty(\Sigma_j)}= O( \lambda_j^{-1})
$$
with corresponding higher-order estimates as well as the estimate
$$
\kappa({\Sigma_j})=O(\lambda_j^{-3}).
$$ Passing to a subsequence, if necessary, the rescaled surfaces $\tilde \Sigma_j=\lambda_j^{-1}\,\Sigma_j$ converge smoothly to a Euclidean Willmore surface $\tilde \Sigma\subset \mathbb{R}^3$ satisfying
$$
\int_{\tilde \Sigma} |\hbarcirc|^2\,\mathrm{d}\bar\mu<\epsilon_0, \quad |\tilde \Sigma|=4\,\pi, \quad \frac{1}{4\,\pi}\int_{\tilde \Sigma} y\,\mathrm{d}\bar\mu=\xi_0
$$
where $|\xi_0|\neq 1$. Now, the gap theorem \cite[Theorem 2.7]{kuwertsmall} for Euclidean Willmore surfaces due to E.~Kuwert and R.~Schätzle implies that $\tilde \Sigma={S}_{1}(\xi_0)$. It follows that $\Sigma_j$ is a perturbation of a coordinate sphere for large $j$. By Proposition \ref{LS proposition}, it is  captured in our Lyapunov-Schmidt reduction in the sense that  $\Sigma_j=\Sigma_{\xi_j,\lambda_j}$  where  $\xi_j$ is a critical point of $G_{\lambda_j}$ and $\lim_{j\to\infty}\xi_j= \xi_0$. If $|\xi_0|<1$, Lemma \ref{G der sRT} implies that $\xi_0=0$. However, since $G_{\lambda_j}$ is strictly convex near the origin, it follows that $\xi_j=\xi({\lambda_j})$, a contradiction. If $|\xi_0|>1$, we use Lemma \ref{G der sRT 2} instead to obtain a contradiction in a similar way.
\end{proof}
\begin{prop}
	Assumptions as in Theorem \ref{existence thm}. There is no sequence $\{\Sigma_j\}_{j=1}^\infty$ of connected, closed area-constrained Willmore surfaces   with
		\label{prop lower hawking mass bound}
		\begin{align*}
	\liminf_{j\to\infty}|\Sigma_j|=\infty,\quad \liminf_{j\to\infty} m_H(\Sigma_j)&>-\infty,\quad  \limsup_{j\to\infty}\operatorname{genus}(\Sigma_j)<\infty \\ 0<\liminf_{j\to\infty} \frac{\rho({\Sigma_j})}{\lambda({\Sigma_j})}&\leq \limsup_{j\to\infty} \frac{\rho({\Sigma_j})}{\lambda({\Sigma_j})}<\infty
	\end{align*} 
that are not part of the foliation from Theorem \ref{existence thm}.
\end{prop}
\begin{proof}
	According to Theorem \ref{uniqueness thm}, it suffices to verify that $\Sigma_j$ is a sphere for every $j$ sufficiently large and that
	\begin{align}
	\limsup_{j\to\infty} \int_{\Sigma_j} |\hcirc|^2\,\mathrm{d}\mu=0.
	\label{toshowcoro}
	\end{align}
	Let $K$ denote the Gauss curvature of $\Sigma_j$. 
	 Using the Gauss equation in the form
	 $$
	4\, K=2\,R-4\,\operatorname{Rc}(\nu,\nu)+H^2-2\,|\hcirc|^2
	 $$
  and the Gauss-Bonnet theorem
  $$
  \int_{\Sigma_j} K\,\mathrm{d}\mu=4\,\pi\,(1-\operatorname{genus}(\Sigma_j)),
  $$
  we find that
  \begin{align}
  \int_{\Sigma_j} H^2\,\mathrm{d}\mu = 16\,\pi\,(1-\operatorname{genus}(\Sigma_j))+2\int_{\Sigma_j} |\hcirc|^2\,\mathrm{d}\mu+2\int_{\Sigma_j}\left(2\,\operatorname{Ric}(\nu,\nu)-R\right)\mathrm{d}\mu.
  	\label{integrated Gauss equation 0}
  \end{align}
  In particular,
	\begin{align}
	m_H(\Sigma_j)=\sqrt{\frac{|\Sigma_j|}{(16\,\pi)^3}}\bigg(16\,\pi\,\operatorname{genus}(\Sigma_j)-2\int_{\Sigma_j} |\hcirc|^2\,\mathrm{d}\mu-2\int_{\Sigma_j}\left(2\,\operatorname{Ric}(\nu,\nu)-R\right)\mathrm{d}\mu\bigg).
	\label{integrated Gauss equation}
	\end{align}
	The decay assumptions on $g$ imply that 
	$$
	-\int_{\Sigma_j}\left(2\,\operatorname{Ric}(\nu,\nu)-R\right) \mathrm{d}\mu =O(\lambda(\Sigma_j)^2\,\rho(\Sigma_j)^{-3})=O(\lambda(\Sigma_j)^{-1}).
	$$ 
	In particular,
	$$
	\int_{\Sigma_j} |\hcirc|^2\,\mathrm{d}\mu=O(1).
	$$
	As in Lemma \ref{perturbations}, we  compute that
	$$
	\hbarcirc=(1+|x|)^{2}\,\hcirc+O(|x|^{-2}\,|h|)+O(|x|^{-3})
	$$
	and consequently
	$$
		\int_{\Sigma_j} |\hbarcirc|^2\,\mathrm{d}\bar \mu=\int_{\Sigma_j} | \hcirc|^2\,\mathrm{d} \mu+O(\lambda(\Sigma_j)^{-2}).
	$$
	Now, using the Gauss equation for the Euclidean surface $\Sigma_j\subset\mathbb{R}^3$ and the Gauss-Bonnet theorem, we find that
	$$
	16\,\pi-\int_{\Sigma_j} \bar H^2\,\mathrm{d}\bar\mu=O(\lambda(\Sigma_j)^{-1}).
	$$
	According to the result \cite[Theorem 1.2]{kuwertbauer} of E.~Kuwert and M.~Bauer, for any fixed genus, there exists an embedded surface which attains the infimum of the Euclidean Willmore energy. Since the round spheres are the only compact surfaces with Euclidean Willmore energy equal to $4\,\pi$, it follows that $\operatorname{genus}(\Sigma_j)=0$ for $j$ large. Thus, (\ref{toshowcoro}) follows directly from (\ref{integrated Gauss equation}).
\end{proof} 
\label{section 2} 
\section{Proof of Theorem \ref{far outlying thm}}
In this section, we assume 
 that $g$ is a Riemannian metric on $\mathbb{R}^3$ such that
\begin{align} \label{far decay}
g=(1+|x|^{-1})^4\,\bar{g} +\sigma
\end{align} 
where $\sigma$ is a symmetric, covariant two-tensor with, as $x\to\infty$ for every multi-index $J$ with $|J|\leq 5$,
$$\partial_J \sigma=O(|x|^{-2-|J|}).$$
\indent 
Let $|\xi|>2$ and $\lambda>\lambda_0$ for some $\lambda_0>1$ large. As in Section \ref{section 2}, 
given $l\in\{0,\,1,\,2,\dots\}$, we use $\Lambda_l$ to denote the space of the $l$-th spherical harmonics on $S_{\xi,\lambda}=S_\lambda(\lambda\,\xi)$. Likewise, we define $\Lambda_{>2}$, $\Lambda_{>1}$, and $\Lambda_{>0}$ to be the orthogonal complements of $\Lambda_0\oplus\Lambda_1\oplus \Lambda_2$, $\Lambda_0\oplus\Lambda_1$, and $\Lambda_0$ in  $C^{4,\alpha}(S_{\xi,\lambda})$, respectively. We suppress the dependence on $\xi$ and $\lambda$ to relax the notation. \\ \indent  We recall the definition of the rescaled metric $g_{\xi,\lambda}$ in (\ref{pull back metric}). Note that
 $$||g_{\xi,\lambda}-\bar g||_\mathcal{G}=O(\lambda^{-1}\,|\xi|^{-1}).$$  Consequently, Lemma \ref{inverse function theorem lemma} leads to the following proposition.
\begin{prop}
	There are constants $\lambda_0>1$, $c>1$, and $\epsilon>0$ depending on $g$ such that for every $|\xi|>2$ and $\lambda>\lambda_0$
	there exist $u_{\xi,\lambda}\in C^{\infty}(S_{\xi,\lambda})$ and $\kappa_{\xi,\lambda}\in\mathbb{R}$ such that the following hold. The surface $\Sigma_{\xi,\lambda}=\Sigma_{\xi,\lambda}(u_{\xi,\lambda})$ has the  properties
	\begin{itemize}
		\item[$\circ$] $W(\Sigma_{\xi,\lambda})+\kappa_{\xi,\lambda}\,H(\Sigma_{\xi,\lambda})\in \Lambda_1$,
		\item[$\circ$] $|\Sigma_{\xi,\lambda}|=4\,\pi\, \lambda^2$.
	\end{itemize} There holds $u_{\xi,\lambda}\perp \Lambda_1 $ and  	\label{LS proposition far outlying}
	\begin{equation}
	\begin{aligned}
	|u_{\xi,\lambda}|+\lambda\,|\bar \nabla u_{\xi,\lambda}|+\lambda^2\,|\bar \nabla^2 u_{\xi,\lambda}|+\lambda^3\,|\bar\nabla^3 u_{\xi,\lambda}|+\lambda^4\,|\bar\nabla^4 u_{\xi,\lambda}|&< c\,|\xi|^{-1},
	\\\lambda^3\,|\kappa_{\xi,\lambda}|&<c\,|\xi|^{-1}.
	\end{aligned}
		\label{far-outlying initial estimate}
\end{equation}
Moreover, if $\kappa\in\mathbb{R}$ and  $\Sigma_{\xi,\lambda}(u)$ with $u\perp\Lambda_1(S_{\xi,\lambda})$ are such that
\begin{itemize}
	\item[$\circ$] $\Delta H+(|\hcirc|^2+\operatorname{Ric}(\nu,\nu)+\kappa)\,H\in \Lambda_1,$
	\item[$\circ$] $|\Sigma_{\xi,\lambda}(u)|=4\,\pi\, \lambda^2,$
\end{itemize}
and
\begin{align*}
|u|+\lambda\,|\bar\nabla u|+\lambda^2\,|\bar\nabla^2 u|+\lambda^3\,|\bar\nabla^3 u|+\lambda^4\,|\bar\nabla^4 u|&<\epsilon\,\lambda,\\
\lambda^3\,|\kappa|&<\epsilon\,\lambda,
\end{align*}
then $u=u_{\xi,\lambda}$ and $\kappa=\kappa_{\xi,\lambda}$.
\end{prop}
As in the previous section, we  abbreviate $u=u_{\xi,\lambda}$ and $\kappa=\kappa_{\xi,\lambda}$. Lemma \ref{u, lambda expansions} suggests that a stronger estimate for $u$ than (\ref{far-outlying initial estimate}) might hold. However, more care is required since error terms of order $O(\lambda^{-1})$ may be  larger than any inverse power of $|\xi|$.
\\ \indent 
First, we introduce some notation. We define
$$
\phi(x)=1+|x|^{-1}
$$
to be the conformal factor of the Schwarzschild metric with mass $m=2$. As in \cite{chodosh2019far}, we use a bar underneath a quantity to indicate evaluation at $\lambda\,\xi$. If the quantity includes derivatives, these are taken first before we evaluate. For instance, we have
\label{underline page}
$$
\ubar{\sigma}=\sigma(\lambda\,\xi), \qquad \bar D\ubar{\sigma}=(\bar D\sigma)(\lambda\,\xi).
$$
We note that the decay assumptions and Taylor's theorem imply for example that
$$
\ubar{\sigma}(\lambda\,\bar\nu+\lambda\,\xi)=\ubar{\sigma} +\lambda\, D_{\bar\nu}\ubar{\sigma}+O(\lambda^{-2}\,|\xi|^{-4}).
$$
\indent  We now adapt Lemma \ref{u, lambda expansions} to the current setting. In the statement of the following lemma, we let 
$$
Y^{ij}_2=\frac{1}{2}\,\bigg[3\,\bar g(\bar \nu,e_i)\,g(\bar \nu,e_j)-\delta_{ij}\bigg]\in\Lambda_2
$$ where $i,\,j=1,\,2,\,3$ and $\{e_1,\,e_2,\,e_3\}$ denotes the standard basis of $\mathbb{R}^3$.
\begin{lem} \label{far outlying estimates}
	There holds
\begin{align*}
	\kappa&=O(\lambda^{-4}\,|\xi|^{-4}),\\
	W(\Sigma_{\xi,\lambda})+\kappa\, H(\Sigma_{\xi,\lambda})&={O}(\lambda^{-5}\,|\xi|^{-3}),
	\\
	u&=-2\,|\xi|^{-1}+O(\lambda^{-1}\,|\xi|^{-2})+O(|\xi|^{-3}).
\end{align*}
More precisely, 
\begin{align*}
\operatorname{proj}_{\Lambda_2} u&=-\frac13\,\sum_{i,\,j=1}^3\bigg[4\,|\xi|^{-5}\,\xi^i\,\xi^j+\lambda\,\ubar{\phi}^{-6}\,\ubar{\sigma}(e_i,e_j)\bigg]Y_2^{ij}+O(|\xi|^{-4}), \\
\operatorname{proj}_{\Lambda_{>2}}u&=O(|\xi|^{-4})+O(\lambda^{-1}\,|\xi|^{-3}).
\end{align*}
These identities may be differentiated once with respect to $\xi$.
\end{lem}
\begin{proof}
	From Corollary \ref{Willmore operator on sphere spherical harmonics} we obtain
	\begin{equation}
	\begin{aligned}
	\operatorname{proj}_{\Lambda_0} W({S_{\xi,\lambda}})=&\,O(\lambda^{-5}\,|\xi|^{-4}),\\
	\operatorname{proj}_{\Lambda_1} W({S_{\xi,\lambda}})=&\,O(\lambda^{-5}\,|\xi|^{-3}),\\
	\operatorname{proj}_{\Lambda_2} W({S_{\xi,\lambda}})=&\,O(\lambda^{-4}\,|\xi|^{-3})+O(\lambda^{-5}\,|\xi|^{-2}),\\
	\operatorname{proj}_{\Lambda_{>2}} W({S_{\xi,\lambda}})=&\,O(\lambda^{-4}\,|\xi|^{-4})+O(\lambda^{-5}\,|\xi|^{-3}).
	\end{aligned}
	\label{Willmore operator estimates}
	\end{equation}
 We consider the family of surfaces $$\{\Phi^{t\,u}_{\xi,\lambda}(S_{\xi,\lambda}):t\in[0,1]\}$$ 
 where $\Phi^{t\,u}_{\xi,\lambda}:S_{\xi,\lambda}\to M$ is as in (\ref{graph map}). The initial velocity of this variation with respect to the metric $g$ is given by
	\begin{align*}
	w=u\,g(\bar\nu,\nu).
	\end{align*}
Note that
	\begin{align}
	g(\bar\nu,\nu)=\phi^2+O(\lambda^{-2}\,|\xi|^{-2}). \label{g bar nu nu}
	\end{align}
By (\ref{far-outlying initial estimate}) and Taylor's theorem,
	\begin{align}
	Q_{S_{\xi,\lambda}}w=W({S_{\xi,\lambda}})-W({\Sigma_{\xi,\lambda}})+O(\lambda^{-5}\,|\xi|^{-2}).
	\label{initial linearization}
	\end{align}
	 Using (\ref{Q formula}), we find that
\begin{align*}
Q_{S_{\xi,\lambda}}w&=\Delta^2_{S_{\xi,\lambda}}w+2\,\lambda^{-2}\,\phi^{-4}\,\Delta_{S_{\xi,\lambda}}w+O(\lambda^{-5}\,|\xi|^{-2})\\&=\phi^{-6}\,(\bar\Delta^2_{S_{\xi,\lambda}}u+2\,\lambda^{-2}\,\bar\Delta_{S_{\xi,\lambda}}u)+O(\lambda^{-5}\,|\xi|^{-2}).
\end{align*}
By Proposition \ref{LS proposition far outlying}, we have
$$
W({\Sigma_{\xi,\lambda}})+\kappa\, H({\Sigma_{\xi,\lambda}})=Y_1\in\Lambda_1.
$$ Moreover, (\ref{mean curvature change}), Proposition \ref{LS proposition far outlying}, Lemma \ref{Schwarzschild identities}, and Lemma \ref{perturbations} imply that
$$
\operatorname{proj}_{\Lambda_0}H({\Sigma_{\xi,\lambda}})=2\,\lambda^{-1}+O(\lambda^{-2}\,|\xi|^{-1}).
$$
Thus, projecting  (\ref{initial linearization}) onto $\Lambda_0$,  we  find 
$$
\kappa=O(\lambda^{-4}\,|\xi|^{-2}).
$$
Next, we observe that 
$$
\operatorname{proj}_{\Lambda_{>0}}H({\Sigma_{\xi,\lambda}})=O(\lambda^{-2}\,|\xi|^{-1}).
$$
Projecting (\ref{initial linearization}) onto $\Lambda_1$, we conclude that
$$
Y_1=O(\lambda^{-5}\,|\xi|^{-2}).
$$
Similarly, we obtain
\begin{align}
\operatorname{proj}_{\Lambda_{>1}}u=O(\lambda^{-1}\,|\xi|^{-2})+O(|\xi|^{-3}).
\label{slightly improved estimates}
\end{align}
Finally,  Lemma \ref{area expansion},  (\ref{far-outlying initial estimate}), and the formula for the first variation of area  imply
$$
\int_{S_{\xi,\lambda}} H({S_{\xi,\lambda}})\,w\,\mathrm{d}\mu =-16\,\pi\, \lambda\,|\xi|^{-1}+O(|\xi|^{-2}).
$$
Using the improved estimates (\ref{slightly improved estimates}) for $\operatorname{proj}_{\Lambda_{>1}}u$ and arguing as in the proof of Lemma \ref{Lambda0 estimate}, we obtain
$$
\operatorname{proj}_{\Lambda_0}u=-2\,|\xi|^{-1}+O(\lambda^{-1}\,|\xi|^{-2}).
$$
To remove this scaling effect of the perturbation, we define
\begin{align}\label{tilde lambda def}\tilde u=u+2\,|\xi|^{-1}, \qquad \tilde\lambda=\lambda-2\,|\xi|^{-1} \qquad\text{and}\qquad  \tilde  \xi=\lambda\,(\lambda-2\,|\xi|^{-1})^{-1}\,\xi\end{align}
and note that $\lambda\,\xi=\tilde\lambda\,\tilde \xi$. Then, $\Sigma_{\xi,\lambda}=\Phi^{\tilde u}_{\tilde \xi,\tilde \lambda}(S_{\tilde\xi,\tilde \lambda})$ and there holds
\begin{align}
\tilde{u}=O(\lambda^{-1}\,|\xi|^{-2})+O(|\xi|^{-3}). \label{first tilde u estimate}
\end{align}
We now repeat the above argument to obtain an improved estimate for $\tilde u$. As before, we consider the family of surfaces $$\{\Phi^{t\,\tilde u}_{\tilde\xi,\tilde\lambda}(S_{\tilde\xi,\tilde\lambda}):t\in[0,1]\}.$$ The initial velocity of this family with respect to the metric $g$ is given by
$$
\tilde w=\tilde u\,g(\bar\nu,\nu).
$$
Note that (\ref{first tilde u estimate}) implies the improved estimate
\begin{align}
Q_{S_{\tilde \xi,\tilde \lambda}}\tilde w=W({S_{\tilde \xi,\tilde \lambda}})-W({\Sigma_{\xi,\lambda}})+O(\lambda^{-6}\,|\xi|^{-4})+O(\lambda^{-5}\,|\xi|^{-6}). \label{second velocity}
\end{align}
Revisiting equation (\ref{Q formula}) and recalling (\ref{g bar nu nu}), we find
\begin{align*}
Q_{S_{\tilde \xi,\tilde \lambda}}\tilde w&=\Delta^2_{S_{\tilde \xi,\tilde \lambda}}\tilde w+2\,\tilde \lambda^{-2}\,\phi^{-4}\Delta_{S_{\tilde \xi,\tilde \lambda}}\tilde w+O(\lambda^{-6}\,|\xi|^{-4})+O(\lambda^{-5}\,|\xi|^{-6})
\\&=\phi^{-6}\,(\bar\Delta^2_{S_{\tilde \xi,\tilde \lambda}}\tilde u+2\,\tilde \lambda^{-2}\,\bar\Delta_{S_{\tilde \xi,\tilde \lambda}}\tilde u)+O(\lambda^{-6}\,|\xi|^{-4})+O(\lambda^{-5}\,|\xi|^{-6}).
\end{align*}
Arguing as before, this improved estimate yields
$$
\kappa=O(\lambda^{-4}\,|\xi|^{-4}),\qquad Y_1=O(\lambda^{-5}\,|\xi|^{-3}) \qquad\text{and}\qquad \operatorname{proj}_{\Lambda>2}\tilde u=O(|\xi|^{-4})+O(\lambda^{-1}\,|\xi|^{-3}).
$$
Finally, we project (\ref{second velocity}) onto $\Lambda_2$. Recalling Corollary \ref{Willmore operator on sphere spherical harmonics}, we find
\begin{equation}
\begin{aligned}
&\bar\Delta^2_{S_{\tilde \xi,\tilde \lambda}}\operatorname{proj}_{\Lambda_2}\tilde u\,+2\,\tilde \lambda^{-2}\,\bar\Delta_{S_{\tilde \xi,\tilde \lambda}}\operatorname{proj}_{\Lambda_2} \tilde u\\&\qquad=\,\operatorname{proj}_{\Lambda_2}(\phi^6\,W({S_{\tilde \xi,\tilde\lambda}}))+O(\lambda^{-6}\,|\xi|^{-4})+O(\lambda^{-5}\,|\xi|^{-6})
\\&\qquad=\,-32\,\tilde \lambda^{-4}\,|\xi|^{-3}\,P_2\left(-|\xi|^{-1}\,\bar g(\bar\nu,\xi)\right)-4\,\tilde \lambda^{-3}\,\ubar{\phi}^{-4}\,
(3\,\ubar{\sigma}(\bar\nu,\bar\nu)-\bar {\operatorname{tr}}\,\ubar\sigma)+O(\lambda^{-4}\,|\xi|^{-4})
\end{aligned}
\label{Lambda_2 equation}
\end{equation}
where $P_2$ is the second Legendre polynomial defined by \eqref{legendre generating2}.
To conclude, we use Corollary \ref{Willmore eigenvalues} and the estimate \begin{align}\tilde\lambda=\ubar{\phi}^{-2}\,\lambda+O(\lambda^{-1}\,|\xi|^{-2}),\label{tilde lambda est}\end{align}
which is immediate from the definition (\ref{tilde lambda def}).
\end{proof}
We recall from (\ref{G definition}) that the function $ G_\lambda:\{\xi\in\mathbb{R}^3:|\xi|>2\}\to\mathbb{R}$  is given by $$  G_\lambda(\xi)= \lambda^2\,\bigg(\int_{\Sigma_{\xi,\lambda}} H^2\,\mathrm{d}\mu-16\,\pi\bigg).
$$
Using the improved estimates for $u$ obtained in Lemma \ref{far outlying estimates}, we now show that the expansion in Lemma \ref{G expansion} holds with better error control. This is the key step in the proof of Theorem \ref{far outlying thm}.
\\ \indent In the proof of Lemma \ref{far-outlying expansion}, we use the notation
$$
\sigmacirc=\sigma-\frac13\,(\bar {\operatorname{tr}}\,\sigma)\,\bar g
$$
\begin{lem}
	There holds
	\begin{align*}
	 G_\lambda(\xi)= -\frac{128\,\pi}{15}\,|\xi|^{-6}
	-2\,\lambda \int_{B_{\lambda}(\lambda\,\xi)} R\,\mathrm{d}\bar{v}+O(\lambda^{-1}\,|\xi|^{-6})+O(|\xi|^{-7}).
	\end{align*}
	This identity may  	\label{far-outlying expansion} be differentiated once with respect to $\xi$.
\end{lem}
\begin{proof}
We recall from the proof of Lemma \ref{far outlying estimates} that $\Sigma_{\xi,\lambda}=\Phi^{\tilde u}_{\tilde \xi,\tilde \lambda}(S_{\tilde\xi,\tilde \lambda})$, where
$$\tilde u=u+2\,|\xi|^{-1}, \qquad \tilde\lambda=\lambda-2\,|\xi|^{-1} \qquad\text{and} \qquad \tilde \xi=\lambda\, (\lambda-2\,|\xi|^{-1})^{-1}\,\xi.$$
As in the proof of Lemma \ref{far outlying estimates}, we consider the family of surfaces $$\{\Phi^{t\,\tilde u}_{\tilde \xi,\tilde \lambda}(S_{\tilde\xi,\tilde \lambda}):t\in[0,1]\}$$ connecting $S_{\tilde \xi,\tilde \lambda}$ and $\Sigma_{\xi,\lambda}$. Recall the functional $F_\lambda$ defined in (\ref{F functional}). We compute the Taylor expansion of the function
$$
[0,1]\to\mathbb{R}^3, \qquad t\mapsto F_\lambda(\Phi^{t\,\tilde u}_{\tilde \xi,\tilde \lambda}(S_{\tilde\xi,\tilde \lambda}))
$$ at $t=0$. To this end, we abbreviate $W=W(S_{\tilde \xi,\tilde\lambda})$ and $Q=Q_{S_{\tilde\xi,\tilde\lambda}}$. Since the initial velocity is given by $\phi^2\,\tilde u$, we find, using Lemma \ref{variation lemma}, Lemma \ref{far outlying estimates}, as well as (\ref{Willmore operator estimates}), that
	\begin{align*}
	F_\lambda(\Sigma_{\xi,\lambda})&=F_\lambda(S_{\tilde\xi,\tilde \lambda})-2\,\lambda^2\int_{S_{\tilde\xi,\tilde \lambda}}\phi^2\,W\,\tilde u\,\mathrm{d}\mu
	+\lambda^2\int_{S_{\tilde \xi,\tilde \lambda}}\phi^4\,\left[\tilde u\,Q\tilde u-W\,H\,\tilde u^2\right]\mathrm{d}\mu+O(\lambda^{-1}\,|\xi|^{-6}) 
	\\&=F_\lambda(S_{\tilde\xi,\tilde \lambda})-2\,\lambda^2\int_{S_{\tilde\xi,\tilde \lambda}}\phi^2\,W\,\tilde u \,\mathrm{d}\mu
	+\lambda^2\int_{S_{\tilde\xi,\tilde \lambda}}\phi^4\,\tilde u\,Q\tilde u \,\mathrm{d}\mu+O(\lambda^{-1}\,|\xi|^{-6})
	\\&=F_\lambda(S_{\tilde\xi,\tilde \lambda})-2\,\lambda^2\int_{S_{\tilde\xi,\tilde \lambda}}\operatorname{proj}_{\Lambda_2}(\phi^6\,W)\,\operatorname{proj}_{\Lambda_2}\tilde u\,\mathrm{d}\bar\mu
	+\lambda^2\int_{S_{\tilde \xi,\tilde \lambda}}\phi^4\,\tilde u\,Q\tilde u\,\mathrm{d}\mu+O(\lambda^{-1}\,|\xi|^{-6}).
	\end{align*}
Revisiting equation (\ref{Q formula})   again and using (\ref{Lambda_2 equation}), we find
	\begin{align*}
	\lambda^2\int_{S_{\tilde\xi,\tilde \lambda}}\phi^4\,\tilde u\,Q\tilde u\,\mathrm{d}\mu&=\lambda^2\int_{S_{\tilde\xi,\tilde \lambda}} \left[(\bar\Delta\tilde u)^2+2\,\tilde \lambda^{-2}\,\tilde u\,\bar\Delta \tilde u\right]\mathrm{d}\bar\mu+O(\lambda^{-1}\,|\xi|^{-6})
	\\&=\lambda^2\int_{S_{\tilde\xi,\tilde \lambda}} \operatorname{proj}_{\Lambda_2}( \phi^6\, W)\,\operatorname{proj}_{\Lambda_2}\tilde u\,\mathrm{d}\bar\mu+O(\lambda^{-1}\,|\xi|^{-6})
	\\&=24\,\lambda^{2}\,\tilde\lambda^{-4}\int_{S_{\tilde\xi,\tilde \lambda}} (\operatorname{proj}_{\Lambda_2}\tilde u)^2\,\mathrm{d}\bar\mu+O(\lambda^{-1}\,|\xi|^{-6}).
	\end{align*}
	It follows that
	\begin{align}
	F_\lambda(\Sigma_{\xi,\lambda})=F_\lambda(S_{\tilde\xi,\tilde \lambda})-24\,\lambda^{2}\,\tilde\lambda^{-4}\int_{S_{\tilde\xi,\tilde \lambda}} (\operatorname{proj}_{\Lambda_2}\tilde u)^2\,\mathrm{d}\bar\mu+O(\lambda^{-1}|\xi|^{-6}).
	\label{1}
	\end{align}
	For ease of notation, we  assume that $\xi$ is a multiple of $e_3$. Using Lemma \ref{far outlying estimates} and Lemma \ref{spherical identities}, we compute 
	\begin{align}
\notag	&-24\,\lambda^{2}\,\tilde\lambda^{-4}\,\int_{S_{\tilde\xi,\tilde \lambda}} (\operatorname{proj}_{\Lambda_2}\tilde u)^2\,\mathrm{d}\bar\mu
	\\\notag&\qquad=\,-\frac{8}{3}\,\lambda^2\,\tilde\lambda^{-2}\int_{S_1(0)} \bigg(16\,|\xi|^{-6}\,Y^{3,3}_2\,Y^{3,3}_2+8\,\lambda\,|\xi|^{-3}\,\ubar{\phi}^{-6}\,\ubar{\sigma}(e_i,e_j)\,Y^{3,3}_2\,Y^{ij}_2\\\notag&\,\qquad\qquad\qquad\qquad\qquad\qquad+\lambda^{2}\,\ubar{\phi}^{-12}\,\sum_{i,\,j,\,k,\,\ell=1}^3\ubar{\sigma}(e_i,e_j)\, \ubar{\sigma}(e_k,e_l)\,Y_2^{ij}\,Y^{kl}_2\bigg)\,\mathrm{d}\bar\mu
	\\\label{2}&\qquad=\,-\frac{512\,\pi}{15}\,\lambda^2\,\tilde\lambda^{-2}\,|\xi|^{-6}-\frac{128\,\pi}{15}\, \lambda^3\,\tilde \lambda^{-2}\,|\xi|^{-3}\,\left(3\,|\xi|^{-2}\,\ubar{\sigma}(\xi,\xi)-\bar {\operatorname{tr}}\,\ubar{\sigma}\right)-\frac{48\,\pi}{15}\,\lambda^4\,\tilde\lambda^{-2}\,\ubar{\phi}^{-12}\,|\ubar{\sigmacirc}|^2\\\notag&\qquad\qquad\,+O(\lambda^{-1}\,|\xi|^{-6})
	\\\notag&\qquad=\,-\frac{512\,\pi}{15}\,|\xi|^{-6}-\frac{128\,\pi}{15}\,\lambda\, |\xi|^{-3}\,\left(3\,|\xi|^{-2}\,\ubar{\sigma}(\xi,\xi)-\bar {\operatorname{tr}}\,\ubar{\sigma}\right)-\frac{48\,\pi}{15}\,\lambda^2\,\ubar{\phi}^{-8}\,|\ubar{\sigmacirc}|^2\\\notag&\qquad\qquad\,+O(\lambda^{-1}\,|\xi|^{-6}).
	\end{align}
		In the last equation, we have used (\ref{tilde lambda est}).
	Conversely, Lemma \ref{Willmore energy sphere} and Taylor's theorem give that
	\begin{align}
	\notag F_\lambda(S_{\tilde\xi,\tilde \lambda})=&\,\lambda^2\,\tilde\lambda^{-2}\,F_{\tilde\lambda}(S_{\tilde\xi,\tilde \lambda})
	\\	\label{3}=&\,\frac{128\,\pi}{5}\,|\xi|^{-6}-2\,\lambda^2\,\tilde\lambda^{-1}\,\ubar{\phi}^{4} \int_{{B_{\tilde \lambda}(\tilde \lambda\,\tilde \xi)}} R\,\mathrm{d}\bar{v}\\\notag&\qquad+\frac{48\,\pi}{15}\,\lambda^{2}\,\ubar{\phi}^{-8}\,|\ubar{\sigmacirc}|^2+\frac{128\,\pi}{15}\,\lambda\,|\xi|^{-3}\,\left(3\,|\xi|^{-2}\,\ubar{\sigma}(\xi,\xi)-\bar {\operatorname{tr}}\,\ubar{\sigma}\right)+O(\lambda^{-1}\,|\xi|^{-6})+O(|\xi|^{-7}).
	\end{align}
	Finally, using \eqref{far decay} and (\ref{tilde lambda est}), we find
	\begin{equation} \label{4}
	\begin{aligned}
	\lambda^2\,\tilde\lambda^{-1}\,\ubar{\phi}^{4} \int_{{B_{\tilde \lambda}(\tilde \lambda\,\tilde \xi)}}R\,\mathrm{d}\bar{v}=&\,	\lambda\,\ubar{\phi}^{6} \int_{{B_{\ubar{\phi}^{-2}\,\lambda}( \lambda\, \xi)}}R\,\mathrm{d}\bar{v}+O(\lambda^{-1}\,|\xi|^{-6})
		\\=&\,	\lambda \int_{B_{\lambda}(\lambda\,\xi)}R\,\mathrm{d}\bar{v}+O(\lambda^{-1}\,|\xi|^{-6}).	\end{aligned}
		\end{equation}
Assembling (\ref{1}), (\ref{2}), (\ref{3}), and (\ref{4}), the assertion follows.
\end{proof}
\begin{proof}[Proof of Theorem \ref{far outlying thm}]
Suppose, for a contradiction, that there exists a sequence of outlying area-constrained Willmore spheres $\{\Sigma_j\}_{j=1}^\infty$ with
$$
\lim_{j\to\infty} |\Sigma_j|=\infty, \qquad \limsup_{j\to\infty}\int_{\Sigma_j}|\hcirc|^2\,\mathrm{d}\mu<\epsilon_0, \qquad \lim_{j\to\infty} \frac{\rho({\Sigma_j})}{\lambda({\Sigma_j})}=\infty.
$$  
As in the proof of Theorem \ref{uniqueness thm}, we may assume that $$\Sigma_j=\Sigma_{\xi_j,\lambda_j}$$ for suitable  $\xi_j\in\mathbb{R}^3$ and $\lambda_j\in(\lambda_0,\infty)$ where
$$
\lim_{j\to\infty}|\xi_j|=\infty \qquad\text{and}\qquad \lim_{j\to\infty}\lambda_j=\infty.
$$
Arguing as in the proof of Lemma \ref{G der sRT 2}, but this time using the exact growth condition
$$
\sum_{i=1}^3x^i\,\partial_i(|x|^2\,R)(x)\leq 0 ,
$$
we find that
$$
\sum_{i=1}^3\xi_j^i\,\partial_i\bigg(-2\,\lambda_j \int_{B_{\lambda}(\lambda_j\,\xi_j)}R\,\mathrm{d}\bar{v}\bigg)\geq 0.
$$ In conjunction with Lemma \ref{far-outlying expansion}, this gives 
$$
\sum_{i=1}^3\xi_j^i\,(\partial_i  G_\lambda)(\xi_j)\geq\frac{256\,\pi}{5}\,|\xi_j|^{-6}+O(\lambda_j^{-1}\,|\xi_j|^{-6})+O(|\xi_j|^{-7}).
$$
In particular, $$\sum_{i=1}^3\xi_j^i\,(\partial_iG_{\lambda_j})(\xi_j)>0.$$  This is incompatible with Lemma \ref{variational to 3 dim}.
\end{proof}
\section{Proof of Theorems \ref{counterexample thm 1}, \ref{counterexample thm 2} and  \ref{counterexample thm 3}}
We first prove Theorem \ref{counterexample thm 1} and Theorem \ref{counterexample thm 3}. To this end, we adapt a construction from \cite[\S3 and \S6]{chodosh2019far}. The metrics in this construction are rotationally symmetric. Their scalar curvature has a pulse. \\\indent
We briefly recall some steps from \cite{chodosh2019far}. \\ \indent  Given a function $S:(0,\infty)\to(-\infty,0]$ with
$$
S^{(l)}=O(s^{-4-l})
$$  
for every integer $l\geq0$, we define the function $\Psi:(0,\infty)\to(-\infty,0]$ by
\begin{align}
\Psi(s)=s^{-1}\int_s^\infty (t-s)\,t\, S(t)\,\mathrm{d} t.
\label{metric construction 1}
\end{align}
Note that 
$$
\Psi^{(l)}=O(s^{-2-l})
$$
for every integer $l\geq0$. The metric 
$$
g=(1+|x|^{-1}+\Psi(|x|))^4\,\bar g
$$
on $\mathbb{R}^3\setminus\{0\}$ is $C^k$-asymptotic to Schwarzschild with mass $m=2$ for every $k\geq 2$. Its scalar curvature $R$ is given by
\begin{align}
R(x)=-8\,(1+O(|x|^{-1}))\,S(|x|). \label{metric construction 2}
\end{align}
In particular, $R\geq 0$ outside  a compact set.
\begin{proof}[Proof of Theorem \ref{counterexample thm 1}] First, as shown in Figure \ref{fig:test}, we construct a metric $g_2$ which admits large outlying area-constrained Willmore spheres $\{\Sigma_j\}_{j=1}^\infty$ with
$$
2\sqrt{2}<\frac{\rho({\Sigma_j})}{\lambda({\Sigma_j})}<5 \qquad \text{and}\qquad m_H(\Sigma_j)>-o(1).
$$ Let $\chi\in C^{\infty}(\mathbb{R})$ be such that $\chi(t)>0$ for all $t\in(3,4)$ and $\operatorname{supp}\chi\subset[3,4]$. Let
$$
S(s)=-B\sum_{k=0}^\infty  10^{-4\,k}\,\chi(10^{-k}\,s).
$$
The constant $B>0$ will be chosen (large) later.  Let $j\geq1$ be a  large integer and $\lambda_j=10^j$.
Recall the definition of $G_{\lambda}$ in (\ref{G definition}).
$ G_{\lambda_j}$ is rotationally symmetric on $\{\xi\in\mathbb{R}^3:|\xi|>2\}$. As  in the proof of Lemma \ref{G der sRT}, we have
$$
G_{\lambda_j}=G_1+G_{2,\lambda_j}+o(1).
$$ 
This expansion may be differentiated twice.  Here, $G_1$ is strictly increasing in radial directions and independent of both $\lambda$ and $B$, while 
\begin{align*}
\sum_{i=1}^3\xi^i\,(\partial_iG_{2,\lambda_j})(\xi)=-2\,\lambda_j^2 \int_{S_{\xi,\lambda_j}}\bar g(\xi,\bar\nu) \,R\,\mathrm{d} \bar\mu
=-16\, B \int_{S_{1}(\xi)} \bar g(\xi,\bar\nu) \,\chi\,\mathrm{d} \bar\mu+o(1).
\end{align*}
The integral on the right hand side vanishes if $|\xi|=5$ and is negative if $|\xi|=2\sqrt{2}$. Thus, using that $ G_{\lambda_j}$ is rotationally symmetric and that $G_1$ is strictly increasing in radial directions, we may increase $B>0$ appropriately so $ G_{\lambda_j}$ attains a local minimum at some $\xi_j\in\mathbb{R}^3$ with $$2\,\sqrt{2}<|\xi_j|<5$$ for every sufficiently large $j$.

\begin{figure}
	\centering
	\begin{subfigure}{0.5\textwidth}
		
		\includegraphics[width=1\linewidth]{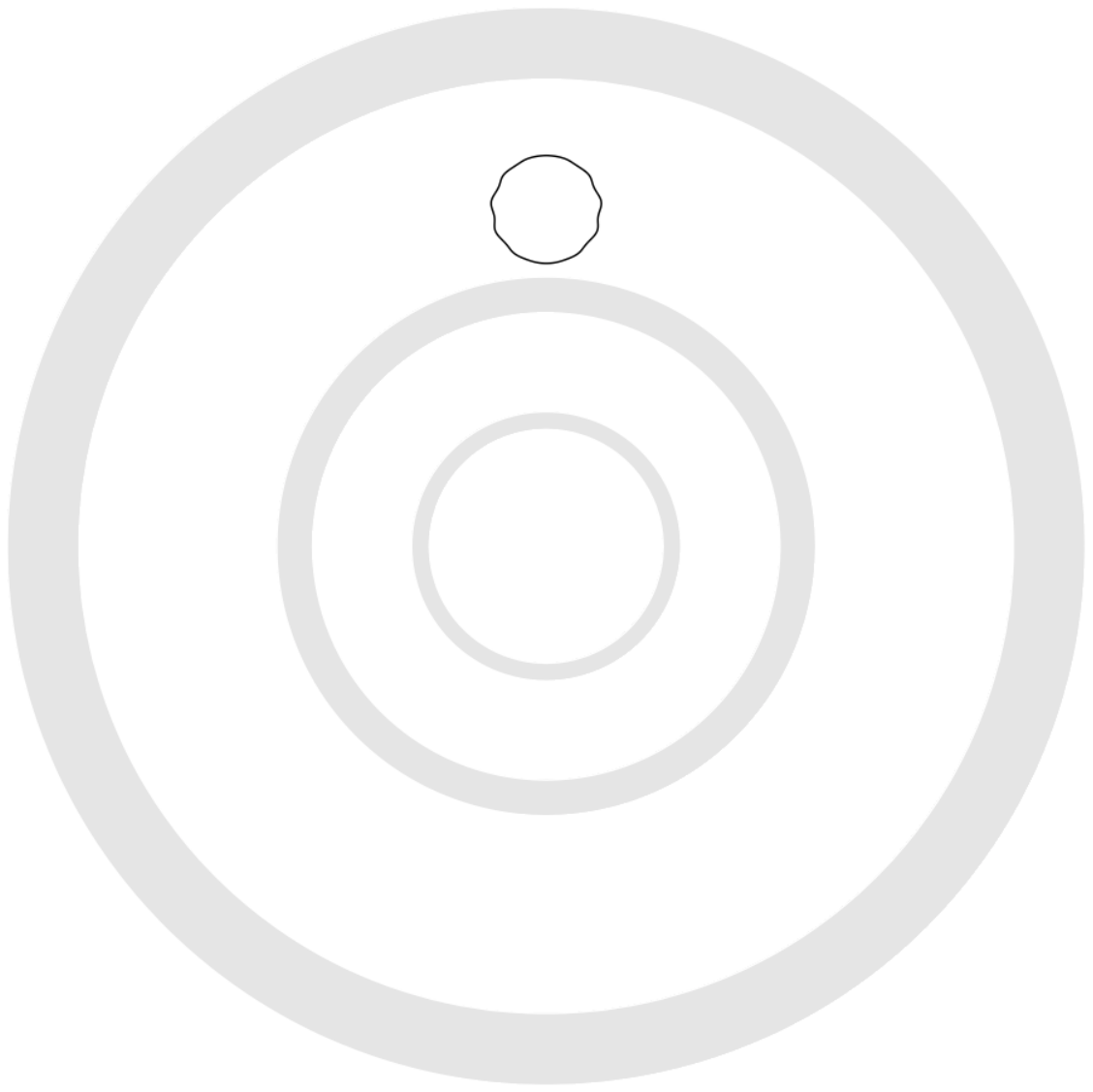}

	\end{subfigure}%
	\begin{subfigure}{0.5\textwidth}

		\includegraphics[width=1\linewidth]{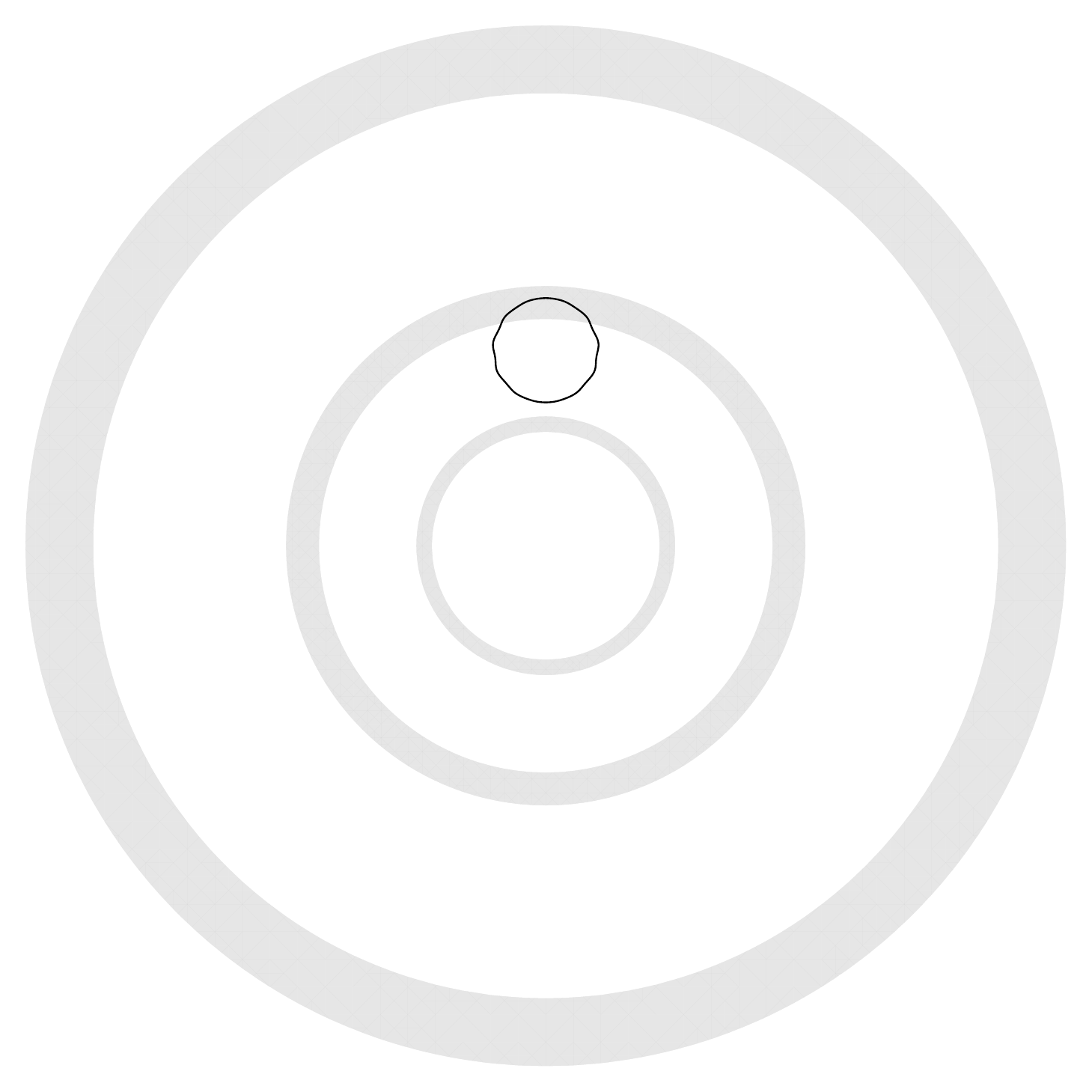}
	\end{subfigure}
\caption{An illustration of the construction for the proof of Theorem \ref{counterexample thm 1}. The scalar curvature is positive in the shaded region and vanishes elsewhere. On the left, the surface $\Sigma_{\xi,\lambda_j}$ with $|\xi|=5$ is shown. It does not overlap with the shaded region. In particular, the radial derivative of $G_{2,\lambda_j}$ vanishes. On the right, the surface corresponds to the choice $|\xi|=2\,\sqrt{2}$.  If this surface is moved upwards, the overlap with the shaded region increases. The radial derivative of $G_{2,\lambda_j}$ is negative.}
	\label{fig:test}
\end{figure}

 \indent
Next, we construct a metric $g_1$ which admits large on-center area-constrained Willmore spheres that are not part of the foliation from Theorem \ref{existence thm}. This time, we choose a smooth function $\chi$ such that $\chi(t)>0$ for all $t\in(9/8,11/8)$ and $\operatorname{supp}\chi\subset[9/8,11/8]$. We write
$$
 G_{\lambda_j}= G_1+ G_{2,\lambda_j}+o(1).
$$
As before, $G_1$ is strictly increasing in radial directions and independent of both $\lambda$ and $B$, while
\begin{align*}
\sum_{i=1}^3\xi^i\,(\partial_iG_{2,\lambda_j})(\xi)=-16\,B \int_{S_{1}(\xi)} \bar g(\xi,\bar\nu) \,\chi\,\mathrm{d} \bar\mu+o(1).
\end{align*}
The integral on the right hand side is negative if $|\xi|=1/4$ and positive if $|\xi|=7/8$. Thus, we may again increase $B$ appropriately such that for every $j$ large, $G_{\lambda_j}$ attains a local minimum $\xi_j\in\mathbb{R}^3$ with $$1/4<|\xi_j|<7/8.$$ 
This completes the proof of Theorem \ref{counterexample thm 1}.
\end{proof}
\begin{proof}[Proof of Theorem \ref{counterexample thm 3}]
 \indent To construct $g_4$, we choose a suitable function $\Psi$ in (\ref{metric construction 1}) which satisfies
 \begin{align}
 \Psi^{(l)}=O(s^{-3-l}).
 \label{rapidly decaying Psi}
 \end{align}
 In particular, $\Psi$ decays one order faster than the perturbations used to construct $g_1$ and $g_2$. Due to the fast decay of the Schwarzschild contribution, this perturbation will still be strong enough to admit large far-outlying area-constrained Willmore spheres with Hawking masses bounded from below. More precisely, we choose $\chi\in C^\infty(\mathbb{R})$ with $\chi(t)>0$ for all $t\in(4,6)$, $\operatorname{supp}\chi\subset[4,6]$, and  $\chi'(5)=1$. Let
$$
S(s)=-\sum_{k=0}^\infty 10^{-5\,k}\,\chi(10^{-k}\,s)
$$
and note that $S^{(l)}=O(s^{-5-l})$ for every integer $l\geq 0$. This ensures that (\ref{rapidly decaying Psi}) holds. Now, let $j\geq 1$ be large, $\lambda_j=10^j$, and $\xi_t=t\,10^j\,  a$ where $t\in[3,7]$ and $a\in\mathbb{R}^3$ is such that $|a|=1$. From (\ref{metric construction 2}), we see that
$$
R=O(|x|^{-5}).
$$
Using Taylor's theorem, we find that
$$
-2\,\lambda \int_{B_{\lambda_j}(\lambda_j\,\xi_t)} R\,\mathrm{d}\bar{v}=-\frac{8\,\pi}{3}\,\lambda_j^4\,\ubar{R}+O(\lambda^{-1}\,|\xi_t|^{-6})+O(|\xi_t|^{-7}).
$$
Lemma \ref{far-outlying expansion} implies that
\begin{align*}
G_{\lambda_j}(\xi_t)&=-\frac{128\,\pi}{15}\,|\xi_t|^{-6}-\frac{8\,\pi}{3}\,\lambda_j^4\,\ubar{R}+O(\lambda^{-1}\,|\xi_t|^{-6})+O(|\xi_t|^{-7})
\\&=-\frac{128\,\pi}{15}\,t^{-6}\,10^{-6j}-\frac{64\,\pi}{3}\,\chi(t)\,10^{-6j}+O(10^{-7j}).
\end{align*}
The derivative of the quantity on the right hand side with respect to $t$ is positive if $t=7$ and negative if $t=5$  provided $j$ is large. Since $G_{\lambda_j}$ is rotationally symmetric, it follows that, for every $j$  large, there is a number $t_j\in[5,7]$ with $(\bar D G_{\lambda_j})(\xi_j)=0$ where $\xi_j=t_j\,10^j\,a$. In particular, $\{\Sigma_{\xi_j,\lambda_j}\}_{j=1}^\infty$ is a  sequence of far-outlying stable area-constrained Willmore spheres with diverging area and Hawking mass bounded from below. 
\end{proof} 
\begin{rema} \normalfont
	The proofs of Theorem \ref{counterexample thm 1} and Theorem \ref{counterexample thm 3} follow the proofs of Theorem 1.3 and 1.8 in \cite{chodosh2019far} closely. Note that the analysis of the function $G_{\lambda}$ differs from the analysis of the reduced area functional  in \cite{chodosh2019far}. The construction of the metric $g_1$ has features not considered in \cite{chodosh2019far}.
\end{rema}
\begin{figure}
	\centering
	\begin{subfigure}{0.5\textwidth}
		
		\includegraphics[width=1\linewidth]{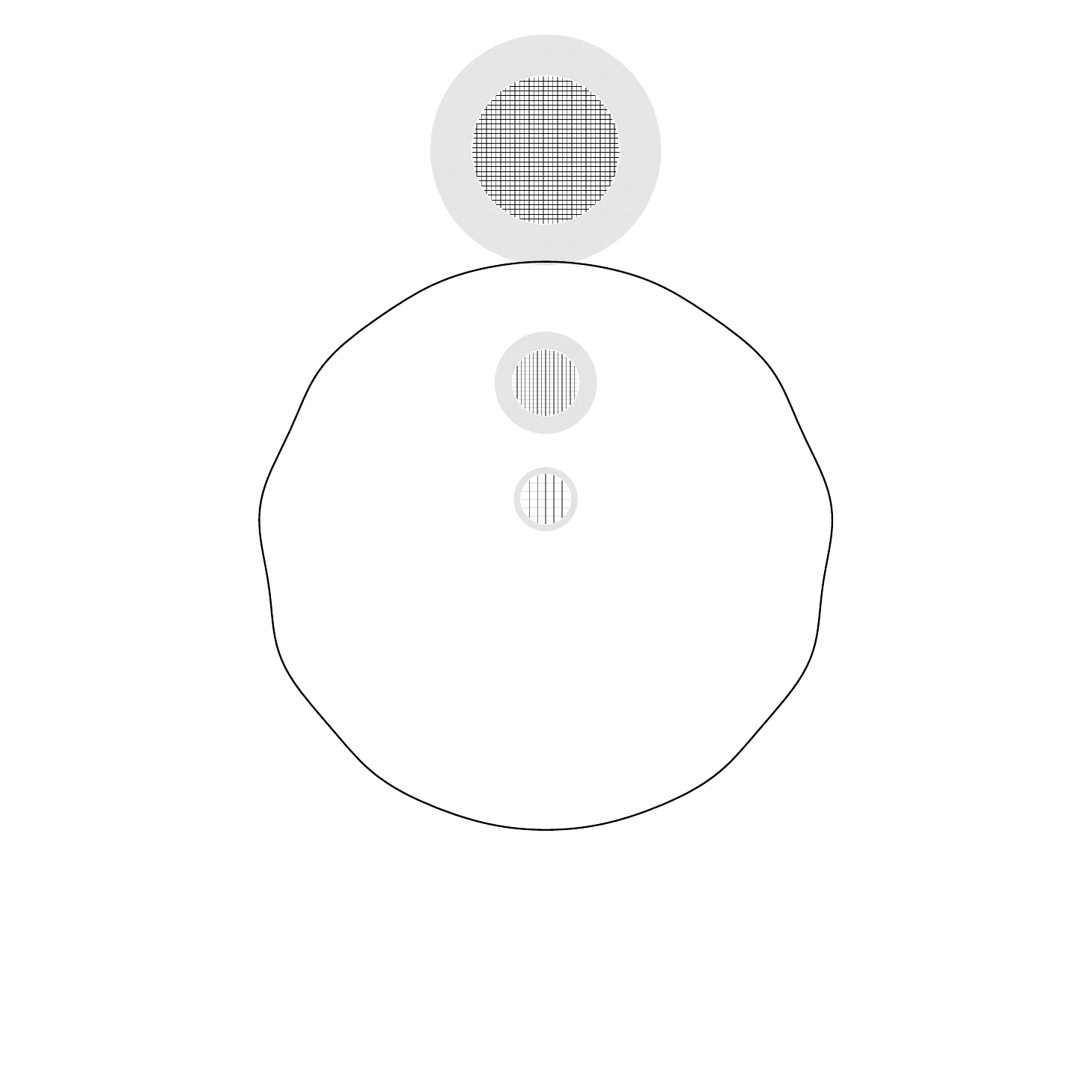}
		
	\end{subfigure}%
	\begin{subfigure}{0.5\textwidth}
		
		\includegraphics[width=1\linewidth]{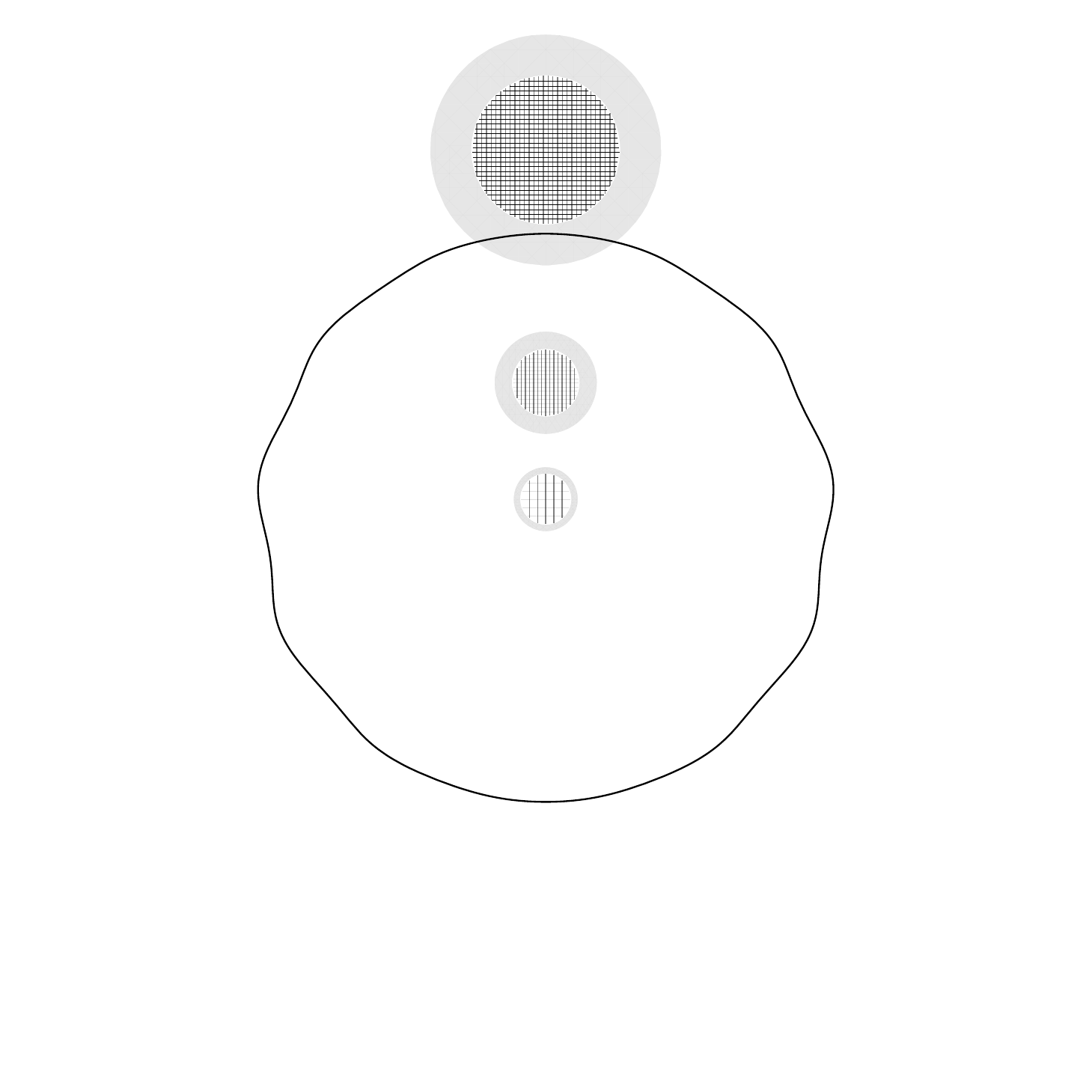}
	\end{subfigure}
	\caption{An illustration of the construction for the proof of Theorem \ref{counterexample thm 2}. The odd part of the scalar curvature is  positive in the shaded region and negative in the hatched region. On the left and right, the surfaces $\Sigma_{\xi,\lambda_j}$ corresponding to the choices $\xi=0$ and $\xi=t\, e_1$ for some small $t>0$, respectively, are shown. The latter has larger overlap with the shaded region, causing the derivative of $G_{2,\lambda_j}$ in direction $e_1$ to be negative at $\xi=0$.}
	\label{Figure no centering}
\end{figure}
\begin{proof}[Proof of Theorem \ref{counterexample thm 2}]
We construct a metric $g_3$ on $\mathbb{R}^3\setminus\{0\}$ that admits a foliation by area-constrained Willmore spheres whose leaves do not center around the origin; see Figure \ref{Figure no centering}. To this end, we let  $\chi:\mathbb{R}^3\to[0,\infty)$ be a  standard bump function
$$
\chi(y)=\begin{cases} &e^{-(1-|y|^2)^{-1}} \text{ if } |y|<1,
\\ &0\qquad\qquad\hspace{0.15cm} \text{ if } |y|\geq 1.
\end{cases}
$$
If $y\in B_1(0)$, we compute
$$
(\bar\Delta\chi)(y)=2\,\chi(y)\,\frac{4\,|y|^2+|y|^4-3}{(1-|y|^2)^{4}}.
$$
In particular, $\chi$ is strictly subharmonic on  $\{y\in\mathbb{R}^3:\sqrt{3}/2<|y|<1\}$. Let
$$
\psi(x)=\sum_{k=0}^\infty 10^{-2\,k}\, \chi\left(2\cdot10^{-k}\,(x-10^{k}\,e_1)\right)
$$
and define the conformally flat metric
$$
g_3=[1+|x|^{-1}-\epsilon\, (|x|^{-2}+\delta\,\psi(x))]^4\,\bar g 
$$
on $\mathbb{R}^3\setminus\{0\}$
where  $\epsilon,\,\delta>0$ will be chosen (small) later. Note that $g_3$ is $C^k$-asymptotic to Schwarzschild with mass $m=2$ for every $k\geq 2$. The scalar curvature of $g_3$ is given by
\begin{align*}
R&=8\, \epsilon\,\bar\Delta(|x|^{-2}+\delta\,\psi(x))+O(|x|^{-5})
\\&=8\,\epsilon\,\bigg(2\,|x|^{-4}+4\,\delta\sum_{k=0}^\infty 10^{-4\,k}\, (\bar\Delta\chi)\left(2\cdot10^{-k}\,(x-10^{k}\,e_1)\right)\bigg)+O(|x|^{-5}).
\end{align*}
In particular, $R\geq 0$ outside  a bounded set, provided $\delta>0$ is sufficiently small. A similar computation shows that, taking $\delta>0$ smaller if necessary, 
$$
\sum_{i=1}^3x^i\,\partial_i(|x|^{2}\,R)\leq 0
$$
outside  a bounded set.\\\indent  Arguing as in the proof of Lemma \ref{G der sRT}, we may choose $\epsilon>0$  small  such that $G_\lambda$ is strictly convex in $\{\xi\in\mathbb{R}^3:|\xi|<1/2\}$ and strictly radially increasing near $\{\xi\in\mathbb{R}^3:|\xi|=1/2\}$ provided $\lambda>1$ is sufficiently large. It follows that $G_\lambda$ has a unique critical point $\xi(\lambda)$ with $|\xi(\lambda)|<1/2$. Let
$$
\lambda_j=\frac{9}{16}\,10^j.
$$ As in (\ref{G decomposition}), we consider the decomposition
$$
G_{\lambda_j}=G_1+G_{2,\lambda_j}+o(1).
$$
Note that $ (\bar DG_1)(0)=0$. Using that $\chi$ is strictly subharmonic on $\{y\in\mathbb{R}^3:\sqrt{3}/2<|y|<1\}$ and that odd functions integrate to zero on a sphere, we obtain
$$
(\partial_1 G_{2,\lambda_j})(0)=-4\,\bigg(\frac{9}{16}\bigg)^4\,\epsilon\,\delta\int_{S_1(0)} (9\,x-16\, e_1)\,\bar g(e_1,\bar\nu) \,(\bar\Delta\chi)\,\mathrm{d} \bar\mu+o(1)=-c\,\epsilon\,\delta+o(1)
$$
where $c>0$ is  independent of $j$.
By Lemma \ref{C2  estimate}, the $C^2$-norm of $G_\lambda$ is bounded. It follows that there is $z\in(0,1/2)$ with $$|\xi({\lambda_j})|\geq z$$ provided $j$ is sufficiently large. \\ \indent To conclude, note that the same argument as in the proof of Proposition \ref{foliation prop}  shows that the family of surfaces $\{\Sigma_{\lambda,\xi(\lambda)}:\lambda>\lambda_0\}$ forms a smooth asymptotic foliation of $\mathbb{R}^3$.
\end{proof}
\begin{appendices}
	\section{The Willmore energy}
	\label{Willmore appendix}
	In this section, we provide some background material on the Willmore energy. We refer to \cite[\S3]{lamm2011foliations} for proofs and further information.\\ \indent  Let $(M,g)$ be a  Riemannian 3-manifold without boundary. Let $\Sigma\subset M$ be a closed, two-sided surface with unit normal $\nu$. The Willmore energy of $\Sigma$ is the quantity
	\begin{align}
	\mathcal{W}(\Sigma)=\frac14 \int_{\Sigma} H^2\,\mathrm{d}\mu, \label{Willmore energy}
	\end{align}
	where $H$ is the mean curvature scalar computed as the divergence of $\nu$ along $\Sigma$. \\ \indent  Let $\epsilon>0$ and $U\in C^{\infty}(\Sigma\times(-\epsilon,\epsilon))$ with $U(\,\cdot\,,\,0)=0$. 
	Decreasing $\epsilon>0$ if necessary, we obtain a smooth variation $\{\Sigma_s:s\in(-\epsilon,\epsilon)\}$ of embedded surfaces
	$\Sigma_s=\Phi_s(\Sigma_s)$ where
	$$
	\Phi_s:\Sigma\to M\qquad\text{  is given by }\qquad \Phi_s(x)=\operatorname{exp}_x(U(x,s)\,\nu(x)).
	$$
	We denote the initial velocity and initial acceleration of the variation by
	$$
	u(x)=\dot{U}(x,\,0)\qquad \text{  and } \qquad v(x)=\ddot{U}(x,\,0).
	$$
	\indent 	 In the following lemma, we recall the formulae for the first and the second variation of the Willmore energy (\ref{Willmore energy}). To this end, we recall that 
	\begin{align}
	Lf=-\Delta f-(|h|^2+\operatorname{Ric}(\nu,\nu))\,f
	\label{stability operator}
	\end{align}
	denotes the linearization of the mean curvature operator. In particular,
	\begin{align}
	Lu=\frac{d}{ds}\bigg|_{s=0}(H(\Sigma_s)\circ\Phi_s). \label{mean curvature change}
	\end{align}
	\begin{lem}[{\cite[\S 3]{lamm2011foliations}}] There holds \label{variation lemma}
		\begin{align}
		\frac{d}{ds}\bigg|_{s=0}\int_{\Sigma_s} H^2\,\mathrm{d}\mu=-2\int_{\Sigma} W\,u\,\mathrm{d}\mu
		\label{first variation}
		\end{align}
		where
		$$
		W=\Delta H +(|\hcirc|^2+\operatorname{Ric}(\nu,\nu))\,H.
		$$
		Moreover,
		\begin{align*}
	\frac{d^2}{ds^2}\bigg|_{s=0}\int_{\Sigma_s} H^2 \,\mathrm{d}\mu=-2\int_{\Sigma}\left[u\,Qu+H\,W\,u^2+W\,v\right]\mathrm{d}\mu
		\end{align*}
		where
		\begin{equation}
		\label{Q general formula}
		\begin{aligned}
		Qu=&\,L(Lu)+\frac12 \,H^2\, Lu+2\,H\,g(\hcirc,\nabla^2u)+2\,H\,\operatorname{Ric}(\nu,\nabla u)+2\,\hcirc(\nabla H,\nabla u)\\&\qquad +u\,\bigg[|\nabla H|^2+2\,\operatorname{Ric}(\nu,\nabla H)
		+H\,\Delta H+2\,g(\hcirc,\nabla^2 H)+2\,H^2\,|\hcirc|^2\\&\qquad\qquad\quad +2\,H\,g(\operatorname{Ric},\hcirc)-H\,(D_{\nu}\operatorname{Ric})(\nu,\nu)\bigg]	
			\end{aligned}
			\end{equation}
is the linearization of the Willmore operator.	
\end{lem}
The operator $Q$ measures how $W$ changes along a normal variation of $\Sigma$. More precisely,
\begin{align}
Qu=-\frac{d}{ds}\bigg|_{s=0}(W(\Sigma_s)\circ \Phi_s).
\label{Q change of W}
\end{align}
\indent Surfaces that are critical for the Willmore energy are called Willmore surfaces. The corresponding Euler-Lagrange equation is
$$
-W=0.
$$	
Surfaces that are critical for the Willmore energy among so-called area-preserving variations are called area-constrained Willmore surfaces. They satisfy the area-constrained Willmore equation
$$
-W=\kappa\, H
$$
where $\kappa\in\mathbb{R}$ is a Lagrange parameter. A variation $\{\Sigma_s:|s|<\epsilon\}$ is called area-preserving if
$
|\Sigma_s|=|\Sigma|
$
for all $s\in(-\epsilon,\epsilon)$. \\ \indent An area-constrained Willmore surface $\Sigma$ is stable if it passes the second derivative test for the Willmore energy among all area-preserving variations. Note that this is always satisfied if $\Sigma$ is a minimal surface. If $\Sigma$ is not a minimal surface, we define
$
u^\perp=u+s\,H
$
where
$$
s=-\bigg(\int_{\Sigma}H\,u\,\mathrm{d}\mu\bigg)^{-1}\,\int_{\Sigma}H^2\,\mathrm{d}\mu
$$
is chosen such that
$$
\int_{\Sigma} u^\perp\, H\,\mathrm{d}\mu=0.
$$
 It can be seen that $\Sigma$ is stable if and only if 
$$
\kappa \int_{\Sigma} u^\perp\,Lu^\perp\,\mathrm{d}\mu\leq \int_{\Sigma} u^\perp\,Qu^\perp\,\mathrm{d}\mu
$$
for every $u\in C^{\infty}(\Sigma)$; see \cite[(41) on p.~16]{lamm2011foliations}.
\section{Spherical harmonics and Legendre polynomials}
\label{spherical harmonics appendix}
In this section, we collect some standard facts about the Laplace operator on the unit sphere. 
\begin{lem}\label{spherical harmonics lemma 1}
	The eigenvalues of the operator 
	$$-\bar\Delta:H^{2}(S_1(0))\to L^2(S_1(0))$$
are given by
$$
\{l\,(l+1):l=0\,,1\,,2,\dots\}.
$$	
\end{lem}
We denote the eigenspace corresponding to the eigenvalue $l\,(l+1)$ by $$\Lambda_l(S_1(0))=\{f\in C^{\infty}(S_1(0)):-\bar\Delta f=l\,(l+1)f\}.$$ Recall that these eigenspaces are finite dimensional and that
$$
L^2(S_1(0))=\bigoplus_{l=0}^\infty \Lambda_l(S_1(0)). 
$$
\begin{coro}
	\label{Willmore eigenvalues}
	The eigenvalues of the operator
	$$
	\bar\Delta^2+2\,\bar\Delta: H^4(S_1(0))\to L^2(S_1(0))
	$$
	are given by
	$$\{(l-1)\,l\,(l+1)\,(l+2):l=0\,,1\,,2,\dots\}.$$
\end{coro}
\begin{lem} 
	There holds
	\begin{align*}
	\Lambda_0(S^1(0))&=\operatorname{span}\{1\},\\
	\Lambda_1(S^1(0))&=\operatorname{span}\{y^1,\,y^2,\,y^3\},\\
	\Lambda_2(S^1(0))&=\operatorname{span}\{Y_2^{11},\,Y_2^{22},\,Y_2^{12},\,Y_2^{13},\,Y_2^{23}\}.\\
	\end{align*}
	Here $y^1,\,y^2,\,y^3$ are the coordinate functions and
	$$
	Y_2^{ij}=\frac12\,(3\,y^i\,y^j-\delta^{ij}).
	$$
\end{lem}
We also record the following useful orthogonality relations.
\begin{lem}	\label{spherical identities}
	Let $i,j,k,l\in\{1,\,2\,,3\}$. There holds
	\begin{align*}
	\int_{S_1(0)} y^i\,y^j\,\mathrm{d}\bar\mu&=\frac{4\,\pi}{3}\,\delta_{ij},\\
	\int_{S_1(0)} y^i\,y^j\,y^k\,y^l\,\mathrm{d}\bar\mu&=\frac{4\,\pi}{15}\,(\delta_{ij}\,\delta_{kl}+\delta_{ik}\,\delta_{jl}+\delta_{il}\,\delta_{jk}),
	\\
	\int_{S_1(0)}Y_2^{ij}\,Y_2^{kl}\,\mathrm{d}\bar\mu&=\frac{\pi}{5}\,(3\,\delta_{ik}\,\delta_{jl}+3\,\delta_{il}\,\delta_{jk}-2\,\delta_{ij}\,\delta_{kl}),\\\int_{B_1(0)} y^i\,y^j\,\mathrm{d}\bar{v}&=\frac{4\,\pi}{15}\,\delta_{ij}.
	\end{align*}

\end{lem}
The Legendre polynomials $P_0,\,P_1,\,P_2,\dots$ may be defined via a generating function.
More precisely, given $s\in[0,1]$ and $t\in[0,1)$, there holds
\begin{align}
(1-2\,s\,t+t^2)^{-\frac12}=\sum_{l=0}^\infty P_l(s)\,t^l. \label{legendre generating2}
\end{align} 
A proof of the following lemma can be found in \cite[\S8]{CourantHilbert}.
\begin{lem}	\label{Legendre identities} Let $i\in\{1,\,2,\,3\}$ and $l,\,l_1,\,l_2\in\{0,\,1,\,2\dots\}$. 
	There holds
	\begin{align}
	P_l(-y^i)\in\Lambda_l(S_1(0)).
	\label{legendre polynomials are spherical harmonics}
	\end{align}
Moreover,
	$$
	\int_{S_1(0)} P_{l_1}(-y^i)\,P_{l_2}(-y^i)\,\mathrm{d}\bar\mu=\frac{4\,\pi}{2\,l+1}\,\delta_{l_1l_2}.
	$$
\end{lem}
 The next lemma extends the calculation of inverse powers of $|y+\xi|$ in  \cite[p.~668]{brendle2014large}.
\begin{lem}
	\label{inverse expansions lemma} Let $\xi\in\mathbb{R}^3$ with $|\xi|\neq 1$ and   $k\in\{0,\,1\,,2\,,3\}$.
	There holds, for all $y\in S_1(0)$
	\begin{align*}
	|y+\xi|^{-2k-1}=\begin{dcases}&\sum_{l=0}^\infty\, a_{k,l}(\xi)\,|\xi|^l\, P_l(-|\xi|^{-1}\,\bar g(y,\xi))\quad\hspace{0.6cm} \text{ if } |\xi|<1,\\&\sum_{l=0}^\infty\, \tilde a_{k,l}(\xi)\,|\xi|^{-l-1}\, P_l(-|\xi|^{-1}\,\bar g(y,\xi))\quad \text{ if } |\xi|>1.
	\end{dcases}
	\end{align*}
	Here,
	\begin{align*}
	a_{0,l}(\xi)&=1,\\
	 a_{1,l}(\xi)&=(2\,l+1)\,\frac{1}{1-|\xi|^2},\\
	 a_{2,l}(\xi)&=(2l+1)\,\frac{(2\,l+3)-(2\,l-1)\,|\xi|^2}{3\,(1-|\xi|^2)^3}, \\ 
	a_{3,l}(\xi)&=(2l+1)\,\frac{(2\,l+3)\,(2\,l+5)-2\,(2\,l-3)\,(2\,l+5)\,|\xi|^2+(2\,l-3)\,	(2\,l-1)\,|\xi|^4}{15\,(1-|\xi|^2)^5},
	\end{align*}
	and
	$$
	\tilde a_{k,l}(\xi)=(-1)^k\,|\xi|^{-2k}\,a_{k,l}(|\xi|^{-2}\,\xi).
	$$
\end{lem}
\begin{proof}
	If $k=0$, the expansions follow from (\ref{legendre generating2}) with $s=-|\xi|^{-1}\,\bar g(y,\xi)$ and choice of
	$$
	t=\begin{dcases}& |\xi|\qquad\,\,\,\,\,\, \text{if }|\xi|<1, \\
	&|\xi|^{-1}\qquad \text{if }|\xi|>1.
	\end{dcases}
	$$ 
	The asserted formula follows from the recursive relation
	$$
	|y+\xi|^{-2k-1}=\frac{1}{1-|\xi|^2}\,\bigg(\frac{2}{2\,k-1}\,\sum_{i=1}^3\xi^i\,\partial_i( |y+\xi|^{-2k+1})+|y+\xi|^{-2k+1}\bigg),
	$$
	where the partial derivative is with respect to $\xi$.
\end{proof}
\begin{lem}
	Let $t\in(-1,1)$. There holds
	$$
	\log(1+t)=\sum_{l=1}^\infty \frac{(-1)^{l-1}}{l}\,t^l
	\qquad \text{ and } \qquad 
	\frac{1}{(1-t^2)^2}=\sum_{l=0}^\infty (1+l)\, t^{2\,l}.
	$$
	\label{log identities}
\end{lem}
	\section{Some geometric expansions}
In this section, we compute several geometric expansions needed in this paper. Recall that $S_{\xi,\lambda}=S_\lambda(\lambda\,\xi)$. We distinguish between the cases   $|\xi|<1-\delta$ and $|\xi|>1+\delta$ where $\delta\in(0,1/2)$. We will assume that $\lambda>\lambda_0$ where $\lambda_0>1$ is large. We note that, for every $x\in S_{\xi,\lambda}$,
		\begin{align*}
		\delta\leq \lambda^{-1}\, |x| \leq 2-\delta \qquad \text{if}\qquad |\xi|<1-\delta
		\end{align*} 
		and
		\begin{align*} 
		|\xi|-1&\leq
		\lambda^{-1}\,|x|\leq |\xi|+1\qquad \text{if}\qquad  |\xi|>1+\delta.
		\end{align*} \indent 
Throughout, we assume that  $g$ is a Riemannian metric on $\mathbb{R}^3$ such that
\begin{align*}
g=(1+|x|^{-1})^4\,\bar{g} +\sigma
\end{align*} 
where $\sigma$ is a symmetric, covariant two-tensor with, as $x\to\infty$  for every multi-index $J$ with $|J|\leq 4$,
\begin{align} \label{ac sigma}  \partial_J \sigma=O(|x|^{-2-|J|}).\end{align} 
 We point out additional assumptions on $g$ where required. \\ \indent 
 The estimates below depend on $\lambda_0>1,\,\delta\in(0,1/2)$, and $g$. They are otherwise  independent of $\xi$ and  $\lambda$. For outlying surfaces, where $|\xi|>1+\delta$, we recall that a bar underneath a quantity indicates evaluation at $\lambda\,\xi$, possibly after taking derivatives.  For example,
 $$
 \sigma(\lambda\,\bar\nu+\lambda\,\xi)=\ubar{\sigma} +\lambda\, D_{\bar\nu}\ubar{\sigma}+O(\lambda^{-2}\,|\xi|^{-4})
 $$
 is shorthand for
 $$
\sigma(\lambda\,\bar\nu+\lambda\,\xi)={\sigma}(\lambda\,\xi) +\lambda\, (D_{\bar\nu}{\sigma})(\lambda\,\xi)+O(\lambda^{-2}\,|\xi|^{-4}).
 $$ \indent When stating that an error term such as  $\mathcal{E}=O(\lambda^{-l_1}\,|\xi|^{-l_2})$ may be differentiated with respect to $\xi$  with $l_1,\,l_2\in\mathbb{Z}$, we mean that
 $$
\bar{D} \mathcal{E}= O(\lambda^{-l_1}\,|\xi|^{-l_2-1}),
 $$
 where differentiation is with respect to $\xi$.\\\indent
For the statements below, recall that 
$$
\phi(x)=1+|x|^{-1}
$$
denotes the conformal factor of the Schwarzschild metric.
\begin{lem}
	 We have, for $i,\,j=1,\,2,\,3$,
	 	\begin{align}
	 (\operatorname{Ric}_S)(e_i,e_j)=2 \,\phi^{-2} \,|x|^{-3}(\delta_{ij}-3\,|x|^{-2}\,x^i\,x^j).
	 \end{align}
Moreover, there holds
	\label{Schwarzschild identities}
	$$
\nu_S(S_{\xi,\lambda})=\phi^{-2}\,\bar\nu \qquad \text{ 
	and } \qquad 
H_S(S_{\xi,\lambda})=2\,\phi^{-2}\,\lambda^{-1}-4\,\phi^{-3}\, |x|^{-3}\,\bar g(x,\bar\nu).
$$
\end{lem}	
A more precise version of the expansion in the following lemma was computed in \cite[p.~670]{brendle2014large}. 
\begin{lem}[\cite{brendle2014large}]
	\label{area expansion}
	There holds
	\begin{align*}
	|S_{\xi,\lambda}|=\begin{dcases}&	
	4\,\pi \,\lambda^2+ 16\,\pi\, \lambda+O(1)\hspace{2.9cm} \text{ if } |\xi|<1-\delta,\\
	&4\,\pi\, \lambda^2+ 16\,\pi\, \lambda\,|\xi|^{-1}+O(|\xi|^{-2}) \hspace{1.45cm}\text{ if } |\xi|>1+\delta.
	\end{dcases}	
	\end{align*}
\end{lem}
We need a more precise expansion of the Willmore energy of $S_{\xi,\lambda}$ than that computed in \cite{chodosh2017global}. To this end, we first compute the dependence of certain geometric quantities on the perturbation $\sigma$ away from $g_S$. 
\begin{lem}
	Let $\{e_1,\, e_2\}$ be a local Euclidean orthonormal frame for $TS_{\xi,\lambda}$. There holds
	\begin{align*}
	\nu-\nu_S=&\,-\frac12\,\phi^{-6}\,\sigma(\bar\nu,\bar\nu)\,\bar\nu-\phi^{-6}\sum_{\alpha=1}^2\sigma(\bar\nu,e_\alpha)\,e_\alpha+O(\lambda^{-4}\,(1+|\xi|)^{-4}),\\
	\hcirc^\alpha_\beta=&\,-\frac12\,\lambda^{-1}\,\phi^{-6}\,\left[2\,\sigma(e_\alpha,e_\beta)-(\bar {\operatorname{tr}}\,\sigma-\sigma(\bar\nu,\bar\nu))\,\delta_{\alpha\beta}\right]\\&\qquad -\frac12\,\left[(\bar D_{e_\alpha}\sigma)(\bar\nu,e_\beta)+(\bar D_{e_\beta}\sigma)(\bar\nu,e_\alpha)-(\bar D_{\bar\nu}\sigma)(e_\alpha,e_\beta)\right]\\&\qquad +\frac14\,\left[2\,(\bar{\operatorname{div}}\sigma)(\bar\nu)-\bar D_{\bar\nu}\bar {\operatorname{tr}}\,\sigma- (\bar D_{\bar\nu}\sigma)(\bar\nu,\bar\nu)\right]\,\delta_{\alpha\beta}+O(\lambda^{-4}\,(1+|\xi|)^{-4}),
\\	H-H_S=&\,\lambda^{-1}\,\phi^{-6}\,\left[2\,\sigma(\bar\nu,\bar\nu)-\bar {\operatorname{tr}}\,\sigma\right]+\frac12 \,\left[\bar D_{\bar\nu}\bar {\operatorname{tr}}\,\sigma+(\bar D_{\bar\nu}\sigma) (\bar\nu,\bar\nu)-2\,(\bar{\operatorname{div}}\sigma)(\bar\nu)\right]\\&\qquad +O(\lambda^{-4}\,(1+|\xi|)^{-4}),\\
\bar\Delta(H-H_S)=&\,4\,\lambda^{-3}\,\ubar{\phi}^{-10}\,\left[\bar {\operatorname{tr}}\,\sigma-3\,\sigma(\bar\nu,\bar\nu)\right]+Y_1+Y_3+O(\lambda^{-5}\,(1+|\xi|)^{-4})
	\end{align*}
	where $\alpha,\,\beta=1,\,2$.
	\label{perturbations}
	Here, $Y_1$ and $Y_3$ are respectively first and third spherical harmonics with
	$$
	Y_1=O(\lambda^{-5}\,(1+|\xi|)^{-3})\qquad\text{ and } \qquad  Y_3=O(\lambda^{-5}\,(1+|\xi|)^{-3}).
	$$
	If $g$ satisfies \eqref{ac sigma} for every multi-index $J$ with $|J|\leq 5$,  these identities may be differentiated once with respect to $\xi$.
\end{lem}
\begin{proof}
Given $t\in[0,1]$, we define the family of metrics $
g_t=g_S+t\,\sigma
$
such that $g_0=g_S$ and $g_1=g$. The identities can be obtained upon linearizing the respective quantities at $t=0$.
\end{proof} Second, we recall the conformal killing operator
$$
\mathcal{D}Z=\mathcal{L}_Z g-\frac13\,\operatorname{tr}\,(\mathcal{L}_Z g)\, g
$$
where $Z$ is a vector field. 
\begin{lem}
	Let $Z=\phi^{-2}\,(x-\lambda\,\xi)$. If $|\xi|<1-\delta$, there holds, in $\mathbb{R}^3\setminus B_\lambda(\lambda\,\xi)$,
	$$
	\mathcal{D}_SZ=O(\lambda^{-1}\,|x|^{-1})\qquad\text{and}\qquad  \mathcal{D}_SZ-\mathcal{D}Z=O(\lambda^{-1}\,|x|^{-2}).
	$$
	If $|\xi|>1-\delta$, there holds, in $ B_\lambda(\lambda\,\xi)$,
	$$
	\mathcal{D}_SZ=O(\lambda^{-2}\,|\xi|^{-2})\qquad\text{and}\qquad  \mathcal{D}_SZ-\mathcal{D}Z=O(\lambda^{-3}\,|\xi|^{-3}).
	$$
	\label{killing lemma}
	
\end{lem}
\begin{proof}
Note that
	$$
	(\mathcal{D}Z)(e_i,e_j)=g(D_{e_i} Z,e_j)+g(D_{e_j}Z,e_i)-\frac23\,(\operatorname{div}Z)\,g(e_i,e_j).
	$$
We compute
$$
\bar D_{e_i}Z=\phi^{-2}\,\lambda^{-1}\,e_i+2\,\phi^{-1}\,|x|^{-3}\,\bar g(x,e_i)\,Z
$$
and
$$
\bar{\operatorname{div}}Z=3\,\phi^{-2}\,\lambda^{-1}+2\,\phi^{-1}\,|x|^{-3}\,\bar g(x,Z).
$$
Moreover, we have
\begin{align*}
&\circ\qquad D_SZ-\bar DZ=O(|x|^{-2}\,|Z|), \\
&\circ\qquad DZ- D_SZ=O(|x|^{-3}\,|Z|), \\
&\circ\qquad {\operatorname{div}}_SZ-\bar{\operatorname{div}} Z=O(|x|^{-2}\,|Z|),\text{ and} \\
&\circ\qquad {\operatorname{div}}Z- {\operatorname{div}}_SZ=O(|x|^{-3}\,|Z|).
\end{align*}
 Finally, note that
$Z=O(\lambda^{-1}\,|x|)$ in $\mathbb{R}^3\setminus B_\lambda(\lambda\,\xi)$ if $|\xi|<1-\delta$ and that  $Z=O(1)$ in $B_\lambda(\lambda\,\xi)$ if $|\xi|>1+\delta$. \\ \indent The assertion follows from these estimates.
\end{proof}
For Lemma \ref{Willmore energy sphere} below, recall that
$$
\sigmacirc=\sigma-\frac13\,(\bar {\operatorname{tr}}\,\sigma)\,\bar g.
$$
\begin{lem}
	\label{Willmore energy sphere}
	If $|\xi|<1-\delta$, there holds
	\begin{align*}
	\int_{{S}_{\xi,\lambda}} H^2\,\mathrm{d}\mu=&\,16\,\pi-64\,\pi\,{\lambda}^{-1} +8\,\pi\,\lambda^{-2}\bigg[\frac{10-6\,|\xi|^2}{(1-|\xi|^2)^2}+3\,|\xi|^{-1}\log\frac{1+|\xi|}{1-|\xi|}\bigg]\\&\qquad +2\,\lambda^{-1}\int_{\mathbb{R}^3\setminus{B_{\lambda}(\lambda\,\xi)}} R\,\mathrm{d}\bar{v} +O(\lambda^{-3}).	\end{align*}
 If $|\xi|>1+\delta$, there holds
	\begin{align*}
	\int_{{{S}_{\xi,\lambda}}} H^2\,\mathrm{d}\mu=\,&16\,\pi +8\,\pi\,\lambda^{-2}\bigg[\frac{10-6\,|\xi|^2}{(|\xi|^2-1)^2}+3\,|\xi|^{-1}\log\frac{|\xi|+1}{|\xi|-1}\bigg]-2\,\lambda^{-1}\,\ubar{\phi}^{4} \int_{{B_{\lambda}(\lambda\,\xi)}} R\,\mathrm{d}\bar{v}\\\,&\qquad +\frac{48\,\pi}{15}\,\ubar{\phi}^{-8}\,|\ubar{\sigmacirc}|^2+\frac{128\,\pi}{15}\,\lambda^{-1}\,|\xi|^{-3}\,(3\,|\xi|^{-2}\,\ubar{ \sigma}(\xi,\xi)-\bar {\operatorname{tr}}\,\ubar{\sigma}) +O(\lambda^{-3}\,|\xi|^{-6}).
\end{align*}
Both expressions may be differentiated twice with respect to $\xi$.
\end{lem}
\begin{proof}  
	We first compute the Schwarzschild contribution. Note that
		\begin{align}
	2\,\bar g(x, \bar{\nu})=\lambda\,(1-|\xi|^2)+\lambda^{-1}\,|x|^2.
	\label{angle}
	\end{align}
	In the case where $|\xi|<1-\delta$, we use Lemma \ref{Schwarzschild identities} to compute
	\begin{align*}
	\int_{S_{\xi,\lambda}} H_S^2\,\mathrm{d}\mu_S&=\int_{S_{\xi,\lambda}}\left[4\,\lambda^{-2}-16\,\lambda^{-1}\,(|x|^{-3}-|x|^{-4})\,\bar g(x,\bar\nu)+16\, |x|^{-6}\,\bar g(x,\bar\nu)^2\right]\,\mathrm{d}\bar\mu +O(\lambda^{-3})
	\\&=16\,\pi-64\,\pi\,{\lambda}^{-1} +16\,\pi\,\lambda^{-2}\frac{5-3\,|\xi|^2}{(1-|\xi|^2)^2}+24\,\pi\,\lambda^{-2}\,|\xi|^{-1}\log\frac{1+|\xi|}{1-|\xi|}+O(\lambda^{-3}).
	\end{align*}
	In the case where $|\xi|>1+\delta$, we compute that
		\begin{align*}
	\int_{S_{\xi,\lambda}} H_S^2\,\mathrm{d}\mu_S=16\,\pi +16\,\pi\,\lambda^{-2}\frac{5-3\,|\xi|^2}{(|\xi|^2-1)^2}+24\,\pi\,\lambda^{-2}|\xi|^{-1}\log\frac{|\xi|+1}{|\xi|-1}+O(\lambda^{-3}\,|\xi|^{-6}).
	\end{align*}
\indent 	To compute the contribution from the perturbation $\sigma$ off Schwarzschild, we start from the identity 
	\begin{align}
	\int_{S_{\xi,\lambda}} H^2\,\mathrm{d}\mu=16\,\pi+2\int_{S_{\xi,\lambda}} |\hcirc|^2\, \mathrm{d}\mu+2\int_{S_{\xi,\lambda}} \left(2\,\operatorname{Ric}(\nu,\nu)-R\right)\,\mathrm{d}\mu,
	\label{integrated Gauss II}
	\end{align}
	cf.~(\ref{integrated Gauss equation 0}). The first integral on the right  vanishes if $\sigma=0$.  If $|\xi|<1-\delta$ and $\sigma\neq 0$, using Lemma \ref{perturbations}, we estimate
	$$
	\int_{S_{\xi,\lambda}} |\hcirc|^2\, \mathrm{d}\mu=O(\lambda^{-4}).
	$$
	Conversely, if $|\xi|>1+\delta$, using Lemma \ref{perturbations}, Taylor expansion, and cancellations due to symmetry, we find that
	\begin{align*}
		\int_{S_{\xi,\lambda}} |\hcirc|^2\, \mathrm{d}\mu=\,&\ubar{\phi}^{-8}\,\lambda^{-2}\int_{S_{\xi,\lambda}} 
\left[\bar g(\ubar{\sigma}_{|_{S_{\xi,\lambda}}},\ubar{\sigma}_{|_{S_{\xi,\lambda}}})-\frac12\,(\bar {\operatorname{tr}}_{S_{\xi,\lambda}}\,\ubar{\sigma})^2\right] \mathrm{d}\bar\mu+O(\lambda^{-4}\,|\xi|^{-6})\\
		=\,&\ubar{\phi}^{-8}\,\lambda^{-2}\int_{S_{\xi,\lambda}} 
		\bigg[|\ubar{\sigma}|^2-\frac12\,(\bar {\operatorname{tr}}\,\ubar{\sigma})^2+
		\frac12\, \ubar\sigma(\bar\nu,\bar\nu)\, \ubar\sigma(\bar\nu,\bar\nu)+ \ubar\sigma(\bar\nu,\bar\nu)\,\bar {\operatorname{tr}}\,\ubar\sigma \\&\qquad\quad\quad\qquad -2\sum_{i=1}^3\ubar{\sigma}(\bar\nu,e_i)\,\ubar{\sigma}(\bar\nu,e_i)\bigg] \mathrm{d}\bar\mu+O(\lambda^{-4}\,|\xi|^{-6})\\
		=\,&\frac{24\,\pi}{15}\,\ubar{\phi}^{-8}\,|\ubar{\sigmacirc}|^2+O(\lambda^{-4}\,|\xi|^{-6}).
	\end{align*}
	In the last step, we have used Lemma \ref{spherical identities}. \\ \indent 
	To compute the second integral in (\ref{integrated Gauss II}), first recall that the Einstein tensor
	$$
	E=\operatorname{Ric}-\frac12\, R\,g
	$$
	 is divergence free. If $|\xi|<1-\delta$, this leads to the following form of the Pohozaev identity
	\begin{align}
\int_{S_{\xi,\lambda}}E(Z,\nu)\,\mathrm{d}\mu=\int_{S_r(0)} E(Z,\nu)\,\mathrm{d}\mu-\int_{B_r(0)\setminus B_\lambda(\lambda\,\xi)}\left[\frac12\, g(E,\mathcal{D}Z)-\frac16\,(\operatorname{div}Z)\,R\right]\mathrm{d} v,	
\label{pohozaev 1}
	\end{align}
		valid for every vector field $Z$ and every $r>2\,\lambda$. Similarly, if $|\xi|>1+\delta$, we have 
			\begin{align}
		\int_{S_{\xi,\lambda}}E(Z,\nu)\,\mathrm{d}\mu=\int_{ B_\lambda(\lambda\,\xi)}\left[\frac12\, g(E,\mathcal{D}Z)-\frac16\,(\operatorname{div}Z)\,R\right]\mathrm{d} v.
		\label{pohozaev 2}
		\end{align}  We refer to \cite[\S 6.4]{lamm2011foliations} for a discussion of the Pohozaev identity and a related application.
		\\\indent	
		Let $$Z=\phi^{-2}\lambda^{-1}(x-\lambda\,\xi)$$
		and note that $Z=\nu_S$ on $S_{\xi,\lambda}$. Consequently,
		$$
		\int_{S_{\xi,\lambda}} E_S(\nu_S,\nu_S)\,\mathrm{d}\mu_S= \int_{S_{\xi,\lambda}} E_S(\nu_S,Z)\,\mathrm{d}\mu_S
		$$
		and
		$$
		\int_{S_{\xi,\lambda}} E(\nu,\nu)\,\mathrm{d}\mu= \int_{S_{\xi,\lambda}} E(\nu,Z)\,\mathrm{d}\mu+\int_{S_{\xi,\lambda}} E(\nu,\nu-\nu_S)\,\mathrm{d}\mu.
		$$
\indent 		If $|\xi|<1-\delta$, we let $r\to\infty$ in \eqref{pohozaev 1} and obtain, using Lemma \ref{Schwarzschild identities}, Lemma \ref{killing lemma}, and  that $R_S=0$, 
		\begin{align*}
		\int_{S_{\xi,\lambda}}E_S(\nu_S,Z)\,\mathrm{d}\mu_S=\,&-16\,\pi\,\lambda^{-1}-\frac12\,\int_{\mathbb{R}^3\setminus B_{\lambda}(\lambda\,\xi)} g_S(E_S,\mathcal{D}_SZ)\,\mathrm{d} v_S	\\=\,&-16\,\pi\,\lambda^{-1}-\frac12\,\int_{\mathbb{R}^3\setminus B_{\lambda}(\lambda\,\xi)} g(E,\mathcal{D}Z)\,\mathrm{d} v	+O(\lambda^{-3}).
		\end{align*}
Likewise, we have		
	$$
	\int_{S_{\xi,\lambda}}E(\nu,Z)\,\mathrm{d}\mu=-16\,\pi\,\lambda^{-1}-\frac12\,\int_{\mathbb{R}^3\setminus B_{\lambda}(\lambda\,\xi)} g(E,\mathcal{D}Z)\,\mathrm{d} v+\frac16\,\int_{\mathbb{R}^3\setminus B_{\lambda}(\lambda\,\xi)}(\operatorname{div}Z)\,R\,\mathrm{d}v.
	$$
	Finally, we use the coarse estimates
	$$
	\int_{\mathbb{R}^3\setminus B_{\lambda}(\lambda\,\xi)}(\operatorname{div}Z)\,R\,\mathrm{d}v=3\,\lambda^{-1}\,\int_{\mathbb{R}^3\setminus B_{\lambda}(\lambda\,\xi)}R\,\mathrm{d}v+O(\lambda^{-3})
	$$	
	and 
	$$
	\int_{S_{\xi,\lambda}} E(\nu,\nu-\nu_S)\,\mathrm{d}\mu=O(\lambda^{-3}).
	$$
	\indent 
		If $|\xi|>1+\delta$, we use \eqref{pohozaev 2}, Lemma \ref{killing lemma}, and $R_S=0$ to obtain that
	\begin{align*}
	\int_{S_{\xi,\lambda}}E_S(\nu_S,Z)\,\mathrm{d}\mu_S=\,&\frac12\,\int_{ B_{\lambda}(\lambda\,\xi)} g_S(E_S,\mathcal{D}_SZ)\,\mathrm{d} v_S
\\	=\,&\frac12\,\int_{ B_{\lambda}(\lambda\,\xi)} g(E,\mathcal{D}Z)\,\mathrm{d} v+O(\lambda^{-3}\,|\xi|^{-6}).
	\end{align*}
	Likewise, we have		
	$$
	\int_{S_{\xi,\lambda}}E_S(\nu,Z)\,\mathrm{d}\mu=\frac12\,\int_{B_{\lambda}(\lambda\,\xi)} g(E,\mathcal{D}Z)\,\mathrm{d} v-\frac16\,\int_{ B_{\lambda}(\lambda\,\xi)}(\operatorname{div}Z)\,R\,\mathrm{d}v.
	$$
Note that
	$$
	-\frac16 \int_{ B_\lambda(\lambda\,\xi)}(\operatorname{div}Z)\,R\,\mathrm{d} v=-\frac12\,  \ubar{\phi}^4\,\lambda^{-1} \int_{ B_\lambda(\lambda\,\xi)}R\,\mathrm{d} \bar{v}+O(\lambda^{-3}\,|\xi|^{-6}).
	$$
	Moreover, using that $R=O(|x|^{-4})$, we have
	$$
	\int_{S_{\xi,\lambda}}E(\nu-\nu_S,\nu)\,\mathrm{d}\mu=\int_{S_{\xi,\lambda}}\operatorname{Ric}(\nu-\nu_S,\nu)\,\mathrm{d}\mu+O(\lambda^{-3}\,|\xi|^{-6}).
	$$
Finally, using Lemma \ref{Schwarzschild identities} and the expansion for $\nu-\nu_S$ from Lemma 	\ref{perturbations}, we
	obtain
	\begin{align*}
		\int_{S_{\xi,\lambda}}\operatorname{Ric}&(\nu-\nu_S,\nu)\,\mathrm{d}\mu\\=&
		-\int_{S_{\xi,\lambda}} |x|^{-3}\,\left[\ubar{\sigma}(\bar\nu,\bar\nu)\,\big(1+3\,|x|^{-2}\,\bar g(x,\bar\nu)^2\big)-6\,|x|^{-2}\,\ubar{\sigma}(x,\bar\nu)\,g(x,\bar\nu)\right]\mathrm{d}\bar \mu
		+O(\lambda^{-3}\,|\xi|^{-6})
		\\=&
		-\int_{S_{\xi,\lambda}} \lambda^{-3}\,|\xi|^{-3}\,\left[\ubar{\sigma}(\bar\nu,\bar\nu)\,\big(1+3\,|\xi|^{-2}\,\bar g(\xi,\bar\nu)^2\big)-6\, |\xi|^{-2}\,\ubar{\sigma}(\xi,\bar\nu)\,\bar g(\xi,\bar \nu)\right]\mathrm{d}\bar \mu
		+O(\lambda^{-3}\,|\xi|^{-6})
		\\=&\,\frac{32\,\pi}{15}\,\lambda^{-1}\,|\xi|^{-3}\,\left[3\,|\xi|^{-2}\,\ubar{ \sigma}(\xi,\xi)-\bar {\operatorname{tr}}\,\ubar{\sigma}\right]+O(\lambda^{-3}\,|\xi|^{-6}).
	\end{align*}
	We have used Lemma \ref{spherical identities} in the third equality. \\ \indent The assertion follows from these estimates.
\end{proof}
\begin{rema}\normalfont
Let $\{S_j\}_{j=1}^\infty$ be a sequence of coordinate spheres $S_j=S_{\lambda_j}(\lambda_j\,\xi_j)$ with $\lambda_j>1$ and \label{slow divergence} $\xi_j\in\mathbb{R}^3$ that are slowly divergent in the sense that $\lim_{j\to\infty}\rho_j=\infty$ and $\lim_{j\to\infty} \lambda_j^{-1}\,\rho_j=0$
 where $\rho_j=\rho(S_j)$. If the spheres are on-center, we compute 
\begin{equation*}
\begin{aligned}
\int_{S_j} H_S^2\,\mathrm{d}\mu_S=
16\,\pi-32\,\pi\,\lambda_j^{-1}(2-\rho_j^{-1})+8\,\pi\,\rho_j^{-2}+O(\lambda_j^{-2}\,\log\lambda_j)+O(\lambda_j^{-1}\,\rho_j^{-2})+O(\rho_j^{-3}).
\end{aligned}
\end{equation*}
If the spheres are outlying, we have
\begin{align*}
\int_{S_j} H_S^2\,\mathrm{d}\mu_S = 16\,\pi+32\,\pi\,\lambda_j^{-1}\,\rho_j^{-1}+8\,\pi\,\rho_j^{-2}+O(\lambda_j^{-2}\,\log\lambda_j)+O(\lambda_j^{-1}\,\rho_j^{-2})+O(\rho_j^{-3}).
\end{align*}
Using Lemma \ref{Schwarzschild identities}, we compute that, in either case,
$$
\min_{x\in S_j} (\phi^2\,H_S)=2\,\lambda_j^{-1}-4\,\rho_j^{-2}+O(\rho_j^{-3}).
$$
Thus, if $\rho^2_j=o(\lambda_j)$, it follows that $\operatorname{min}_{x\in S_j}H_S<0$ and  $m_H(S_j)<0$ for all $j$ large.
\end{rema}

Next, we express the Willmore operator $-W({S_{\xi,\lambda}})$ in terms of spherical harmonics. 
\begin{lem}
	There holds
	\begin{align*}
	{W}({{S}_{\xi,\lambda}})&=\,\frac{1}{2}\,\phi^{-8}\big[-9\,\lambda^{-3}\,|x|^{-1}+(3\,|\xi|^2-7)\,\lambda^{-1}\,|x|^{-3}-3\,(1-|\xi|^2)\,(7\,|\xi|^2+5)\,\lambda\,|x|^{-5}\\  &\qquad\qquad +15\,(1-|\xi|^2)^3\,\lambda^3\,|x|^{-7}\big] 
	\\&\qquad \,+4\,\ubar{\phi}^{-10}\,\lambda^{-3}\,[\bar {\operatorname{tr}}\,\ubar{\sigma}-3\,\ubar{\sigma}(\bar\nu,\bar\nu)]+Y_1+Y_3+O(\lambda^{-5}\,(1+|\xi|)^{-4}).
	\end{align*}
	Here $Y_1$ is a first spherical harmonic and $Y_3$ is a third spherical harmonic. They satisfy 	\label{willmore operator one lemma}
	$$
	Y_1=O(\lambda^{-5}\,(1+|\xi|)^{-3})\qquad\text{ and } \qquad  Y_3=O(\lambda^{-5}\,(1+|\xi|)^{-3}).
	$$
	If $g$ satisfies \eqref{ac sigma} for every multi-index $J$ with $|J|\leq 5$, this identity may be differentiated once with respect to $\xi$.
\end{lem}
\begin{proof}
	Using Lemma \ref{Schwarzschild identities} and (\ref{angle}), we  find
	\begin{align*}
	H\,\operatorname{Ric}(\nu,\nu)&=\phi^{-8}\big[-3\,\lambda^{-3}\,|x|^{-1}+2\,(3\,|\xi|^2-1)\,\lambda^{-1}\, |x|^{-3}-3\,  (1-|\xi|^2)^2\,\lambda\,|x|^{-5}\big]\\&\qquad+{O}(\lambda^{-5}\,(1+|\xi|)^{-4}).
	\end{align*}
	Next, using the transformation of the Laplacian under a conformal change of the metric, we find
	\begin{align*}
	\Delta_{S_{\xi,\lambda}} H=\phi^{-4}\,\bar{\Delta}_{S_{\xi,\lambda}} H_S+\phi^{-4}\,\bar\Delta_{S_{\xi,\lambda}}(H-H_S)+{O}(\lambda^{-6}\,(1+|\xi|)^{-4}).
	\end{align*}
	Let $$\psi:\mathbb{R}^3\to\mathbb{R}\qquad\text{be given by} \qquad \psi(x)=2\,\phi^{-3}\,\lambda^{-1}-2\,\phi^{-3}\,\lambda\,(1-|\xi|^2)\,|x|^{-3}.$$
	By Lemma \ref{Schwarzschild identities} and \eqref{angle}, we have
	$\psi(x)=H_S(x)$ for all $x\in S_{\xi,\lambda}$. Consequently,
	$$
	\bar{\Delta}_{S_{\xi,\lambda}}H_S=\bar\Delta_{\mathbb{R}^3}\psi-\bar D^2_{\bar\nu,\bar\nu}\psi-2\,\lambda^{-1}\,\bar D_{\bar\nu}\psi.
	$$
	We compute
	$$
	\partial_i \psi=6\,\phi^{-4}\,\lambda^{-1}\,|x|^{-3}\,x^i+6\,\phi^{-4}\,\lambda\,(1-|\xi|^2)\,|x|^{-5}\,x^i
	$$
and 
\begin{align*} 
\partial_i\partial_j\psi&=6\,\phi^{-4}\,\lambda^{-1}\,|x|^{-3}\,(\delta_{ij}-3\,|x|^{-2}\,x^i\,x^j)+6\,\phi^{-4}\,\lambda\,(1-|\xi|^2)\,|x|^{-5}\,(\delta_{ij}-5\,|x|^{-2}\,x^i\,x^j)\\&\qquad+O(\lambda^{-5}\,(1+|\xi|)^{-4}).
\end{align*} 
Using this and \eqref{angle}, we obtain
\begin{align*} 
\bar\Delta_{\mathbb{R}^3}\psi&=-12\,\phi^{-4}\,\lambda\,(1-|\xi|^2)\,|x|^{-5}+O(\lambda^{-5}\,(1+|\xi|)^{-4}),\\
\bar D^2_{\bar\nu,\bar\nu}\psi&=-\frac92\,\phi^{-4}\,\lambda^{-3}\,|x|^{-1}-\frac12\,\phi^{-4}\,\lambda^{-1}\,(21-33\,|\xi|^2)\,|x|^{-3}-\frac12\,\phi^{-4}\,\lambda\,(1-|\xi|^2)\,(27-39\,|\xi|^2)\,|x|^{-5}\\&\qquad -\frac{15}{2}\,\phi^{-4}\,\lambda^3\,(1-|\xi|^2)^3\,|x|^{-7}+O(\lambda^{-5}\,(1+|\xi|)^{-4}),\\
\bar D_{\bar\nu}\psi&=3\,\phi^{-4}\,\lambda^{-2}\,|x|^{-1}+6\,\phi^{-4}\,(1-|\xi|^2)\,|x|^{-3}+3\,\phi^{-4}\,\lambda^2\,(1-|\xi|^2)^2\,|x|^{-5}.
\end{align*}
	The assertion follows from this and Lemma \ref{perturbations}.

\end{proof}
The  following corollary  is an immediate consequence of   Lemma \ref{inverse expansions lemma} and Lemma \ref{willmore operator one lemma}.
\begin{coro}
	\label{Willmore operator on sphere spherical harmonics}
	If $|\xi|<1-\delta$, there holds
	$$
	{W}({{S}_{\xi,\lambda}})=4\,\lambda^{-4}\sum_{l=0}^\infty (l-1)\,(l+1)\,(l+2)\,|\xi|^{l}\,P_l(-|\xi|^{-1}\,\bar g(\bar \nu,\xi))+{O}(\lambda ^{-5}).
	$$
	If $|\xi|>1+\delta$, there holds
\begin{equation} \label{W operator outlying}
	\begin{aligned}	{W}({{S}_{\xi,\lambda}})&=-4\,\lambda^{-4}\sum_{l=0}^\infty (l-1)\,l\,(l+2)\,|\xi|^{-l-1}\,P_l(-|\xi|^{-1}\,\bar g(\bar \nu,\xi))\\&\qquad 
	-4 \,\lambda^{-3}\,\ubar{\phi}^{-10}\,(3\,\ubar{\sigma}(\bar\nu,\bar\nu)-\bar {\operatorname{tr}}\,\ubar\sigma)+Y_1+Y_3
	+{O}(\lambda^{-5}\,|\xi|^{-4}).
	\end{aligned}
\end{equation}
	Here, $Y_1$ and $Y_3$ are, respectively, first and third spherical harmonics with
	$$
	Y_1=O(\lambda^{-5}\,|\xi|^{-3})\qquad \text{ and } \qquad Y_3=O(\lambda^{-5}\,|\xi|^{-3}).
	$$
	If $g$ satisfies \eqref{ac sigma} for every multi-index $J$ with $|J|\leq 5$, then (\ref{W operator outlying}) may be differentiated once with respect to $\xi$.
\end{coro}
\begin{rema} \normalfont
Note that
	$$
3\,\ubar{\sigma}(\bar\nu,\bar\nu)-\bar {\operatorname{tr}}\,\ubar\sigma=	\sum_{i,\,j=1}^3\ubar{\sigma}(e_i,e_j)\,(3\,\bar g(\bar\nu,e_i)\,\bar g(\bar\nu, e_j)-\delta_{ij})\in\Lambda_2(S_{\xi,\lambda}).
	$$
\end{rema}
In the next lemma, we specify the formula for the linearization of the Willmore operator (\ref{Q general formula}) to a sphere.
\begin{lem}
	\label{Q lem}
		 For every $u\in C^{\infty}(S_{\xi,\lambda})$ there holds 
	\begin{equation}
	\begin{aligned}
	Q_{S_{\xi,\lambda}}u&=L(Lu)+\frac12\, H^2\, Lu+(\nabla ^2 u)*O(\lambda^{-4}\,(1+|\xi|)^{-2})+(\nabla u)*O(\lambda^{-4}\,(1+|\xi|)^{-3})\\&\qquad+u*O(\lambda^{-5}\,(1+|\xi|)^{-3}).
	\end{aligned}
	\label{Q formula}
	\end{equation}
 If $g$ satisfies \eqref{ac sigma} for every multi-index $J$ with $|J|\leq 5$, (\ref{Q formula}) may be differentiated once with respect to $\xi$.
\end{lem}
\begin{proof}
	This follows from (\ref{Q general formula}), the decay of the metric, Lemma \ref{perturbations}, the estimates $$\nabla H=O(\lambda^{-3}\,(1+|\xi|)^{-2}), \qquad \nabla^2H=O(\lambda^{-4}\,(1+|\xi|)^{-2}),$$
	and the  estimate
	$$
	\Delta H=W+O(\lambda^{-4}\,(1+|\xi|)^{-3})=O(\lambda^{-4}\,(1+|\xi|)^{-3}).
	$$
	We have used Corollary \ref{Willmore operator on sphere spherical harmonics} in the last equation.
\end{proof}
\section{The foliation property}
 Recall from the proof of Theorem \ref{existence thm} that  $\Sigma(\lambda)$ is the sphere $$\Sigma(\lambda)=\Sigma_{\xi(\lambda),\lambda}=\Sigma_{\xi(\lambda),\lambda}(u_{\xi(\lambda),\lambda})$$ where $\xi(\lambda)\in\mathbb{R}^3$ is the unique local minimum near the origin of the function $G_\lambda$ defined in (\ref{G definition}). In particular, by Lemma \ref{variational to 3 dim}, $\Sigma(\lambda)$ is a stable area-constrained Willmore surface. Moreover, we have seen in the proof of Theorem \ref{existence thm} that,  as $\lambda\to\infty$,
 \begin{align}
 \xi(\lambda)=o(1). \label{xi decay}
 \end{align}
 By Proposition \ref{LS proposition} and Remark \ref{u ambient derivatives ren}, we have 
 \begin{align}
 u_{\xi(\lambda),\lambda}=O(1),\qquad 
 (\bar D u)|_{(\xi(\lambda),\lambda)}=O(\lambda^{-1}),  \qquad  u'|_{(\xi(\lambda),\lambda)}=O(\lambda^{-2}),
 \label{implicit f}
 \end{align}
 where we recall that $\bar D$ and the dash indicate differentiation with respect to the parameters $\xi$ and $\lambda$, respectively.\\ 
    \indent  We now verify that the family of spheres $\{\Sigma(\lambda):\lambda>\lambda_0\}$ forms a smooth foliation, provided $\lambda_0>1$ is sufficiently large.   
\begin{prop}\label{foliation prop}
	Suppose that $(M,g)$ is $C^4$-asymptotic to Schwarzschild with mass $m>0$ and that the scalar curvature $R$ satisfies (\ref{sRT}) and (\ref{sRT2}). Then	the family $\{\Sigma(\lambda):\lambda>\lambda_0\}$ of stable area-constrained Willmore spheres is a smooth foliation of the complement of a compact subset of $M$ provided $\lambda_0>1$ is sufficiently large.
\end{prop} 
\begin{proof}
 By Lemma \ref{G der sRT}, there are constants $\tau>0$ and $\delta_0>0$ such that \begin{align}\bar D^2 G_\lambda\geq \tau\,\operatorname{Id}
	\label{strict convexity}\end{align} on $\{\xi\in\mathbb{R}^3:|\xi|<\delta_0\}$ provided $\lambda>\lambda_0$ and $\lambda_0>1$ is sufficiently large. In particular, by the implicit function theorem, the dependence of $\xi(\lambda)$ on $\lambda$ is smooth.  It follows that the map
	  $$\Psi:{S}_1(0)\times(\lambda_0,\infty)\to M \qquad \text{ given by }\qquad
	\Psi(y,\,\lambda)=\Phi_{\xi(\lambda),\lambda}^{u_{\xi(\lambda),\lambda}}({\lambda\,y+\xi})
	$$
	is smooth. Using (\ref{xi decay}) and (\ref{implicit f}), we find that $\Sigma(\lambda)$ encloses every given compact set, provided $\lambda>1$ is sufficiently large. 
		 \\ \indent   We claim that
	\begin{align}
		\xi'(\lambda)=o(\lambda^{-1}).
		\label{toshow}
	\end{align}
	To see this, let $a\in \mathbb{R}^3$. Differentiating the identity $(\bar DG_\lambda)|_{\xi(\lambda)}(a)=0$ with respect to $\lambda$, we obtain
	\begin{align}
	(\bar D^2G_\lambda)|_{\xi(\lambda)}(a,\, \xi'(\lambda))+(\bar{D} G_\lambda')|_{\xi(\lambda)}(a)=0.
	\label{critical point differentiated}
	\end{align}
	The argument presented in the proof of Lemma \ref{G der sRT} also shows that we may differentiate the error terms in Lemma \ref{G expansion} with respect to $\lambda$.  Applying Lemma \ref{G expansion} and using (\ref{xi decay}), we thus find 
	$$(\bar{D} G'_\lambda)|_{\xi(\lambda)}(a)=-4\,\lambda \int_{S_{\xi(\lambda),\lambda}}  \bar g(a,\bar\nu) \,R\,\mathrm{d}\bar\mu -2\,\lambda^2 \int_{S_{\xi(\lambda),\lambda}}  g(a,\bar\nu) \,(\bar D_{\bar \nu} R)\,\mathrm{d}\bar\mu+o(\lambda^{-1}).
	$$
Since $(M,g)$ is $C^4$-asymptotic to Schwarzschild, we obtain from (\ref{sRT2})  that
	\begin{align}
\sum_{i=1}^3	\big[x^i(\partial_i R)(x)+x^i(\partial_i R)(-x)\big]=o(|x|^{-4}).
	\label{dSRT}
	\end{align}
	Indeed, if (\ref{dSRT}) failed, integration along radial lines would yield that (\ref{sRT2}) must be violated, too. From this, we find that
	$$
	(\bar{D} G'_\lambda)|_{\xi(\lambda)}(a)=\xi(\lambda)\,O(\lambda^{-1})+o(\lambda^{-1})=o(\lambda^{-1}).
	$$
	Choosing $
	a=\xi'(\lambda)
	$
	and using (\ref{critical point differentiated}) as well as (\ref{strict convexity}), we obtain the asserted estimate (\ref{toshow}).\\ \indent	Note that $\Psi(\,\cdot\,,\,\lambda)$ parametrizes $\Sigma(\lambda)$ and that $\bar\nu=y+O(\lambda^{-1})$. Using (\ref{implicit f}) and (\ref{toshow}), we compute that
	\begin{align*}
	\bar g(\Psi',y)=&\,1+\bar g(\xi(\lambda),y)+\lambda\,\bar g( \xi'(\lambda),y)+(\bar D_{\xi'(\lambda)}u)|_{({\xi(\lambda),\lambda})}+ u'|_{(\xi(\lambda),\lambda)}\\=&\,1+o(1).
	\end{align*}
	In particular, $\bar g(\Psi',\bar \nu)>0$. This finishes the proof.
\end{proof}

\section{Remark on far-outlying stable constant mean curvature surfaces}
\label{CMC appendix}
In this section, we show that the assumptions of Theorem \ref{far outlying thm} are sufficient to preclude large far-outlying stable constant mean curvature spheres in $(M,g)$ as well. In the statement of the following result, $\operatorname{vol}(\Sigma)$ denotes the volume of the compact domain bounded by  $\Sigma$. 
\begin{thm}
	Suppose that $(M,g)$ is $C^5$-asymptotic to Schwarzschild with mass $m>0$ and that its scalar curvature $R$ satisfies
	\begin{align}
\sum_{i=1}^3	x^i\,\partial_i(|x|^2R)\leq 0.
	\label{CMC R Growth}
	\end{align}
	There is no sequence $\{\Sigma_j\}^\infty_{j=1}$ of outlying stable constant mean curvature spheres $\Sigma_j\subset M$ with
	\label{CMC outlying}
	$$
	\lim_{j\to\infty}\operatorname{vol}(\Sigma_j)=\infty\qquad \text{ and } \qquad \lim_{j\to\infty}\rho({\Sigma_j})\,H(\Sigma_j)=\infty.
	$$
\end{thm} 
\begin{rema} \normalfont
	The hypotheses of Theorem \ref{CMC outlying} are weaker than those of Corollary 1.7 in \cite{chodosh2019far}. First, we only require $C^5$-decay of the metric, while $C^7$-decay of the metric is assumed in \cite{chodosh2019far}. Second, the growth condition (\ref{CMC R Growth}) is weaker than the radial convexity  assumption
	$$
\sum_{i,\,j=1}^3x^i\,x^j\,\partial_i\partial_j R\geq 0
	$$ in \cite{chodosh2019far}.
\end{rema}	
The proof of Theorem \ref{CMC outlying} is a small variation of the proof of Corollary 1.7 in \cite{chodosh2019far}. The asserted improvement is obtained by first taking the radial derivative of the area functional of a coordinate sphere and then estimating the resulting terms, rather than first estimating the area functional and then taking the radial derivative. This approach brings out the contribution of the scalar curvature in a more precise way. We only point out the necessary modifications of the proof.
\\ \indent  We recall that $$\phi(x)=1+|x|^{-1}$$ denotes the conformal factor of the Schwarzschild metric with mass $m=2$. Moreover, continuing the notation introduced on p.~\pageref{underline page}, we use a bar underneath a quantity to indicate evaluation at $\lambda\,\xi$.
\begin{proof}[Proof of Theorem \ref{CMC outlying}]
	As in the proof of Theorem \ref{far outlying thm}, there exists a constant $\lambda_0>1$ which only depends on $(M,g)$ such that for every $\lambda>\lambda_0$ and $\xi\in\mathbb{R}^3 $ with $|\xi|>2$, there is a surface $\Sigma_{\xi,\lambda}$ with the following properties:
	\begin{itemize}
		\item[$\circ$]  $\Sigma_{\xi,\lambda}$ is a perturbation of the sphere $S_{\tilde \xi,\tilde \lambda}$, where $$\tilde\lambda\,\tilde \xi=\lambda\,\xi.$$ Moreover, we have
		\begin{align}
		\tilde \lambda =\lambda\, \ubar{ \phi}^{-2}+O(\lambda^{-1}\,|\xi|^{-2}). \label{tilde lambda relation}
		\end{align}
		\item[$\circ$] The mean curvature of $\Sigma_{\xi,\lambda}$ is constant up to first spherical harmonics.
		\item[$\circ$] There holds
		$$
		\operatorname{vol}(\Sigma_{\xi,\lambda})=\frac{4\,\pi}{3}\,\lambda^3.
		$$
		\item[$\circ$]  $\Sigma_{\xi,\lambda}$ has constant mean curvature if and only if $\xi$ is a critical point of the function
		$$
		A_\lambda:\{\xi\in\mathbb{R}^3:|\xi|>2\}\to\mathbb{R}^3\qquad \text{ given by }\qquad  A_\lambda(\xi)=|\Sigma_{\xi,\lambda}|.
		$$ 
	\end{itemize}
We note that $\tilde \lambda$ is denoted by $r$ in \cite{chodosh2019far}. \\ \indent 
	In \cite[p.~182]{chodosh2019far}, it was shown that
	\begin{align}
	A_\lambda(\xi)=4\,\pi\, \lambda^2-\frac{2\,\pi}{15}\,\lambda^4\,\ubar{R}-\frac{\pi}{105}\,\lambda^6\, \bar\Delta\ubar{R}-\frac{8\,\pi}{35}\,|\xi|^{-6}+O(\lambda^{-1}\,|\xi|^{-6})+O(|\xi|^{-7}).
	\label{CE 1}
	\end{align}
	This identity may be differentiated once with respect to $\xi$. 
	In the derivation of this identity and its differentiability in \cite[\S 4]{chodosh2019far}, $C^6$-decay of the metric rather than $C^5$-decay is used to analyze the contribution of the term 
		\begin{align}
	\frac12 \int_{S_{\tilde \xi,\tilde \lambda}}[\bar {\operatorname{tr}}\,\sigma-\sigma(\bar\nu,\bar\nu)]\,\mathrm{d}\bar \mu
	-\tilde\lambda^{-1}\int_{B_{\tilde\lambda}(\lambda \,\xi)} \bar {\operatorname{tr}}\,\sigma \,\mathrm{d} \bar{v}
	\label{CE 2}
	\end{align}
	to (\ref{CE 1}). In \cite[\S 4.1 and \S 4.2]{chodosh2019far}, (\ref{CE 2}) was computed to be 
	$$
	-\frac{2\,\pi}{15}\,  \lambda^4\, \ubar {R}
	-
	\frac{\pi}{105}\,\lambda^6\,  \bar\Delta \ubar{ R}
	-\frac{8\,\pi}{15}\, \lambda|\xi|^{-3}\,\left[\bar {\operatorname{tr}}\,\ubar{\sigma}-3|\xi|^{-2}\sigma(\xi,\xi)+\lambda\,\bar D_{\xi} \bar {\operatorname{tr}}\,\ubar{\sigma}\right]+O(\lambda^{-1}\,|\xi|^{-6})+O(|\xi|^{-7}).
	$$
	It follows that
	\begin{align*}
	A_\lambda(\xi)&=4\,\pi\, \lambda^2-\frac{8\,\pi}{35}\,|\xi|^{-6}+\frac{8\,\pi}{15}\, \lambda\,|\xi|^{-3}\,(\bar {\operatorname{tr}}\,\ubar{\sigma}-3|\xi|^{-2}\,\sigma(\xi,\xi)+\lambda\,\bar D_{\xi} \bar {\operatorname{tr}}\,\ubar{\sigma})
	\\&\qquad +\frac12 \int_{S_{\tilde \xi,\tilde \lambda}}[\bar {\operatorname{tr}}\,\sigma-\sigma(\bar\nu,\bar\nu)]\,\mathrm{d}\bar \mu
	-\tilde\lambda^{-1}\int_{B_{\tilde\lambda}(\lambda \,\xi)} \bar {\operatorname{tr}}\,\sigma \,\mathrm{d} \bar{v}+O(\lambda^{-1}\,|\xi|^{-6})+O(|\xi|^{-7}).
	\end{align*}
	This expansion may be differentiated once with respect to $\xi$ provided $(M,g)$ is $C^5$-asymptotic to Schwarzschild. We proceed by computing the radial derivative of $A_\lambda$. Using Taylor's theorem and cancellations due to symmetry, we find
	\begin{equation} \label{rd 1}
	\begin{aligned}
	&\sum_{i=1}^3\xi^i\,\partial_i\bigg(\frac{8\,\pi}{15}\, \lambda|\xi|^{-3}\,(\bar {\operatorname{tr}}\,\ubar{\sigma}-3\,|\xi|^{-2}\,\sigma(\xi,\xi)+\lambda\,\bar D_{\xi} \bar {\operatorname{tr}}\,\ubar{\sigma})\bigg)
	\\ &\qquad =\,-2\int_{B_{\lambda}(\lambda\,\xi)} \left[|x|^{-3}\bar {\operatorname{tr}}\,{\sigma}-3\,|x|^{-5}\,\sigma(x,x)+|x|^{-3}\,\bar D_{x}\bar {\operatorname{tr}}\,\sigma\right]\,\bar g(\lambda\,\xi-x,\xi)\,\mathrm{d}\bar{v}\\&\qquad\qquad +O(\lambda^{-1}\,|\xi|^{-6}).
	\end{aligned}
	\end{equation}
	Next, we compute
	\begin{equation}
	\label{rd 2}
	\begin{aligned}
	&\,\sum_{i=1}^3\xi^i\,\partial_i\bigg(\frac12 \int_{S_{\tilde \xi,\tilde \lambda}}[\bar {\operatorname{tr}}\,\sigma-\sigma(\bar\nu,\bar\nu)]\,\mathrm{d}\bar \mu
	-\tilde\lambda^{-1}\int_{B_{\tilde\lambda}(\lambda\, \xi)} \bar {\operatorname{tr}}\,\sigma \,\mathrm{d} \bar{v}\bigg)
	\\&\qquad =\,\frac12 \,\lambda\int_{S_{\tilde \xi,\tilde \lambda}}\left[\bar D_{\xi}\bar {\operatorname{tr}}\,\sigma-\bar D_{\xi}\sigma(\bar\nu,\bar\nu)
	-2\,\tilde\lambda^{-1}\, \bar {\operatorname{tr}}\,\sigma\,\bar g(\xi,\bar\nu)\right] \,\mathrm{d}\bar \mu
	\\&\qquad\qquad \,+\frac12\,\xi^i\,\partial_i\tilde\lambda\,\bigg(\int_{S_{\tilde \xi,\tilde \lambda}}\left[\bar D_{\bar\nu}\bar {\operatorname{tr}}\,\sigma-\bar D_{\bar\nu}\sigma(\bar\nu,\bar\nu)
	-2\,\tilde \lambda^{-1}\,\sigma(\bar\nu,\bar\nu)\right] \mathrm{d}\bar\mu+2\,\tilde\lambda^{-2}\int_{B_{\tilde\lambda}(\lambda\, \xi)}\bar {\operatorname{tr}}\,\sigma\,\mathrm{d}\bar{v}\bigg).
	\end{aligned}
	\end{equation}
	From (\ref{tilde lambda relation}), we find that 
	$$
\sum_{i=1}^3	\xi^i\,\partial_i \tilde \lambda=2\,|\xi|^{-1}+O(\lambda^{-1}\,|\xi|^{-2}).
	$$
	Using cancellations due to symmetry, we compute, using Taylor's theorem to expand all terms up to second derivatives of $\sigma$ and Lemma \ref{spherical identities}, that the last line of  (\ref{rd 2}) equals
	\begin{equation}
	\label{rd 3}
	\begin{aligned}
	&-\frac{16\,\pi}{15}\,\lambda^3\,|\xi|^{-1}\,(\operatorname{div}\operatorname{div}\ubar{\sigma}- \bar\Delta\bar {\operatorname{tr}}\,\ubar{\sigma})+O(\lambda^{-1}\,|\xi|^{-6})\\&\qquad =\,4 \int_{B_{\lambda}(\lambda\,\xi)}(\operatorname{div}\operatorname{div}\ubar{\sigma}- \bar\Delta\bar {\operatorname{tr}}\,\ubar{\sigma})\,(|\xi|^{-1}-|x|^{-1})\,\bar g(\xi,\lambda\,\xi-x)\,\mathrm{d}\bar{v}+O(\lambda^{-1}\,|\xi|^{-6}).
	\end{aligned}
	\end{equation}
	Here, we have also used that
	$$
	|\xi|^{-1}=|x|^{-1}-\lambda^{-2}\,|\xi|^{-3}\,\bar g(\xi,\lambda\,\xi-x)+O(\lambda^{-1}|\xi|^{-3}).
	$$
	Finally, using $\lambda\,\xi=\tilde\lambda\,\tilde\xi$, we can argue exactly as in \cite[\S 2.1]{chodosh2019far} to show that
	the second line of (\ref{rd 2}) equals
	\begin{equation} \label{rd 4}
	\begin{aligned}
	&\frac12  \int_{B_{\tilde \lambda}(\lambda\,\xi)}(\operatorname{div}\operatorname{div}{\sigma}- \bar\Delta\bar {\operatorname{tr}}\,{\sigma})\,\bar g(\tilde \xi,\lambda\,\xi-x) \,\mathrm{d}\bar\mu\\
	&\qquad=\,-\frac{2\,\pi}{15}\,\tilde \lambda^5\, \bar D_{\tilde \xi}(\operatorname{div}\operatorname{div}\ubar{\sigma}-\bar\Delta\bar {\operatorname{tr}}\,\ubar{\sigma})+O(\lambda^{-1}\,|\xi|^{-6})
	\\&\qquad=\,-\frac{2\,\pi}{15}\,\ubar{\phi}^{-8}\, \lambda^5\, \bar D_{ \xi}(\operatorname{div}\operatorname{div}\ubar{\sigma}-\bar\Delta\bar {\operatorname{tr}}\,\ubar{\sigma})+O(\lambda^{-1}\,|\xi|^{-6})
	\\
	&\qquad=\,\frac12\,\, \ubar{\phi}^{-8} \int_{B_{\lambda}(\lambda\,\xi)}(\operatorname{div}\operatorname{div}{\sigma}- \bar\Delta\bar {\operatorname{tr}}\,{\sigma})\,\bar g(\xi,\lambda\,\xi-x) \,\mathrm{d}\bar\mu+O(\lambda^{-1}\,|\xi|^{-6}).
	\end{aligned}
	\end{equation}
	In the first and third equality, we have used Taylor's theorem to expand the integrand up to fourth derivatives of $\sigma$, the $C^5$-decay of the metric, cancellations due to symmetry, and Lemma \ref{spherical identities}. In the second equality, we have used (\ref{tilde lambda relation}). According to \cite[\S 4.9]{chodosh2019far}, there holds 
	$$
	R=\phi^{-8}\,(\operatorname{div}\operatorname{div}\ubar{\sigma}- \bar\Delta\bar {\operatorname{tr}}\,\ubar{\sigma})-4\,\left[|x|^{-3}\,\bar {\operatorname{tr}}\,{\sigma}-3\,|x|^{-5}\,\sigma(x,x)+|x|^{-3}\,\bar D_{x}\bar {\operatorname{tr}}\,\sigma\right]+O(\lambda^{-1}\,|\xi|^{-6})
	$$
	while 
	$$
	\ubar{\phi}^{-8}=\phi^8+8\,(|x|^{-1}-|\xi|^{-1})+O(\lambda^{-1}\,|\xi|^{-2}).
	$$
	Combing this with (\ref{rd 1}), (\ref{rd 2}), (\ref{rd 3}), and (\ref{rd 4}), we conclude that
	$$
	\sum_{i=1}^3\xi^i\,(\partial_i A_\lambda)(\xi)=\frac{48\,\pi}{35}\,|\xi|^{-6}+\frac12 \int_{B_\lambda(\lambda\,\xi)} \bar g(\xi,\lambda\,\xi-x)\,R\,\mathrm{d}\bar\mu+O(\lambda^{-1}\,|\xi|^{-6})+O(|\xi|^{-7}).
	$$
	In \cite[\S 2.2]{chodosh2019far}, it has been shown that this integral is non-negative provided that (\ref{CMC R Growth}) holds.
	In particular, $$\sum_{i=1}^3\xi^i\,(\partial_i A_\lambda)(\xi)>0$$  provided both $\xi\in\mathbb{R}^3$ and $\lambda>1$ are large. We may now conclude the proof as in \cite{chodosh2019far}.

\end{proof}\end{appendices}

\end{document}